%% file: YB_bij_arxiv_V3.tex
\documentclass[letterpaper,11pt,oneside,reqno]{amsart}

\usepackage[sorting=nyt,style=alphabetic,backend=bibtex,doi=false,maxbibnames=9,maxcitenames=4]{biblatex}
\makeatletter
\def\blx@maxline{77}
\makeatother
\usepackage{amsmath,amssymb,amsthm,amsfonts}
\usepackage[colorlinks=true,linkcolor=blue,citecolor=red]{hyperref}
\usepackage{graphicx,color,colortbl}
\usepackage{upgreek}
\usepackage[mathscr]{euscript}

\allowdisplaybreaks
\numberwithin{equation}{section}

\usepackage{tikz}
\usetikzlibrary{
	shapes,
	arrows,
	positioning,
	decorations.markings,
	decorations.pathmorphing
}

\usepackage{array}
\usepackage{adjustbox}
\usepackage{cleveref}
\usepackage{enumerate}

\usepackage[DIV=12]{typearea}

\newtheorem{proposition}{Proposition}[section]
\newtheorem{lemma}[proposition]{Lemma}
\newtheorem{corollary}[proposition]{Corollary}
\newtheorem{theorem}[proposition]{Theorem}
\theoremstyle{definition}
\newtheorem{definition}[proposition]{Definition}
\newtheorem{remark}[proposition]{Remark}
\input{single_vertex_text.tex}

\input{YB_vertices.tex}

\input{six_vertex.tex}

\input{two_vertices_text.tex}

\input{six_vertex_two.tex}
\input{single_vertex_full.tex}

\input{table_colors.tex}

\begin{document}
\title[Yang-Baxter field for spin Hall-Littlewood symmetric
	functions]{Yang-Baxter field for spin Hall-Littlewood\\symmetric functions}

\author[A. Bufetov]{Alexey Bufetov}\address{A. Bufetov, Massachusetts
	Institute of Technology, Department of Mathematics, 77 Massachusetts Avenue,
	Cambridge, MA 02139, USA}\email{alexey.bufetov@gmail.com}

\author[L. Petrov]{Leonid Petrov}\address{L. Petrov, University of Virginia,
	Department of Mathematics, 141 Cabell Drive, Kerchof Hall, P.O. Box 400137,
	Charlottesville, VA 22904, USA, and Institute for Information Transmission
	Problems, Bolshoy Karetny per. 19, Moscow, 127994,
	Russia}\email{lenia.petrov@gmail.com}

\date{}

\begin{abstract}
	Employing bijectivisation of summation identities, we introduce local
	stochastic moves based on the Yang-Baxter equation for
	$U_q(\widehat{\mathfrak{sl}_2})$. Combining these moves leads to a new object
	which we call the spin Hall-Littlewood Yang-Baxter field --- a probability
	distribution on two-dimensional arrays of particle configurations on the
	discrete line. We identify joint distributions along down-right paths in the
	Yang-Baxter field with spin Hall-Littlewood processes, a generalization of
	Schur processes. We consider various degenerations of the Yang-Baxter field
	leading to new dynamic versions of the stochastic six vertex model and of the
	Asymmetric Simple Exclusion Process.
\end{abstract}

\maketitle

\tableofcontents

\section{Introduction}

\subsection{Overview}

The past two decades have seen a wave of progress in understanding large
scale, long time asymptotics of driven nonequilibrium stochastic particle
systems in the one space and one time dimension belonging to the
Kardar-Parisi-Zhang (KPZ) universality class (about the KPZ class see, e.g.,
\cite{CorwinKPZ}, \cite{Corwin2016Notices}, \cite{halpin2015kpzCocktail}).
Much of this progress has been achieved by discovering exact distributional
formulas in these particle systems, and leveraging these formulas towards
asymptotic analysis. Stochastic particle systems possessing such exact
formulas are known under the name \emph{integrable}. Since the early days
(e.g., \cite{johansson2000shape}), success in discovering integrability (at
least for special initial data) has often been triggered by applications of
techniques coming from the algebra of symmetric functions \cite[Ch.
	I]{Macdonald1995}. Among the most notable frameworks for these applications
are Schur processes \cite{okounkov2001infinite},
\cite{okounkov2003correlation}, \cite{Borodin2010Schur}, \cite{Betea_etal2014}
and Macdonald processes \cite{BorodinCorwin2011Macdonald}, \cite{BCGS2013}.
The success of this approach naturally leads to a more extensive study of
structural properties of various families of symmetric functions and their
relations to probabilistic systems.

In this work we investigate stochastic systems related to \emph{spin
	Hall-Littlewood symmetric rational functions}
introduced in 
\cite{Borodin2014vertex}.
These functions are naturally at the interplay of the theory of symmetric
functions and the Yang-Baxter equation (see, e.g.,
\cite{tsilevich2006quantum}, \cite{betea2016refined}, \cite{betea2015refined},
\cite{wheeler2015refined} for other related examples). The main results of the
present paper are:
\begin{itemize}
	\item
	      We consider the general idea of bijectivisation
	      (\Cref{sec:bijectivisation}) and apply it to the Yang-Baxter equation
	      obtaining local stochastic moves acting on vertex model configurations
	      (\Cref{sec:main_construction}). We hope that the usefulness of this general
	      idea will not be limited by the results of this paper.
	\item
	      We introduce the spin Hall-Littlewood Yang-Baxter field
	      (\Cref{sec:YB_field}), a two-dimensional array of random particle
	      configurations on the discrete line. Its main properties are explicit formulas
	      for distributions along any down-right path
	      (\Cref{thm:YB_field_spin_HL_process}), and Markov projections turning the
	      Yang-Baxter field into a two-dimensional scalar field or its multilayer
	      versions (\Cref{prop:YB_Markov_projections,prop:dyn6V_is_YB_field}).
	\item
	      We consider a number of degenerations of the Yang-Baxter field,
	      including new dynamic versions of the stochastic six vertex model
	      (\Cref{sec:dynamicS6V}) and the Asymmetric Simple Exclusion Process
	      (\Cref{sub:degen_ASEP}). Our results about these dynamic models generalize
	      those of the recent works \cite{BorodinBufetovWheeler2016} and
	      \cite{BufetovMatveev2017}.
\end{itemize}

Let us describe our results in more detail.

\subsection{Random fields of Young diagrams}
\label{sub:YB_field_intro_section}

One of the key properties behind probabilistic applications of Macdonald (in
particular, Schur) symmetric functions is that they satisfy \emph{Cauchy
	summation identities} \cite[Ch. I.4 and Ch. VI.2]{Macdonald1995}
(see also \Cref{sub:Cauchy_identity} below for Cauchy identities
for the spin Hall-Littlewood symmetric functions). Regarding these identities
as expressing probability normalizing constants (=~partition functions) allows
to define and analyze \emph{Macdonald processes}. These are certain
probability distributions on collections of Young diagrams\footnote{In
	probabilistic applications, Young diagrams are often interpreted as particle
	configurations on the discrete line.} whose probability weights are
proportional to products of the (skew) Macdonald symmetric polynomials. A lot
of recent research is devoted to the study of these processes and their
degenerations, with applications to KPZ type and other asymptotics, e.g., see
\cite{Oconnell2009_Toda}, \cite{COSZ2011}, \cite{OSZ2012},
\cite{BorodinCorwin2011Macdonald}, \cite{BorodinCorwinFerrariVeto2013},
\cite{BorodinGorin2013beta}, \cite{BorodinPetrov2013Lect}.

It is much less articulated in the existing literature that one can consider
Macdonald (Schur, etc.) \emph{fields} --- certain ways to couple many
processes together leading to two-dimensional arrays of random Young diagrams.
Such fields are highly non-unique, and coming up with a ``good'' way to couple
processes together involves additional considerations like the presence of
Markov projections (see below). Various elements of Young diagram random
fields have appeared in the literature mainly as ways to match observables of
$(1+1)$-dimensional stochastic interacting particle systems with observables
of Macdonald or Schur processes. The latter observables then can be analyzed
to the point of asymptotics thanks to the algebraic structure coming from
symmetric functions. Two ways to construct such random fields were mainly
employed which we briefly discuss in \Cref{sub:intro_RSK,sub:intro_BF} below.

\subsection{RSK type fields}
\label{sub:intro_RSK}

RSK type fields were applied in probabilistic context in connection with Schur
measures as early as in \cite{baik1999distribution},
\cite{johansson2000shape}, \cite{PhahoferSpohn2002} to study asymptotics of
longest increasing subsequences, last passage percolation, TASEP (Totally
Asymmetric Simple Exclusion Process), and PNG (polynuclear growth). These
fields arise (in the Schur case) as results of applying the
Robinson-Schensted-Knuth (RSK) insertion algorithm to a random input, hence
the name. More precisely, a Schur RSK type field can be realized using Fomin
growth diagrams (an equivalent way to interpret the RSK insertion
\cite{Fomin1986}, \cite{fomin1995schensted}) with random integer inputs.
The idea to apply RSK insertion to random input seems to have first appeared in
\cite{Vershik1986}, and was substantially developed in
\cite{Baryshnikov_GUE2001},
\cite{OConnell2003}, \cite{OConnell2003Trans}.

Recently RSK type fields associated with deformations of Schur processes (see
\Cref{fig:polys}) were constructed for 
Whittaker processes
\cite{Oconnell2009_Toda},
\cite{COSZ2011},
\cite{OSZ2012},
$q$-Whittaker processes
\cite{OConnellPei2012}, \cite{BorodinPetrov2013NN}, \cite{Pei2013Symmetry},
\cite{MatveevPetrov2014}, \cite{pei2016qRSK}, and Hall-Littlewood
processes \cite{BufetovPetrov2014}, \cite{BorodinBufetovWheeler2016},
\cite{BufetovMatveev2017}. 
Constructions at the Whittaker level relied 
on the geometric (also sometimes called ``tropical'') 
lifting of the RSK correspondence
\cite{Kirillov2000_Tropical}, \cite{NoumiYamada2004},
while the $q$-Whittaker and Hall-Littlewood
developments required nontrivial
randomizations of the original RSK insertion algorithm.

Via Markov projections, 
this work uncovered connections 
of Whittaker, $q$-Whittaker, and Hall-Littlewood processes
with known and
new $(1+1)$-dimensional stochastic particle systems. In the
Whittaker case, these are various integrable models of directed random polymers
\cite{OConnellYor2002}, \cite{Seppalainen2012}.
For the $q$-Whittaker processes, these are the $q$-TASEP and related systems
\cite{BorodinCorwin2013discrete}, \cite{CorwinPetrov2013},
\cite{MatveevPetrov2014}. In the Hall-Littlewood case these are the ASEP
\cite{macdonald1968bioASEP}, \cite{Spitzer1970} and the stochastic six vertex
model \cite{GwaSpohn1992}, \cite{BCG6V}.

\begin{figure}[htpb]
	%\begin{noindent}
	\centering
	\scalebox{1}{
		\begin{tikzpicture}
		[scale=1, very thick]
		\node[fill=white,draw,rectangle] (schur) at (0,0) {Schur};
		\node[fill=white,draw,rectangle] (whit) at (4,0) {Whittaker};
		\node[fill=white,draw,rectangle] (HL) at (-4,2) {Hall-Littlewood ($t$)};
		\node[fill=white,draw,rectangle] (sHL) at (-5.6,4.2) {spin Hall-Littlewood ($t,s$)};
		\node[fill=white,draw,rectangle] (jack) at (-.6,2) {Jack};
		\node[fill=white,draw,rectangle] (M) at (0,4.2) {Macdonald ($q,t$)};
		%\node[fill=white,draw,rectangle] (sqW) at (5,4.2) {spin $q$-Whittaker ($q,s$)};
		\node[fill=white,draw,rectangle] (qW) at (4,2) {$q$-Whittaker ($q$)};
		\draw[->] (M)--(jack) node [midway, xshift=23, yshift=-10] 
		{\parbox{.085\textwidth}{$t=q^{\beta/2}$\\\phantom{.}\hfill$\to1$}};
		\draw[->] (jack)--(schur) node [midway, xshift=-18, yshift=2] {$\beta=2$};
		\draw[->] (sHL)--(HL)  node [midway, xshift=-20] {$s=0$};
		\draw[->] (HL)--(schur) node [midway, xshift=-25] {$t=0$};
		\draw[->] (M)--(HL) node [midway, yshift=10, xshift=-5] {$q=0$};
		\draw[->] (M)--(qW) node [midway, xshift=25] {$t=0$};
		%\draw[->] (sqW)--(qW) node [midway, xshift=20] {$s=0$};
		\draw[->] (qW)--(schur) node [midway, xshift=25] {$q=0$};
		\draw[->] (qW)--(whit) node [midway, xshift=20] {$q\nearrow1$};
\end{tikzpicture}}
	\caption{A part of the hierarchy of symmetric functions satisfying summation 
	identities of Cauchy type. Arrows mean degenerations.}
	\label{fig:polys}
	%\end{noindent}
\end{figure}
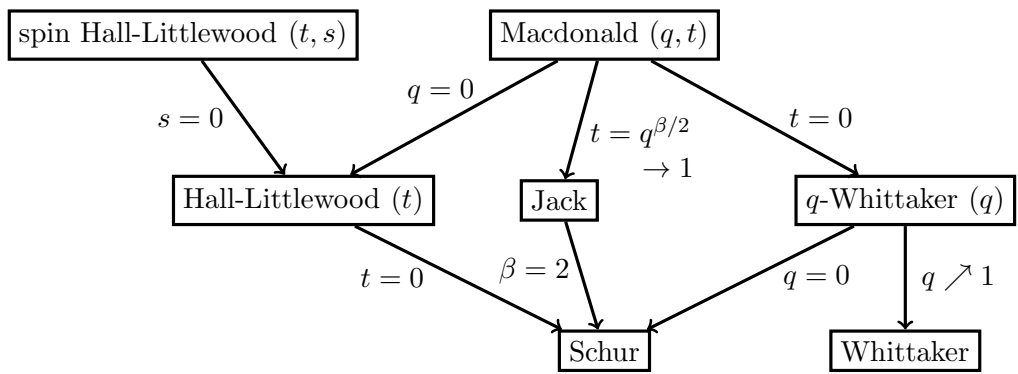

\subsection{BF type fields}
\label{sub:intro_BF}

Another method of constructing random fields of Young diagrams is based on
interpreting the skew Cauchy identity as an intertwining relation between
certain Markov transition matrices, and stitching these matrices together into
a multivariable Markov chain using an idea of Diaconis and Fill
\cite{DiaconisFill1990}. In symmetric functions context this method was first
applied (in the Schur case) in a work by Borodin and Ferrari
\cite{BorFerr2008DF}, hence the name.

In principle, this approach is applicable to a wider variety of models than
the RSK one, and does not require intricate combinatorial constructions. This
generality comes at a cost of having fewer Markovian projections than the RSK
constructions, especially away from the Schur case. An exception in the
literature is that the half-continuous BF type field in the setting of
$q$-Whittaker processes has led to the discovery of the continuous time
$q$-TASEP, a notable deformation of the TASEP with a richer algebraic
structure \cite{BorodinCorwin2011Macdonald}.

A unified approach to both the RSK type and the BF type fields in the
half-continuous setting (details on half-continuous degenerations of random
fields may be found in
\Cref{sub:degen_half_continuous_DS6V,sub:hc_degen_Schur}) was suggested in
\cite{BorodinPetrov2013NN}. In fully discrete setting, elements of BF type
fields for Schur polynomials appeared in \cite{warrenwindridge2009some},
\cite{BorFerr2008DF}.

\subsection{Yang-Baxter field}
\label{sub:intro_intro_YB_fields}

We present a third way of constructing random fields associated with symmetric
functions and the corresponding processes. Our approach is based on the
Yang-Baxter equation which is behind many families of symmetric functions
including Schur, Hall-Littlewood, and spin Hall-Littlewood ones. We focus on
the latter family for which Cauchy summation identities were recently
established in \cite{Borodin2014vertex} with the help of the Yang-Baxter
equation for the quantum $\mathfrak{sl}_2$ \cite{baxter2007exactly}.

In the setting of spin Hall-Littlewood processes, random fields have not been
considered in the literature yet. The Yang-Baxter field we construct in the
present paper yields a new object even in the most basic Schur case
(\Cref{sub:hc_degen_Schur}). The main advantages of our approach are its
simplicity and clear structure of Markov projections yielding new
$(1+1)$-dimensional stochastic systems (see \Cref{sub:degen_intro_section}
below). In comparison, an RSK type approach would likely require very
nontrivial combinatorial considerations (cf. \cite{BufetovMatveev2017} for the
Hall-Littlewood case), further complicated by the fact that the spin
Hall-Littlewood functions are not homogeneous polynomials while the usual
Hall-Littlewood ones are (see \Cref{rmk:spinHL_not_RSK} for more details). A
BF type approach, while clearly being applicable in the spin Hall-Littlewood
case, might not readily produce Markov projections.

Our construction of the Yang-Baxter field uses a very basic idea of
bijectivisation of the Yang-Baxter equation. We briefly describe this idea
next.

\subsection{Bijectivisation of the Yang-Baxter equation}
\label{sub:bijectivisation_intro_section}

In probability theory it is well known that considering couplings of
probability measures is a powerful idea. For our construction of Yang-Baxter
field we apply a similar idea to summation identities which form the
Yang-Baxter equation for quantum $\mathfrak{sl}_2$. We refer to it as a
\emph{bijectivisation} of these combinatorial summation identities. As a
byproduct of couplings thus constructed we obtain conditional distributions,
and we regard them as local stochastic (Markov) moves acting on vertex model
configurations. The bijectivisation of the Yang-Baxter equation we consider is
also not unique, but the space of possible parameters is quite small. We use
this freedom to choose a bijectivisation with the least ``noise'', in the
spirit of RSK type approach, cf. \cite[Section 7.4]{BorodinPetrov2013NN}. See
\Cref{sub:discussion} for details.

We believe that one of important novelties of this paper is the application of
this idea of coupling to combinatorial summation identities. Here we use it in
only one situation, in the setting of the Yang-Baxter equation powering the
spin Hall-Littlewood symmetric functions. However, it seems likely that this
idea might lead to new interesting constructions and results for other forms
of Yang-Baxter equation as well.

\subsection{Dynamic stochastic six vertex model and dynamic ASEP}
\label{sub:degen_intro_section}

A certain Markov projection of our Yang-Baxter random field yields a
scalar-valued random field indexed by the nonnegative integer quadrant. This
scalar field can be interpreted as a random field of values of the height
function in a certain generalization of the stochastic six vertex model in
which the vertex probabilities additionally depend on the value of the height
function. For this reason one can call this model a \emph{dynamic stochastic
	six vertex model} (DS6V). Its detailed description is given in
\Cref{sub:dynamicS6V_subsection}.

The joint distribution of the values of the height function in DS6V along
\emph{down-right paths}\footnote{Also referred to as \emph{space-like paths}
	in the language of stochastic particle systems, cf.
	\cite{derrida1991dynamics}, \cite{Ferrari_Airy_Survey}, \cite{BorFerr08push}.}
can be identified with that of certain observables of a spin Hall-Littlewood
process (\Cref{cor:dyn6V_spin_HL_process}). In the degeneration turning the
spin Hall-Littlewood symmetric functions into the Hall-Littlewood ones, the
DS6V model becomes the usual stochastic six vertex model of
\cite{GwaSpohn1992}, \cite{BCG6V}, and \Cref{cor:dyn6V_spin_HL_process} turns
into the statement established in \cite{BufetovMatveev2017}.

\medskip

Along with single-layer projections leading to DS6V, one can consider
multilayer projections of the full Yang-Baxter field, as was done in
\cite[Sections 4.4 and 4.5]{BufetovMatveev2017} for the Hall-Littlewood RSK
field. In particular, one can check that the two-layer projection of our
Yang-Baxter field, in the Hall-Littlewood degeneration, coincides with the
two-layer stochastic six vertex model of \cite[Section
	4.4]{BufetovMatveev2017}. However, the corresponding degeneration of the full
Yang-Baxter field is different from the full Hall-Littlewood RSK field.
Details may be found in \Cref{sec:degenerations}.

\medskip

In a continuous time limit around the diagonal, the DS6V model turns into the
following dynamic version of the ASEP depending on parameters $t\ge 0$,
$-1<s\le 0$, and $u>0$. Consider a continuous time particle system
$\{y_\ell(\tau)\}_{\ell\in \mathbb{Z}_{\ge1},\;\tau\in
	\mathbb{R}_{\ge0}}$ on $\mathbb{Z}$ (no more than one particle at a
site), started from the step initial configuration $y_\ell(0)=-\ell$. In
continuous time, each particle $y_{\ell}$, $\ell\ge1$, tries to jump to the
right by one at rate\footnote{That is, the waiting time till the jump is an
	independent exponential random variable with mean equal to
	$(\textnormal{rate})^{-1}$.} $\dfrac{u-st^{\ell}}{u-st^{\ell-1}}$, and to the
left by one at rate $t\,\dfrac{u-st^{\ell-1}}{u-st^{\ell}}$. If the
destination is occupied, the corresponding jump is blocked and $y_{\ell}$ does
not move. See \Cref{fig:dyn_ASEP}. The height function in this dynamic ASEP
can be identified in distribution with a certain limit of observables of spin
Hall-Littlewood processes. When $s=0$, the dynamic dependence of jump rates on
the height function disappears, and the system turns into the usual ASEP. See
\cite{BufetovMatveev2017} for connections of ASEP to Hall-Littlewood
processes.

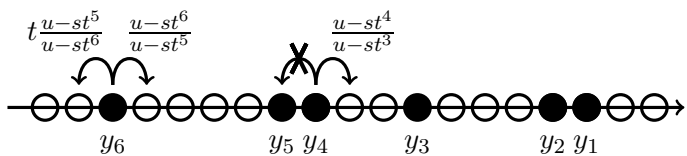
\begin{figure}[htpb]
	%\begin{noindent}
	\centering
	\begin{tikzpicture}
		[scale=1,very thick]
			\def\pt{.17}
			\def\ee{.1}
			\def\h{.45}
			\draw[->] (-.5,0) -- (8.5,0);
			\foreach \ii in {(0,0),(\h,0),(3*\h,0),(4*\h,0),(5*\h,0),(6*\h,0),
			(8*\h,0),(10*\h,0),(9*\h,0),(12*\h,0),(13*\h,0),(14*\h,0),(15*\h,0),(16*\h,0),(17*\h,0),(18*\h,0)}
			{
				\draw \ii circle(\pt);
			}
			\foreach \ii in {(2*\h,0),(7*\h,0),(11*\h,0),(15*\h,0),(8*\h,0),(16*\h,0)}
			{
				\draw[fill] \ii circle(\pt);
			}
			\node at (16*\h,-.5) {$y_1$};
			\node at (15*\h,-.5) {$y_2$};
			\node at (11*\h,-.5) {$y_3$};
			\node at (8*\h,-.5) {$y_4$};
			\node at (7*\h,-.5) {$y_5$};
			\node at (2*\h,-.5) {$y_6$};
			\draw[->, very thick] (2*\h,.3) to [in=180,out=90] (2.5*\h,.65) to [in=90, out=0] (3*\h,.3) 
					node [xshift=5,yshift=20] {$\frac{u-st^6}{u-st^5}$};
			\draw[->, very thick] (2*\h,.3) to [in=0,out=90] (1.5*\h,.65) to [in=90, out=180] (1*\h,.3) 
					node [xshift=-5,yshift=20] {$t\frac{u-st^5}{u-st^6}$};
			\draw[->, very thick] (8*\h,.3) to [in=0,out=90] (7.5*\h,.65) to [in=90, out=180] (7*\h,.3);
			\draw[ultra thick] (7.5*\h,.65)--++(.1,.2)--++(-.2,-.4)--++(.1,.2)--++(-.1,.2)--++(.2,-.4);
			\draw[->, very thick] (8*\h,.3) to [in=180,out=90] (8.5*\h,.65) to [in=90, out=0] (9*\h,.3) 
					node [xshift=5,yshift=20] {$\frac{u-st^4}{u-st^3}$};
	\end{tikzpicture}
	\caption{A new dynamic version of the ASEP.}
	\label{fig:dyn_ASEP}
	%\end{noindent}
\end{figure}

The connection between spin Hall-Littlewood process and DS6V and dynamic ASEP
hint at the possible integrability of the latter models, which might lead to
asymptotic results for them. We do not address this question in the present
paper. Note also that other dynamic generalizations of the stochastic six
vertex model and the ASEP were recently considered in
\cite{borodin2017elliptic}, \cite{aggarwal2017dynamical},
\cite{BorodinCorwin2017dynamic} in connection with vertex models related to
the Yang-Baxter equation for the elliptic quantum group
$E_{\tau,\eta}(\mathfrak{sl}_2)$. These dynamic models are different from the
ones introduced in the present work.

\subsection{Outline}

In \Cref{sec:bijectivisation} we outline the general idea of bijectivisation
of summation identities. In \Cref{sec:main_construction} we describe the
higher spin six vertex weights, the Yang-Baxter equation they satisfy, and its
bijectivisation with minimal ``noise''. In \Cref{sec:spin_HL_functions} we
recall the spin Hall-Littlewood symmetric functions and Cauchy summation
identities they satisfy. This section closely follows
\cite{Borodin2014vertex}. In \Cref{sec:local_transition_probabilities} we use
our bijectivisation of the Yang-Baxter equation sequentially to produce a
bijective proof of the skew Cauchy identity for the spin Hall-Littlewood
symmetric functions. In \Cref{sec:YB_field} we define our main object, the
Yang-Baxter field, and discuss its connection with spin Hall-Littlewood
measure and processes. In \Cref{sec:dynamicS6V} we consider a projection of
the Yang-Baxter field onto the column number zero leading to a new dynamic
version of the stochastic six vertex model. We also discuss a dynamic
Yang-Baxter equation for these dynamic six vertex weights. In
\Cref{sec:degenerations} we consider various degenerations of the dynamic
stochastic six vertex model. One of these degenerations produces a new dynamic
version of the ASEP. In \Cref{app:YB_equation,app:YB_probabilities} we
explicitly list all identities comprising the Yang-Baxter equation, and all
the forward and backward local transition probabilities coming out of our
bijectivisation of the Yang-Baxter equation. In \Cref{sub:another_Cauchy} we
discuss another versions of the skew Cauchy identity satisfied by the spin
Hall-Littlewood symmetric functions. In \Cref{app:inhomogeneous_construction}
we briefly outline extensions of our main constructions to the case of
inhomogeneous parameters spin Hall-Littlewood symmetric functions.

\subsection{Acknowledgments}

We appreciate helpful discussions with Alexei Borodin, Ivan Corwin, Grigori
Olshanski, and Nicolai Reshetikhin. The work was started when the authors
attended the 2017 IAS PCMI Summer Session on Random Matrices, and we are
grateful to the organizers for their hospitality and support. LP is partially
supported by the NSF grant DMS-1664617.

\section{Bijectivisation of summation identities}
\label{sec:bijectivisation}

\subsection{General formalism}

Here we explain the formal concept of bijectivisation of summation identities
which will be applied to the Yang-Baxter equation in
\Cref{sec:main_construction} below. Let $A$ and $B$ be two fixed finite
nonempty sets, and each element $a\in A$ and $b\in B$ is assigned certain
weight $w(a)$ or $w(b)$, respectively. Assume that the following summation
identity holds:
\begin{equation}
	\label{summation_identity}
	\sum_{a\in A}w(a)=\sum_{b\in B}w(b).
\end{equation}

\begin{definition}
	\label{def:bijectivisation}
	We say that the following data provides a \emph{bijectivisation} of
	identity \eqref{summation_identity}:
	\begin{itemize}
		\item
		      There are \emph{forward transition weights}
		      $p^{\mathrm{fwd}}(a,b)$ which satisfy
		      \begin{equation*}
			      \sum_{b\in B}p^{\mathrm{fwd}}(a,b)=1\quad
			      \textnormal{for each $a\in A$};
		      \end{equation*}
		\item
		      There are \emph{backward transition weights}
		      $p^{\mathrm{bwd}}(b,a)$ which satisfy
		      \begin{equation*}
			      \sum_{a\in A}p^{\mathrm{bwd}}(b,a)=1\quad
			      \textnormal{for each $b\in B$};
		      \end{equation*}
		\item
		      The transition weights satisfy the \emph{reversibility
			      condition}
		      \begin{equation}
			      \label{reversibility_condition}
			      w(a)p^{\mathrm{fwd}}(a,b)=
			      w(b)p^{\mathrm{bwd}}(b,a)\quad \textnormal{for each $a\in A$ and $b\in B$.}
		      \end{equation}
	\end{itemize}
\end{definition}

The term ``bijectivisation'' is justified by the following two observations.
First, if $A$ and $B$ have the same numbers of elements, $w(a)=w(b)=1$ for all
$a\in A$, $b\in B$, and each $p^{\mathrm{fwd}}(a,b)$ and
$p^{\mathrm{bwd}}(b,a)$ is either $0$ or $1$, then such a bijectivisation is
simply a bijection between $A$ and $B$.

Second, let us get back to the general situation of \Cref{def:bijectivisation}
and assume that a bijectivisation $\left\{ p^{\mathrm{fwd}}(a,b),
	p^{\mathrm{bwd}}(b,a)\right\}$ is given. Start from the left-hand side
of \eqref{summation_identity} and write
\begin{equation*}
	\sum_{a\in A}w(a)=
	\sum_{a\in A}w(a)
	\Biggl(
	\sum_{b\in B}p^{\mathrm{fwd}}(a,b)
	\Biggr)
	=
	\sum_{b\in B}w(b)
	\Biggl(
	\sum_{a\in A}p^{\mathrm{bwd}}(b,a)
	\Biggr)
	=\sum_{b\in B}w(b).
\end{equation*}
Then, due to the reversibility condition \eqref{reversibility_condition}, in
the middle two double sums the terms are in \emph{one-to-one correspondence}.
Thus, one can say that the transition weights $\left\{ p^{\mathrm{fwd}}(a,b),
	p^{\mathrm{bwd}}(b,a)\right\}$ produce a \emph{refinement} (or a
\emph{bijective proof}) of the initial identity~\eqref{summation_identity}.

\begin{remark}
	\label{rmk:bij_not_unique}
	Clearly, if both $A$ and $B$ have more than one element, then a
	bijectivisation is highly non-unique. However, in a concrete situation (such
	as for the Yang-Baxter equation in \Cref{sec:main_construction}) a particular
	bijectivisation might be more natural than the others. This choice would
	depend on additional structure of individual terms
	in~\eqref{summation_identity}.
\end{remark}

\subsection{Stochastic bijectivisation}

Now assume that the weights $w(a)$ and $w(b)$ in \eqref{summation_identity}
are \emph{stochastic}, i.e., they are positive\footnote{If some weights are
	equal to zero then let us remove the corresponding elements from $A$ and $B$.}
and sum to one: $\sum_{a\in A}w(a)=\sum_{b\in B}w(b)=1$. The latter condition
can always be achieved for positive weights $w(a),w(b)$ by dividing
\eqref{summation_identity} by their sum. If the transition weights in a
bijectivisation $\left\{ p^{\mathrm{fwd}}(a,b),
	p^{\mathrm{bwd}}(b,a)\right\}$ are all nonnegative, we call such
bijectivisation \emph{stochastic}. Another standard term used in Probability
Theory for such an object is \emph{coupling}.

A stochastic bijectivisation may be interpreted as a joint probability
distribution on $A\times B$ having prescribed marginal distributions $\left\{
	w(a) \right\}_{a\in A}$ and $\left\{ w(b) \right\}_{b\in B}$. The
forward and backward transition weights become families of conditional
distributions coming from this joint distribution on $A\times B$. The
reversibility condition \eqref{reversibility_condition} simply states the
compatibility between the two conditional distributions $p^{\mathrm{fwd}}$ and
$p^{\mathrm{bwd}}$.

One can also interpret $\{ p^{\mathrm{fwd}}(a,b)\}_{a\in A, b\in	B}$ as
a Markov transition matrix from $A$ to $B$, and similarly for
$p^{\mathrm{bwd}}$. This explains the terms ``transition weights'' and
``reversibility condition''.

If a stochastic bijectivisation has all transition weights $p^{\mathrm{fwd}},
	p^{\mathrm{bwd}}$ equal to 0 or 1, we call such bijectivisation
\emph{deterministic}.

\subsection{Examples}
\label{sub:examples}

Let us discuss two examples of bijectivisation relevant to the Yang-Baxter
equation considered in \Cref{sec:main_construction} below.

\subsubsection{One of the sets is a singleton}
\label{ssub:singleton}

For the first example, assume that $B=\left\{ b \right\}$ is a singleton while
$A=\{a_1,\ldots,a_n\}$ is an arbitrary finite set. The bijectivisation is
unique in this case and is given by
\begin{equation*}
	p^{\mathrm{fwd}}(a_i,b)=1,\qquad
	p^{\mathrm{bwd}}(b,a_i)=\frac{w(a_i)}{w(b)},\qquad i=1,\ldots,n .
\end{equation*}

\subsubsection{Both sets have two elements}
\label{ssub:two_and_two}

For the second example, consider the situation when both sets $A=\{a_1,a_2\}$,
$B=\{b_1,b_2\}$ have two elements, and all the four weights $w(a_i),w(b_j)$
are nonzero. In this case there are 8 forward and backward transition weights
which must solve 4 equations of the form
$p^{\mathrm{fwd}}(a_1,b_1)+p^{\mathrm{fwd}}(a_1,b_2)=1$, plus 4 more
reversibility equations involving the weights $w(a_i),w(b_j)$. However, since
the weights satisfy \eqref{summation_identity}, the reversibility equations
are not independent, and hence the rank of the system of linear equations on
the transition weights is equal to 7. (Another way to see this is to use
quantities from \eqref{reversibility_condition} as variables: there are 4
variables and 3 linearly independent conditions on them.)

Therefore, there is a one-parameter family of bijectivisations. One readily
checks that these solutions can be expressed in the following form:
\begin{equation}
	%\begin{noindent}
	\label{general_2_2_solution}
	\begin{array}{ll}
		p^{\mathrm{fwd}}(a_1,b_1)=\gamma,
		& \quad
		p^{\mathrm{fwd}}(a_1,b_2)=1-\gamma,
		\\[12pt]
		p^{\mathrm{fwd}}(a_2,b_1)=
		1-\dfrac{w(b_2)}{w(a_2)}+(1-\gamma)
		\dfrac{w(a_1)}{w(a_2)}
		,
		& \quad
		p^{\mathrm{fwd}}(a_2,b_2)=
		\dfrac{w(b_2)}{w(a_2)}-(1-\gamma)\dfrac{w(a_1)}{w(a_2)}
		,
		\\[12pt]
		p^{\mathrm{bwd}}(b_1,a_1)=
		\gamma\dfrac{w(a_1)}{w(b_1)},
		&\quad
		p^{\mathrm{bwd}}(b_1,a_2)=
		1-\gamma\dfrac{w(a_1)}{w(b_1)},
		\\[12pt]
		p^{\mathrm{bwd}}(b_2,a_1)=
		(1-\gamma)\dfrac{w(a_1)}{w(b_2)},
		&\quad
		p^{\mathrm{bwd}}(b_2,a_2)
		=1-(1-\gamma)
		\dfrac{w(a_1)}{w(b_2)}.
	\end{array}
%\end{noindent}
\end{equation}

Let us also consider a particular case of the above example when
$w(a_1)=w(b_1)$ (thus automatically $w(a_2)=w(b_2)$). In this case the
$\gamma$-dependent general solution \eqref{general_2_2_solution} simplifies.
Namely, it depends on the weights $w(\cdot)$ only through the combination
$(1-\gamma)w(a_1)/w(a_2)$. Thus, the most natural bijectivisation of the
summation identity
\begin{equation}
	\label{2_2_equation_easy_case}
	w(a_1)+w(a_2)=w(b_1)+w(b_2),\qquad w(a_1)=w(b_1),\quad w(a_2)=w(b_2)
\end{equation}
corresponds to choosing $\gamma=1$, does not depend on the weights $w(\cdot)$,
and is \emph{deterministic}. Namely, the term $w(a_1)$ is simply mapped to the
term $w(b_1)$ equal to it, and similarly for $w(a_2)$ and $w(b_2)$.

\section{Yang-Baxter equation and its bijectivisation}
\label{sec:main_construction}

The goal of this section is to apply bijectivisation of
\Cref{sec:bijectivisation} to Yang-Baxter equation for the (horizontal
spin-$\frac12$) higher spin six vertex model. This model corresponds to the
quantum group $U_q(\widehat{\mathfrak{sl}_2})$. The main outcome of this
section is the definition of forward and backward transition weights in
\Cref{sub:YB_bijectivisation_new_label}.

\subsection{Vertex weights}
\label{sub:vertex_weights}

Here we recall vertex weights of the higher spin six vertex model
introduced in 
\cite{KulishReshetikhin_YB_1981}.
In our formulas we adopt the parametrization used in
\cite{Borodin2014vertex}. 

The vertex weights depend on
the main ``\emph{quantization}'' parameter $t\in(0,1)$, the \emph{vertical
	spin} parameter $s$, and the \emph{spectral} parameter $u$, with only the
latter explicitly indicated in the notation. These weights are associated to a
vertex $(i_1,j_1;i_2,j_2)$ on the lattice $\mathbb{Z}^2$ which has $i_1$ and
$i_2$ incoming and outgoing vertical arrows, and $j_1$ and $j_2$ incoming and
outgoing horizontal arrows, respectively. We assume that our vertex model has
horizontal spin-$\frac{1}{2}$ and generic higher vertical spin, which is
equivalent to saying that the vertex weights are nonzero only if
$j_1,j_2\in\left\{0,1\right\}$ and $i_1,i_2\in\mathbb{Z}_{\ge0}$. (See also
\Cref{sub:degen_finite_spin} for a discussion of models with finite vertical
spin $I$ obtained by specializing the vertical spin parameter $s$ to
$t^{-I/2}$, $I\in\mathbb{Z}_{\ge1}$.) The arrows at any vertex should satisfy
the \emph{preservation property} $i_1+j_1=i_2+j_2$. Depending on $j_1,j_2$, we
will denote vertices by
\begin{equation}
	\label{single_vertex_notation}
	(g,0;g,0)=\Voo{.6}gg{-4.5pt}
	\,,
	\quad
	(g,0;g-1,1)=\Voi{.6}g{g-1}{-4.5pt}
	\,,
	\quad
	(g,1;g,1)=\Vii{.6}gg{-4.5pt}
	\,,
	\quad
	(g,1;g+1,0)=\Vio{.6}g{g+1}{-4.5pt}
	\,
\end{equation}
(see also \Cref{fig:vertex_weights} for a more detailed graphical
representation). Here $g\in\mathbb{Z}_{\ge0}$ is arbitrary, with the agreement
that $g\ge1$ in the second vertex. The weights of these vertices are defined
as
\begin{equation}
	%\begin{noindent}
	\label{vertex_weights}
	\begin{array}{cll}
													 & \Big[ \Voo{.6}gg{-4.5pt}
		\Big]_{u}:=
		\dfrac{1-st^gu}{1-su}, & \qquad \Big[
		\Voi{.6}g{g-1}{-4.5pt} \Big]_{u}:=
		\dfrac{(1-s^2t^{g-1})u}{1-su}, \\[8pt] &\Big[
		\Vii{.6}gg{-4.5pt} \Big]_{u} :=
		\dfrac{u-st^g}{1-su},  & \qquad \Big[
			\Vio{.6}{g}{g+1}{-4.5pt} \Big]_{u} :=
		\dfrac{1-t^{g+1}}{1-su}.
	\end{array}
%\end{noindent}
\end{equation}
Weights \eqref{vertex_weights} are very special in that they satisfy a
Yang-Baxter equation which we recall in the next subsection.

\begin{figure}[htbp]
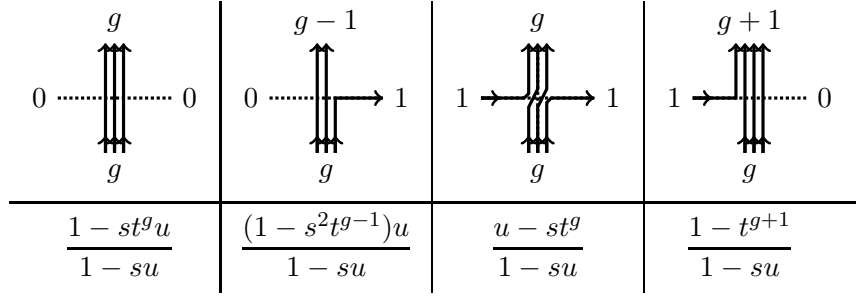

	\begin{tabular}{c|c|c|c}
		\vertexoo{1}
		 &
		\vertexol{1}
		 &
		\vertexll{1}
		 &
		\vertexlo{1}
		\\\hline\rule{0pt}{20pt}
		$\dfrac{1-s t^{g}u}{1-s u}$
		 &
		$\dfrac{(1-s ^{2}t^{g-1})u}{1-s u}$
		 &
		$\dfrac{u-s t^{g}}{1-s u}$
		 &
		$\dfrac{1-t^{g+1}}{1-s u}$
		\phantom{\Bigg|} \\
	\end{tabular}
	\caption{Possible vertices in the (horizontal spin-$\frac{1}{2}$)
		higher spin six vertex model, with their weights \eqref{vertex_weights}.}
	\label{fig:vertex_weights}
\end{figure}

\begin{remark}
	The higher spin weights \eqref{vertex_weights}
	of \cite{KulishReshetikhin_YB_1981}
	generalize the original six vertex weights
	\cite{pauling1935structure},
	\cite{Lieb1967SixVertex},
	\cite{baxter2007exactly}
	to the case when the vertical representation
	is arbitrary highest weight (corresponding to the spin parameter $s$),
	and the horizontal representation is still one-dimensional. 
	Using a procedure called fusion 
	\cite{KR1987Fusion},
	one can define vertex weights corresponding to both representations
	being arbitrary. Explicit formulas for fused vertex weights may be found in, e.g.,
	\cite{Mangazeev2014}, see also \cite{CorwinPetrov2015} for a probabilistic interpretation.
	In the present paper we only use the simpler weights \eqref{vertex_weights}
	and do not employ the fused ones.
\end{remark}

\begin{remark}
	We denote the quantization parameter of the higher spin six vertex
	model by $t$ instead of $q$ used in \cite{Borodin2014vertex},
	\cite{CorwinPetrov2015}, \cite{BorodinPetrov2016inhom}. This is done to
	highlight properties (in particular, Cauchy summation identities) of the spin
	Hall-Littlewood symmetric functions which degenerate at $s=0$ to the
	corresponding properties of the usual Hall-Littlewood symmetric polynomials.
	Vertex models in the context of Hall-Littlewood polynomials and their
	properties were recently studied in, e.g., \cite{BorodinBufetovWheeler2016},
	\cite{BufetovMatveev2017}, and we follow these papers when using the parameter
	$t$. Note that setting $s=t=0$ reduces the picture to the one associated with
	the classical Schur polynomials, see \Cref{sec:degenerations}.
\end{remark}

\subsection{Yang-Baxter equation}
\label{sub:YB}

The Yang-Baxter equation \cite{YangSystem1967},
\cite{baxter2007exactly},
\cite{KulishReshSkl1981yang}
can be regarded as the origin of integrability of the
stochastic
higher spin six vertex model, cf. \cite{BorodinPetrov2016inhom}. It can be
written in a rather compact form involving $4\times 4$ matrices containing
certain combinations of vertex weights. For example, see 
\cite[Proposition 2.5]{Borodin2014vertex} for the statement 
for our particular parametrization.
However, as we aim to construct a bijectivisation of
the Yang-Baxter equation in the sense of \Cref{sec:bijectivisation}, we need
to write the Yang-Baxter equation out in full detail, considering each of its
matrix elements separately.

Let us first define weights of auxiliary \emph{cross vertices}. The cross
vertices' incoming and outgoing arrow directions are rotated by $45^\circ$,
and along each direction there can be at most one arrow. Therefore, due to the
arrow preservation there are 6 possible cross vertices. Their weights depend
on two arbitrary spectral parameters $u,v$ and are defined as follows:
\begin{equation}
	\begin{array}{clll}
		                                                 &
		\left[ \ooYBoo{.35}{3}{14.5pt} \right]_{u,v}:=1, &
		\qquad
		\left[ \oiYBio{.35}{3}{14.5pt}
		\right]_{u,v}:=\uprho=\dfrac{u-v}{u-tv},         &
		\qquad
		\left[ \oiYBoi{.35}{3}{14.5pt}
			\right]_{u,v}:=1-\uprho=\dfrac{(1-t)v}{u-tv},
		\\[8pt]&
		\left[ \iiYBii{.35}{3}{14.5pt} \right]_{u,v}:=1, &
		\qquad
		\left[
			\ioYBoi{.35}{3}{14.5pt}
		\right]_{u,v}:=t\uprho=\dfrac{t(u-v)}{u-tv},     &
		\qquad
		\left[ \ioYBio{.35}{3}{14.5pt}
			\right]_{u,v}:=1-t\uprho=\dfrac{(1-t)u}{u-tv}.
	\end{array}
	\label{cross_vertex_weights}
\end{equation}
Here we employed the shorthand notation $\uprho:=(u-v)/(u-tv)$.

Let us now introduce notation for weights of \emph{pairs of vertices} where
one vertex as in \Cref{fig:vertex_weights} is put on top of another. Because
each of the two vertices in a pair can have at most one incoming and at most
one outgoing horizontal arrow, there are $2^4=16$ types of such pairs. Indeed,
choosing the numbers of horizontal arrows and saying that there are, say, $g$
incoming vertical arrows at the bottom determines the other numbers of
vertical arrows by the arrows preservation. The weight a pair of
vertices\footnote{The total weight of each particular arrow configuration
	containing several vertices is, by definition, equal to the product of weights
	of arrow configurations over all individual vertices.} depends on two spectral
parameters $u,v$, where $u$ corresponds to the bottom vertex. We will denote
pairs of vertices and their weights similarly to
\eqref{single_vertex_notation}--\eqref{vertex_weights}, as in the following
example:
\begin{align*}
	\bigg[ \oiWoi{.6}g{g+1}g{7pt} \bigg]_{u,v}
	=
	\Big[ \Vio{.6}g{g+1}{-4.5pt} \Big]_{u}
	\Big[ \Voi{.6}{g+1}g{-4.5pt} \Big]_{v}=
	\frac{(1-t^{g+1})(1-s^2t^g)v}{(1-su)(1-sv)}
	.
\end{align*}

We are now in a position to discuss the Yang-Baxter equation. In words, this
equation states that the partition function (i.e., the sum of weights of all
arrow configurations) in a configuration of a cross vertex followed by a pair
of vertices with spectral parameters $u,v$ is the same as the partition
function of a pair of vertices with parameters $v,u$ followed by a cross
vertex, provided that the boundary conditions on all 6 external edges are the
same. (In fact, thus defined partition functions are always sums of at most
two terms.) This leads to 16 types of identities \eqref{YB1.1}--\eqref{YB4.4}
(each depending on~$g$) which are listed in \Cref{app:YB_equation}.

\begin{remark}
	\label{rmk:YB_equation_numbers}
	The numbering of identities \eqref{YB1.1}--\eqref{YB4.4} reflects the
	boundary conditions on the left and right (the first and the second number,
	respectively). More precisely, equation numbers $\left\{ 1,2,3,4 \right\}$
	correspond to the boundary conditions %\begin{noindent}
	$\left\{
		\scalebox{.6}{\begin{tikzpicture} [scale=1.5, ultra thick, baseline=6pt]
				\draw[dotted] (-.5,0.1)--++(.4,0);
				\draw[dotted] (-.5,.4)--++(.4,0);
		\end{tikzpicture}},
		\scalebox{.6}{\begin{tikzpicture} [scale=1.5, ultra thick, baseline=6pt]
				\draw[->] (-.5,0.1)--++(.4,0);
				\draw[dotted] (-.5,.4)--++(.4,0);
		\end{tikzpicture}},
		\scalebox{.6}{\begin{tikzpicture} [scale=1.5, ultra thick, baseline=6pt]
				\draw[dotted] (-.5,0.1)--++(.4,0);
				\draw[->] (-.5,.4)--++(.4,0);
		\end{tikzpicture}},
		\scalebox{.6}{\begin{tikzpicture} [scale=1.5, ultra thick, baseline=6pt]
				\draw[->] (-.5,0.1)--++(.4,0);
				\draw[->] (-.5,.4)--++(.4,0);
		\end{tikzpicture}}
	\right\}$.
	%\end{noindent}
\end{remark}

For example, identity \eqref{YB3.3} among these reads
\begin{equation}
	\label{text_YB3.3}
	\biggl[
	\oiYBoi{.3}{3}{13.5pt}\ioWoi{.6}ggg{6.75pt}
	\biggr]_{u,v}
	+
	\biggl[
	\oiYBio{.3}{3}{13.5pt}\oiWoi{.6}g{g+1}g{6.75pt}
	\biggr]_{u,v}
	=
	\biggl[
	\ioWoi{.6}g{g}g{6.75pt}\oiYBoi{.3}{3}{13.5pt}
	\biggr]_{v,u}
	+
	\biggl[
	\ioWio{.6}g{g-1}g{6.75pt}\ioYBoi{.3}{3}{13.5pt}
	\biggr]_{v,u}.
\end{equation}
Here in the left-hand side $u$ is the spectral parameter of the bottom vertex,
while in the right-hand side the spectral parameter $u$ is at the top vertex.
The weights of the cross vertices in both sides are given by
\eqref{cross_vertex_weights} and are not affected by the flipping of the
spectral parameters. Writing out \eqref{text_YB3.3} as an identity between
rational functions, we obtain:
\begin{multline*}
	\frac{(1-t)v}{u-tv}
	\frac{(1-st^gu)(v-st^g)}{(1-su)(1-sv)}
	+
	\frac{u-v}{u-tv}
	\frac{(1-t^{g+1})(1-s^2t^g)v}{(1-su)(1-sv)}
	\\=
	\frac{(1-st^gv)(u-st^g)}{(1-sv)(1-su)}
	\frac{(1-t)v}{u-tv}
	+
	\frac{(1-s^2t^{g-1})v(1-t^g)}{(1-sv)(1-su)}
	\frac{t(u-v)}{u-tv},
\end{multline*}
which can be readily checked by hand. All other explicit Yang-Baxter
identities are listed in \Cref{app:YB_equation}.

\subsection{Bijectivisation of the Yang-Baxter equation}
\label{sub:YB_bijectivisation_new_label}

Our aim is now to bijectivise (in the sense of \Cref{sec:bijectivisation})
each of the 16 types of identities \eqref{YB1.1}--\eqref{YB4.4} given in
\Cref{app:YB_equation}. The forward weights corresponding to the Yang-Baxter
equation with spectral parameters $u,v$\footnote{That is, in the left-hand
	side of the Yang-Baxter equation the parameter $u$ is at the bottom vertex,
	$v$ is at the top vertex, and the weights of the cross vertices in both sides
	are given by \eqref{cross_vertex_weights}.} will be denoted by
$P^{\mathrm{fwd}}_{u,v}$, and the backward ones --- by
$P^{\mathrm{bwd}}_{u,v}$.

Now, note that both sides of each of the Yang-Baxter identities
\eqref{YB1.1}--\eqref{YB4.4} have at most two terms, and so the discussion
from \Cref{sub:examples} applies. First, we see that \Cref{ssub:singleton}
provides unique bijectivisation of 12 out of 16 types of the Yang-Baxter
identities, except \eqref{YB2.2}, \eqref{YB2.3}, \eqref{YB3.2}, and
\eqref{YB3.3}.

Second, among these four remaining identities, \eqref{YB2.3} and \eqref{YB3.2}
are of the form \eqref{2_2_equation_easy_case}, that is, we can identify equal
terms on both sides. Thus, let us choose the corresponding natural
deterministic bijectivisations of these identities as explained in the end of
\Cref{ssub:two_and_two}.

Finally, it remains to choose bijectivisations of identities \eqref{YB2.2} and
\eqref{YB3.3} for which one cannot deterministically identify terms in both
sides. Let us consider \eqref{YB2.2}, identity \eqref{YB3.3} can be treated
very similarly. Moreover, for any bijectivisation of the former identity there
is a unique bijectivisation of the latter satisfying the symmetries discussed
in \Cref{sub:YB_transition_symmetries} below. Thus, having a bijectivisation
of \eqref{YB2.2} we will then simply write down the bijectivisation of
\eqref{YB3.3} obtained using these symmetries.

Identity \eqref{YB2.2} has the form $w(a_1)+w(a_2)=w(b_1)+w(b_2)$, where
\begin{equation*}
	%\begin{noindent}
	a_1=\;
	\ioYBio{.3}{3}{13.5pt}\oiWio{.6}ggg{6.75pt}\;,\qquad
	a_2=\;
	\ioYBoi{.3}{3}{13.5pt}\ioWio{.6}g{g-1}g{6.75pt}\;,\qquad
	b_1=\;
	\oiWio{.6}g{g}{g}{6.75pt}\ioYBio{.3}{3}{13.5pt}\;,\qquad
	b_2=\;
	\oiWoi{.6}g{g+1}{g}{6.75pt}\oiYBio{.3}{3}{13.5pt}\;,
	%\end{noindent}
\end{equation*}
and the weights are given by (here $g\ge1$ because one of the arrow
configurations contains $g-1$ vertical arrows):
\begin{align*}
	w(a_1) & =
	\frac{(1-t)u}{u-tv}
	\frac{(u-st^g)(1-st^g v)}{(1-su)(1-sv)}
	,\qquad
	w(a_2)=
	\frac{u-v}{u-tv}
	\frac{t(1-t^g)(1-s^2t^{g-1})u}{(1-su)(1-sv)},
	\\
	w(b_1) & =
	\frac{(1-t)u}{u-tv}
	\frac{(v-st^g)(1-st^gu)}{(1-sv)(1-su)}
	,\qquad
	w(b_2)=
	\frac{u-v}{u-tv}
	\frac{(1-t^{g+1})(1-s^2t^g)u}{(1-sv)(1-su)}.
\end{align*}
All bijectivisations of \eqref{YB2.2} form a one-parameter family
\eqref{general_2_2_solution} employing the above weights. To select a
particular solution out of this one-parameter family, let us argue as follows.
Note that $w(a_2)$ vanishes when $u=v$, $t=0$, or $s^2=t^{1-g}$. When
$w(a_2)=0$, identity \eqref{YB2.2} simplifies and due to the discussion in
\Cref{ssub:singleton} has a unique bijectivisation. In particular, in this
case it should be $P^{\mathrm{bwd}}_{u,v}(b_1,a_2)=0$ (i.e., no mass can be
transferred into the term $w(a_2)=0$), which means that
\begin{equation*}
	\gamma(u,v,s,t,g)=\frac{w(b_1)}{w(a_1)}=
	\frac{(v-st^g)(1-st^gu)}{(u-st^g)(1-st^gv)}
	\qquad
	\textnormal{when $u=v$, $t=0$, or $s^2=t^{1-g}$.}
\end{equation*}
We will not address the question of whether the above conditions determine
$\gamma(u,v,s,t,g)$ uniquely (in a suitable class of functions), but instead
will take $\gamma(u,v,s,t,g)$ equal to the expression in the right-hand side
\emph{for all possible values} of $u,v,s,t,g$ (more discussion about the
choice of our particular bijectivisation may be found in \Cref{sub:discussion}
below). This choice of $\gamma$ leads via \eqref{general_2_2_solution} to the
following relatively simple forward and backward transition weights:
\begin{equation*}
	%\begin{noindent}
	\begin{array}{ll}
		P^{\mathrm{fwd}}_{u,v}(a_1,b_1)=\dfrac{(v-st^g)(1-st^gu)}{(u-st^g)(1-st^gv)},
		& \quad
		P^{\mathrm{fwd}}_{u,v}(a_1,b_2)=\dfrac{(u-v)(1-s^2t^{2g})}{(u-st^g)(1-st^gv)},
		\\[12pt]
		P^{\mathrm{fwd}}_{u,v}(a_2,b_1)=0
		,
		& \quad
		P^{\mathrm{fwd}}_{u,v}(a_2,b_2)=1
		,
		\\[12pt]
		P^{\mathrm{bwd}}_{u,v}(b_1,a_1)=1
		,
		&\quad
		P^{\mathrm{bwd}}_{u,v}(b_1,a_2)=0
		,
		\\[12pt]
		P^{\mathrm{bwd}}_{u,v}(b_2,a_1)=
		\dfrac{(1-t)(1-s^2t^{2g})}{(1-t^{g+1})(1-s^2t^{g})}
		,
		&\quad
		P^{\mathrm{bwd}}_{u,v}(b_2,a_2)
		=
		\dfrac{(t-t^{g+1})(1-s^2t^{g-1})}
		{(1-t^{g+1})(1-s^2t^g)}
		.
	\end{array}
	%\end{noindent}
\end{equation*}
This is the bijectivisation of identity \eqref{YB2.2} that we will use in the
present work.

A similar argument leads to the following forward and backward transition
weights corresponding to the Yang-Baxter identity \eqref{YB3.3}:
\begin{align*}
	%\begin{noindent}
		P^{\mathrm{fwd}}_{u,v}
		\biggl( \oiYBio{.3}{3}{13.5pt}\oiWoi{.6}{g-1}g{g-1}{6.75pt}\ ,\
		\ioWoi{.6}{g-1}{g-1}{g-1}{6.75pt}\oiYBoi{.3}{3}{13.5pt} \biggr)
		&=
		1-
		P^{\mathrm{fwd}}_{u,v}
		\biggl( \oiYBio{.3}{3}{13.5pt}\oiWoi{.6}{g-1}g{g-1}{6.75pt}\ ,\
		\ioWio{.6}{g-1}{g-2}{g-1}{6.75pt}\ioYBoi{.3}{3}{13.5pt} \biggr)
		=
		\dfrac{(1-t)(1-s^2t^{2g-2})}{(1-t^{g})(1-s^2t^{g-1})}
		;
		\\
		P^{\mathrm{fwd}}_{u,v}
		\biggl( \oiYBoi{.3}{3}{13.5pt}\ioWoi{.6}{g}g{g}{6.75pt}\ ,\
		\ioWoi{.6}{g}{g}{g}{6.75pt}\oiYBoi{.3}{3}{13.5pt} \biggr)
		&=
		P^{\mathrm{bwd}}_{u,v}
		\biggl(
			\ioWio{.6}{g+1}{g}{g+1}{6.75pt}\ioYBoi{.3}{3}{13.5pt}
			\ ,\
			\oiYBio{.3}{3}{13.5pt}\oiWoi{.6}{g+1}{g+2}{g+1}{6.75pt}
		\biggr)
		=1
		;
	\\
		P^{\mathrm{bwd}}_{u,v}
		\biggl(
			\ioWoi{.6}{g}{g}{g}{6.75pt}\oiYBoi{.3}{3}{13.5pt}
			\ ,\
			\oiYBoi{.3}{3}{13.5pt}\ioWoi{.6}{g}g{g}{6.75pt}
		\biggr)
		&=
		1-
		P^{\mathrm{bwd}}_{u,v}
		\biggl(
			\ioWoi{.6}{g}{g}{g}{6.75pt}\oiYBoi{.3}{3}{13.5pt}
			\ ,\
			\oiYBio{.3}{3}{13.5pt}\oiWoi{.6}{g}{g+1}{g}{6.75pt}
		\biggr)
		=
		\frac{(v-st^g)(1-st^gu)}{(u-st^g)(1-st^gv)}
		.
	%\end{noindent}
\end{align*}

All the forward and backward transition weights obtained above are organized
into tables in \Cref{fig:fwd_YB,fig:bwd_YB}, respectively. In
\Cref{app:YB_probabilities} these weights are listed in full detail.

\begin{figure}[htpb]
	%\begin{noindent}
	\centering
	\scalebox{.9}{
		$
			\begin{array}{c||c||c|c||c|c||c}
				P^{\mathrm{fwd}}_{u,v}& \ooYBoo{.3}{3}{13.5pt}
					& \ioYBio{.3}{3}{13.5pt}
					& \oiYBio{.3}{3}{13.5pt}
					& \oiYBoi{.3}{3}{13.5pt}
					& \ioYBoi{.3}{3}{13.5pt}
					& \iiYBii{.3}{3}{13.5pt}
				\phantom{\bigg.}
				\\\hline
				\ooYBoo{.3}{3}{13.5pt}
				\phantom{\Bigg.}
					& 1
					& \dfrac{(1-t)v}{u-t
					v}\dfrac{1-st^{g}u}{1-st^{g}v}
					&
				\coli
				\dfrac{u-v}{u-tv}
				\dfrac{1-st^{g+1}v}{1-st^{g}v}
					&
					\dfrac{(1-t)u}{u-tv}
					\dfrac{1-st^g
					v}{1-st^gu}
					& \colx
					\dfrac{t(u-v)}{u-tv}
					\dfrac{1-s t^{g-1}u}{1-s t^gu}
					& 1
				\\\hline
				\ioYBio{.3}{3}{13.5pt}
				\phantom{\Bigg.}
					& 1
					&
				\dfrac{v-st^g}{u-st^g}
				\dfrac{1-st^gu}{1-st^gv}
					&
				\coli
				\dfrac{u-v}{u-st^g}
				\dfrac{1-s^2t^{2g}}{1-st^gv}
					& 1
					& \colx0
					& 1
				\\\hline
				\oiYBio{.3}{3}{13.5pt}
				\phantom{\Bigg.}
					& \colx1
					& \colx0
					& 1
					&
				\colx
				\dfrac{1-t}{1-t^{g}}
				\dfrac{1-s^2t^{2g-2}}{1-s^2t^{g-1}}
					&
				\colxx
				\dfrac{t-t^{g}}{1-t^{g}}
				\dfrac{1-s^2t^{g-2}}{1-s^2t^{g-1}}
					& \colx1
				\\\hline
				\oiYBoi{.3}{3}{13.5pt}
				\phantom{\Bigg.}
					& 1
					& 1
					& \coli0
					& 1
					& \colx0
					& 1
				\\\hline
				\ioYBoi{.3}{3}{13.5pt}
				\phantom{\Bigg.}
					& \coli1
					& \coli0
					& \colii1
					& \coli0
					& 1
					& \coli1
				\\\hline
				\iiYBii{.3}{3}{13.5pt}
				\phantom{\Bigg.}
					& 1
					& \dfrac{(1-t)u}{u-t
					v}\dfrac{v-st^g}{u-st^g}
					&
				\coli\dfrac{u-v}{u-tv}
				\dfrac{u-st^{g+1}}{u-st^g}
					&
				\dfrac{(1-t)v}{u-tv}
				\dfrac{u-st^{g}}{v-st^{g}}
					&
				\colx\dfrac{t(u-v)}{u-tv}
				\dfrac{v-st^{g-1}}{v-st^{g}}
					& 1
			\end{array}
		$
	}
	\caption{Forward transition weights corresponding to the Yang-Baxter equation.
	Here $g$ is the number of vertical arrows in the middle before the move of the cross vertex.
	The coloring reflects the change of the number of vertical arrows in the middle after the move:
	pink and red correspond to transitions $g\to g+1$ and $g\to g+2$, while lighter and darker gray
	mean $g\to g-1$ and $g\to g-2$, respectively.}
	\label{fig:fwd_YB}
	%\end{noindent}
\end{figure}

\begin{figure}[htpb]
	%\begin{noindent}
	\centering
	\scalebox{.9}{
		$
			\begin{array}{c||c||c|c||c|c||c}
				P^{\mathrm{bwd}}_{u,v}& \ooYBoo{.3}{3}{13.5pt}
					& \ioYBio{.3}{3}{13.5pt}
					& \oiYBio{.3}{3}{13.5pt}
					& \oiYBoi{.3}{3}{13.5pt}
					& \ioYBoi{.3}{3}{13.5pt}
					& \iiYBii{.3}{3}{13.5pt}
				\phantom{\bigg.}
				\\\hline
				\ooYBoo{.3}{3}{13.5pt}
				\phantom{\Bigg.}
					& 1
					& \dfrac{(1-t)u}{u-tv}\dfrac{1-st^gv}{1-st^gu}
					& \coli \dfrac{u-v}{u-tv}\dfrac{1-st^{g+1}v}{1-st^gv}
					& \dfrac{(1-t)v}{u-tv}\dfrac{1-st^gu}{1-st^gv}
					& \colx \dfrac{t(u-v)}{u-tv}\dfrac{1-st^{g-1}u}{1-st^gu}
					& 1
				\\\hline
				\ioYBio{.3}{3}{13.5pt}
				\phantom{\Bigg.}
					& 1
					& 1
					& \coli 0
					& 1
					& \colx 0
					& 1
				\\\hline
				\oiYBio{.3}{3}{13.5pt}
				\phantom{\Bigg.}
					& \colx 1
					& \colx \dfrac{1-t}{1-t^g}\dfrac{1-s^{2}t^{2g-2}}{1-s^2t^{g-1}}
					& 1
					& \colx 0
					& \colxx \dfrac{t-t^g}{1-t^g}\dfrac{1-s^2t^{g-2}}{1-s^2t^{g-1}}
					& \colx 1
				\\\hline
				\oiYBoi{.3}{3}{13.5pt}
				\phantom{\Bigg.}
					& 1
					& 1
					& \coli
					\dfrac{u-v}{u-st^g}\dfrac{1-s^2t^{2g}}{1-st^gv}
					& \dfrac{v-st^g}{u-st^g}\dfrac{1-st^gu}{1-st^gv}
					& \colx 0
					& 1
				\\\hline
				\ioYBoi{.3}{3}{13.5pt}
				\phantom{\Bigg.}
					& \coli 1
					& \coli 0
					& \colii 1
					& \coli 0
					& 1
					& \coli 1
				\\\hline
				\iiYBii{.3}{3}{13.5pt}
				\phantom{\Bigg.}
					& 1
					& \dfrac{(1-t)v}{u-tv}\dfrac{u-st^g}{v-st^g}
					& \coli \dfrac{u-v}{u-tv}\dfrac{u-st^{g+1}}{u-st^g}
					& \dfrac{(1-t)u}{u-tv}\dfrac{v-st^g}{u-st^g}
					& \colx \dfrac{t(u-v)}{u-tv}\dfrac{v-st^{g-1}}{v-st^g}
					& 1
			\end{array}
		$
	}
	\caption{Backward transition weights corresponding to the Yang-Baxter equation.
	This table uses the same conventions as in \Cref{fig:fwd_YB}.}
	\label{fig:bwd_YB}
	%\end{noindent}
\end{figure}

\subsection{Symmetries}
\label{sub:YB_transition_symmetries}

The forward and backward transition weights just defined in
\Cref{sub:YB_bijectivisation_new_label} satisfy the following symmetries:
\begin{proposition}
	\label{prop:symm1}
	Fix any boundary conditions $k_1,k_2,k_1',k_2'\in\left\{ 0,1 \right\}$
	and $i_1,i_2\in \mathbb{Z}_{\ge0}$. Then for any $g_1,g_2\in
		\mathbb{Z}_{\ge0}$ we have the following identity between forward and backward
	transition weights:
	\begin{equation*}
		%\begin{noindent}
		P^{\mathrm{fwd}}_{u,v}
		\biggl(
			\scalebox{.6}{
			\begin{tikzpicture} [scale=1.5, very thick,
					baseline=6.75pt]
				\draw[densely dashed] (-.75,.25)--(-.5,0)--++(.4,0);
				\draw[densely dashed] (-.75,.25)--(-.5,.5)--++(.4,0);
				\draw (-1,0)--(-.75,.25);
				\draw (-1,.5)--(-.75,.25);
				\node[anchor=east] at (-1.05,0) {\LARGE$k_1$};
				\node[anchor=east] at (-1.05,.5) {\LARGE$k_2$};
				\node[anchor=west] at (.52,0) {\LARGE$k_1'$};
				\node[anchor=west] at (.52,.5) {\LARGE$k_2'$};
				\draw (.1,0)--++(.4,0);
				\draw (.1,.5)--++(.4,0);
				\node at (0,-.25) {\LARGE$i_1$};
				\node at (0,.25) {\LARGE$g_1$};
				\node at (0,.75) {\LARGE$i_2$};
			\end{tikzpicture}}
			\ ,\
			\scalebox{.6}{
			\begin{tikzpicture} [scale=1.5, very thick,
					baseline=6.75pt]
				\draw (-.5,0)--++(.4,0);
				\draw (-.5,.5)--++(.4,0);
				\node[anchor=east] at (-.52,0) {\LARGE$k_1$};
				\node[anchor=east] at (-.52,.5) {\LARGE$k_2$};
				\node[anchor=west] at (1.02,0) {\LARGE$k_1'$};
				\node[anchor=west] at (1.02,.5) {\LARGE$k_2'$};
				\draw[densely dashed] (.1,0)--++(.4,0)--++(.25,.25);
				\draw[densely dashed] (.1,.5)--++(.4,0)--++(.25,-.25);
				\draw (.75,.25)--++(.25,.25);
				\draw (.75,.25)--++(.25,-.25);
				\node at (0,-.25) {\LARGE$i_1$};
				\node at (0,.25) {\LARGE$g_2$};
				\node at (0,.75) {\LARGE$i_2$};
			\end{tikzpicture}}
		\biggr)
		=
		P^{\mathrm{bwd}}_{u,v}
		\biggl(
			\scalebox{.6}{
			\begin{tikzpicture} [scale=1.5, very thick,
					baseline=6.75pt]
				\draw (-.5,0)--++(.4,0);
				\draw (-.5,.5)--++(.4,0);
				\node[anchor=east] at (-.52,0) {\LARGE$k_2'$};
				\node[anchor=east] at (-.52,.5) {\LARGE$k_1'$};
				\node[anchor=west] at (1.02,0) {\LARGE$k_2$};
				\node[anchor=west] at (1.02,.5) {\LARGE$k_1$};
				\draw[densely dashed] (.1,0)--++(.4,0)--++(.25,.25);
				\draw[densely dashed] (.1,.5)--++(.4,0)--++(.25,-.25);
				\draw (.75,.25)--++(.25,.25);
				\draw (.75,.25)--++(.25,-.25);
				\node at (0,-.25) {\LARGE$i_2$};
				\node at (0,.25) {\LARGE$g_1$};
				\node at (0,.75) {\LARGE$i_1$};
			\end{tikzpicture}}
			\ ,\
			\scalebox{.6}{
			\begin{tikzpicture} [scale=1.5, very thick,
					baseline=6.75pt]
				\draw[densely dashed] (-.75,.25)--(-.5,0)--++(.4,0);
				\draw[densely dashed] (-.75,.25)--(-.5,.5)--++(.4,0);
				\draw (-1,0)--(-.75,.25);
				\draw (-1,.5)--(-.75,.25);
				\node[anchor=east] at (-1.05,0) {\LARGE$k_2'$};
				\node[anchor=east] at (-1.05,.5) {\LARGE$k_1'$};
				\node[anchor=west] at (.52,0) {\LARGE$k_2$};
				\node[anchor=west] at (.52,.5) {\LARGE$k_1$};
				\draw (.1,0)--++(.4,0);
				\draw (.1,.5)--++(.4,0);
				\node at (0,-.25) {\LARGE$i_2$};
				\node at (0,.25) {\LARGE$g_2$};
				\node at (0,.75) {\LARGE$i_1$};
			\end{tikzpicture}}
		\biggr),
		%\end{noindent}
	\end{equation*}
	with the agreement that weights on both sides are well-defined (i.e.,
	$g_1$ and/or $g_2$ is $\ge1$ if needed). In both weights the numbers of arrows
	at the boundary are given, and the number of vertical arrows in the middle
	($g_1$ or $g_2$) determines the numbers of arrows along the dashed edges
	connecting the cross vertices with the two-vertex configurations.
\end{proposition}
\begin{proof}
	Straightforward verification.
\end{proof}
\begin{proposition}
	\label{prop:symm2}
	For any $k_1,k_2,k_1',k_2'\in \left\{ 0,1 \right\}$ and
	$i_1,i_2,g_1,g_2\in \mathbb{Z}_{\ge0}$ we have the following symmetry of the
	forward transition weights with respect to the change
	$(u,v)\to(v^{-1},u^{-1})$ in the spectral parameters:
	\begin{multline*}
		%\begin{noindent}
	P^{\mathrm{fwd}}_{u,v}
		\biggl(
			\scalebox{.6}{
			\begin{tikzpicture} [scale=1.5, very thick,
					baseline=6.75pt]
				\draw[densely dashed] (-.75,.25)--(-.5,0)--++(.4,0);
				\draw[densely dashed] (-.75,.25)--(-.5,.5)--++(.4,0);
				\draw (-1,0)--(-.75,.25);
				\draw (-1,.5)--(-.75,.25);
				\node[anchor=east] at (-1.05,0) {\LARGE$k_1$};
				\node[anchor=east] at (-1.05,.5) {\LARGE$k_2$};
				\node[anchor=west] at (.52,0) {\LARGE$k_1'$};
				\node[anchor=west] at (.52,.5) {\LARGE$k_2'$};
				\draw (.1,0)--++(.4,0);
				\draw (.1,.5)--++(.4,0);
				\node at (0,-.25) {\LARGE$i_1$};
				\node at (0,.25) {\LARGE$g_1$};
				\node at (0,.75) {\LARGE$i_2$};
			\end{tikzpicture}}
			\ ,\
			\scalebox{.6}{
			\begin{tikzpicture} [scale=1.5, very thick,
					baseline=6.75pt]
				\draw (-.5,0)--++(.4,0);
				\draw (-.5,.5)--++(.4,0);
				\node[anchor=east] at (-.52,0) {\LARGE$k_1$};
				\node[anchor=east] at (-.52,.5) {\LARGE$k_2$};
				\node[anchor=west] at (1.02,0) {\LARGE$k_1'$};
				\node[anchor=west] at (1.02,.5) {\LARGE$k_2'$};
				\draw[densely dashed] (.1,0)--++(.4,0)--++(.25,.25);
				\draw[densely dashed] (.1,.5)--++(.4,0)--++(.25,-.25);
				\draw (.75,.25)--++(.25,.25);
				\draw (.75,.25)--++(.25,-.25);
				\node at (0,-.25) {\LARGE$i_1$};
				\node at (0,.25) {\LARGE$g_2$};
				\node at (0,.75) {\LARGE$i_2$};
			\end{tikzpicture}}
		\biggr)
		\\=
		P^{\mathrm{fwd}}_{v^{-1},u^{-1}}
		\biggl(
			\scalebox{.6}{
			\begin{tikzpicture} [scale=1.5, very thick,
					baseline=6.75pt]
				\draw[densely dashed] (-.75,.25)--(-.5,0)--++(.4,0);
				\draw[densely dashed] (-.75,.25)--(-.5,.5)--++(.4,0);
				\draw (-1,0)--(-.75,.25);
				\draw (-1,.5)--(-.75,.25);
				\node[anchor=east] at (-1.05,0) {\LARGE$1-k_2$};
				\node[anchor=east] at (-1.05,.5) {\LARGE$1-k_1$};
				\node[anchor=west] at (.52,0) {\LARGE$1-k_2'$};
				\node[anchor=west] at (.52,.5) {\LARGE$1-k_1'$};
				\draw (.1,0)--++(.4,0);
				\draw (.1,.5)--++(.4,0);
				\node at (0,-.25) {\LARGE$i_2$};
				\node at (0,.25) {\LARGE$g_1$};
				\node at (0,.75) {\LARGE$i_1$};
			\end{tikzpicture}}
			\ ,\
			\scalebox{.6}{
			\begin{tikzpicture} [scale=1.5, very thick,
					baseline=6.75pt]
				\draw (-.5,0)--++(.4,0);
				\draw (-.5,.5)--++(.4,0);
				\node[anchor=east] at (-.52,0) {\LARGE$1-k_2$};
				\node[anchor=east] at (-.52,.5) {\LARGE$1-k_1$};
				\node[anchor=west] at (1.02,0) {\LARGE$1-k_2'$};
				\node[anchor=west] at (1.02,.5) {\LARGE$1-k_1'$};
				\draw[densely dashed] (.1,0)--++(.4,0)--++(.25,.25);
				\draw[densely dashed] (.1,.5)--++(.4,0)--++(.25,-.25);
				\draw (.75,.25)--++(.25,.25);
				\draw (.75,.25)--++(.25,-.25);
				\node at (0,-.25) {\LARGE$i_2$};
				\node at (0,.25) {\LARGE$g_2$};
				\node at (0,.75) {\LARGE$i_1$};
			\end{tikzpicture}}
		\biggr).
		%\end{noindent}
	\end{multline*}
	An analogous identity holds for the backward transition weights.
\end{proposition}
\begin{proof}
	This can also be checked in a straightforward way, but the
	verification can be made shorter with the help of the previous
	\Cref{prop:symm1}.
\end{proof}

\subsection{Nonnegativity and probabilistic interpretation}
\label{sub:YB_transition_nonnegativity}

Let us now address the question of nonnegativity of the forward and backward
transition weights obtained in \Cref{sub:YB_bijectivisation_new_label}.
\begin{proposition}
	\label{prop:nonnegative_transition_weights}
	Assume that our parameters satisfy
	\begin{equation}
		\label{weights_nonnegativity_region}
		0\le t<1,\qquad  -1<s\le 0, \qquad  0\le v\le u.
	\end{equation}
	Then all the forward and backward transition weights
	$P^{\mathrm{fwd}}_{u,v}$, $P^{\mathrm{bwd}}_{u,v}$ are nonnegative.
\end{proposition}
\begin{proof}
	Observe that the nonnegativity of the forward and backward transition
	weights would hold if all the following quantities
	\begin{align*}
		 & \frac{u-v}{u-tv},\qquad
		\frac{(1-s  t^gu) (v-s t^g)}{(1-s  t^gv)(u-s t^g)},
		\qquad
		\frac{(1-t) (1-s^2 t^{2 g})}{(1-t^{g+1}) (1-s^2 t^g)},
		\\
		 & \frac{ 1-s t^{g+1}v}
		{ 1-s t^gv},
		\qquad
		\frac{t  -s  t^{g+1}u}{ 1-s t^{g+1}u},
		\qquad
		\frac{t  v-s t^{g+1}}{ v-s t^{g+1}},
		\qquad
		\frac{ u-s t^{g+1}}{
			u-s t^g}
	\end{align*}
	(with arbitrary $g\in \mathbb{Z}_{\ge0}$) are between $0$ and $1$.
	The latter directly follows from \eqref{weights_nonnegativity_region}.
\end{proof}

\Cref{prop:nonnegative_transition_weights} implies that under
conditions \eqref{weights_nonnegativity_region} the forward and backward
weights from \Cref{sub:YB_bijectivisation_new_label} define \emph{Markov}
transition steps. We call them the (\emph{local}, \emph{randomized})
\emph{Yang-Baxter moves}:
\begin{definition}
	\label{def:local_moves}
	The \emph{forward Yang-Baxter move} transforms a fixed three-vertex
	configuration with the cross vertex on the left, given boundary conditions
	$k_1,k_1',k_2,k_2'\in\left\{ 0,1 \right\}$, $i_1,i_2\in\mathbb{Z}_{\ge0}$, and
	fixed number $g_1\in\mathbb{Z}_{\ge0}$ of vertical arrows in the middle, into
	a three-vertex configuration with the cross vertex on the right, having the
	same boundary conditions and a \emph{random} number $g_2$ of vertical arrows
	in the middle. Depending on the boundary conditions, $g_2$ can take at most
	two possible values which are two consecutive numbers chosen from $\left\{
		g_1-2,g_1-1,g_1,g_1+1,g_1+2	\right\}$.

	Similarly, the \emph{backward Yang-Baxter move} transforms a fixed
	three-vertex configuration with the cross vertex on the right, given boundary
	conditions, and fixed $g_2$, into a three-vertex configuration with the cross
	vertex on the left, same boundary conditions, and \emph{random} $g_1$.

	The probabilities of forward and backward Yang-Baxter moves are given
	in \Cref{fig:fwd_YB,fig:bwd_YB} (and also in \Cref{app:YB_probabilities} in
	full detail). See \Cref{fig:probabilistic_transition} for an illustration.
\end{definition}

\begin{figure}[htpb]
	\centering
	\scalebox{.6}{
		\begin{tikzpicture}[scale=1.5, very thick]
			\draw[line width=1.8]
			(-.1,0)--(.9,1)--++(1.2,0);
			\draw[line width=1.8]
			(-.1,1)--(.9,0)--++(1.2,0);
			\draw[line width=5] (1.5,-.5)--++(0,2);
			\node[circle, draw, fill=white] at (1.5,0)
			{\large$u$};
			\node[circle, draw, fill=white] at (1.5,1)
			{\large$v$};
			\node[anchor=north] at (1.5,-.55)
			{\LARGE$i_1$};
			\node[anchor=south] at (1.5,1.55)
			{\LARGE$i_2$};
			\node[anchor=west] at (1.55,.5) {\LARGE$g_1$};
			\node[anchor=east] at (-.12,0) {\LARGE$k_1$};
			\node[anchor=east] at (-.12,1) {\LARGE$k_2$};
			\node at (.75,-.25) {\LARGE$j_1$};
			\node at (.75,1.28) {\LARGE$j_2$};
			\node[anchor=west] at (2.12,0) {\LARGE$k_1'$};
			\node[anchor=west] at (2.12,1) {\LARGE$k_2'$};
			\draw[->, densely dashed, line width=2]
			(2.9,1.5) to
			[out=45,in=135] (5.7,1.5);
			\draw[<-, densely dashed, line width=2]
			(2.9,-.5) to
			[out=-45,in=-135] (5.7,-.5);
			\node at (4.3, 1.6) {\LARGE{}forward};
			\node at (4.3, -.5) {\LARGE{}backward};
			\begin{scope}[shift={(5.6,0)}]
				\draw[line width=1.8]
				(.9,1)--++(1.2,0);
				\draw[line width=1.8]
				(.9,0)--++(1.2,0);
				\draw[line width=5]
				(1.5,-.5)--++(0,2);
				\draw[line width=1.8]
				(1.5,0)--++(.6,0)--++(1,1);
				\draw[line width=1.8]
				(1.5,1)--++(.6,0)--++(1,-1);
				\node[circle, draw, fill=white] at
				(1.5,0)
				{\large$v$};
				\node[circle, draw, fill=white] at
				(1.5,1)
				{\large$u$};
				\node[anchor=north] at (1.5,-.55)
				{\LARGE$i_1$};
				\node[anchor=south] at (1.5,1.55)
				{\LARGE$i_2$};
				\node[anchor=west] at (1.55,.5)
				{\LARGE$g_2$};
				\node[anchor=east] at (.88,0)
				{\LARGE$k_1$};
				\node[anchor=east] at (.88,1)
				{\LARGE$k_2$};
				\node at (2.2,-.25) {\LARGE$j_1'$};
				\node at (2.2,1.28) {\LARGE$j_2'$};
				\node[anchor=west] at (3.12,0)
				{\LARGE$k_1'$};
				\node[anchor=west] at (3.12,1)
				{\LARGE$k_2'$};
			\end{scope}
		\end{tikzpicture}
	}
	\caption{Randomized Yang-Baxter moves turning fixed $g_1$ to random
		$g_2$ or vice versa. Note that the numbers of arrows $j_1,j_2$ or $j_1',j_2'$
		in the middle are uniquely determined by $k_1,k_2,k_1',k_2',i_1,i_2$ and $g_1$
		or $g_2$, respectively, and thus do not need to be specified explicitly.}
	\label{fig:probabilistic_transition}
\end{figure}
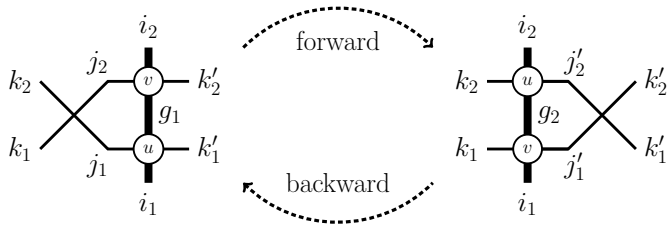

\subsection{On the choice of bijectivisation}
\label{sub:discussion}

In Section \Cref{sub:YB_bijectivisation_new_label} we presented a particular
choice of bijectivisation of the Yang-Baxter equation, and the rest of the
paper will be devoted to the study of the objects associated with the choice.
However, there are other reasonable choices, for which a very similar
discussion would be possible. To simplify the exposition, we will not focus on
them and just briefly mention possible variations in this section.

The Yang-Baxter equation consists of 16 identities between rational functions
listed in \Cref{app:YB_equation}. Twelve of them contain only one term in at
least one side of an equation and thus have a unique bijectivisation.
Identities \eqref{YB2.2}, \eqref{YB2.3}, \eqref{YB3.2}, and \eqref{YB3.3}
contain two terms on each side, so according to \Cref{ssub:two_and_two} each
of these identities admits a one-parameter family of bijectivisations. It is
easy to check that the choice of bijectivisations of these identities
presented in \Cref{sub:YB_bijectivisation_new_label} uniquely determined by
the following properties:
\begin{enumerate}
	\item
	      (\emph{Nonnegativity})
	      Transition probabilities are non-negative.
	\item
	      (\emph{Minimal ``noise'' property})
	      As many transition probabilities as possible are equal to $0$.
\end{enumerate}
Indeed, in \eqref{YB2.3} and \eqref{YB3.2} two of the forward probabilities
can be made zero, and in \eqref{YB2.2} and \eqref{YB3.3} one forward
probability can be made zero. Which of these probabilities are zero is
uniquely determined by the non-negativeness.

Let us discuss the above conditions. The first one is a must have since we
want to obtain a stochastic object. Thus, it forces our four parameters to lie
within certain segments of the real line. However, the second condition has a
combinatorial flavor which is not crucial for obtaining reasonable
probabilistic models. For example, one can introduce another bijectivisation
by replacing it with a condition
\begin{enumerate}
	\item
	      [(2')]
	      (\emph{Independence from input}) Forward transition
	      probabilities do not depend on the state of the cross vertex
	      before the move.
\end{enumerate}

Condition (2') uses the idea of \cite{DiaconisFill1990}
(applied in a symmetric function
setting in \cite{BorFerr2008DF}). Also, as far as we know, the dynamics coming
from condition (2') was used by Andrea Sportiello \cite{Sportiello-private}
for simulations in our setting. However, this idea was not applied to
bijectivise the Cauchy identity (which requires both forward and backward
probabilities) or to construct a random field of signatures
(\Cref{sec:YB_field}).

We focus on condition (2) rather than (2') (or any other choice of four
parameters satisfying condition (1)) because due to less interaction it leads
to slightly simpler models. However, since 12 out of 16 identities coming from
the Yang-Baxter equation work in the same way for any bijectivisation, all
these models are fairly similar. In particular, the dynamic version of the six
vertex model (\Cref{sec:dynamicS6V}) and all its degenerations
(\Cref{sec:degenerations}) will appear for all bijectivizations.

Finally, let us notice that yet another motivation for a certain specific
choice of bijectivisation might come from the algebraic side related to the
matrix interpretation of the Yang-Baxter equation. We were not able to find a
natural condition along these lines.

\section{Spin Hall-Littlewood symmetric functions}
\label{sec:spin_HL_functions}

In this section we recall the symmetric rational functions defined in
\cite{Borodin2014vertex} and their basic properties including the Cauchy
summation identities. In this section we do not assume that the transition
weights are nonnegative.

\subsection{Signatures}
\label{sub:signatures}

We need to introduce some notation. For each $N\in \mathbb{Z}_{\ge1}$ let
\begin{equation*}
	\mathsf{Sign}_{N}:=\left\{ \lambda\in \mathbb{Z}^{N}\colon
	\lambda_1\ge \ldots\ge\lambda_N  \right\}
\end{equation*}
denote the set of \emph{signatures} with $N$ components.\footnote{Signatures
	are also sometimes called \emph{highest weights} as the set
	$\mathsf{Sign}_{N}$ indexes irreducible representations of the unitary group
	$U(N)$, e.g., see \cite{Weyl1946}.} For $\lambda\in\mathsf{Sign}_{N}$ denote
$\ell(\lambda):=N$ and call this the \emph{length} of $\lambda$. By agreement,
$\mathsf{Sign}_0$ consists of the single empty signature $\varnothing$. We
will also use the notation $|\lambda|:=\lambda_1+\ldots+\lambda_N$.

A signature $\lambda\in \mathsf{Sign}_{N}$ is called \emph{nonnegative} if
$\lambda_N\ge0$. The set of nonnegative signatures is denoted by
$\mathsf{Sign}_{N}^+\subset \mathsf{Sign}_{N}$. Let us set
$\mathsf{Sign}:=\bigcup_{N=0}^{\infty}\mathsf{Sign}_{N}$ and
$\mathsf{Sign}^+:=\bigcup_{N=0}^{\infty}\mathsf{Sign}_{N}^+$.

Nonnegative signatures are often referred to as (integer) \emph{partitions},
which are represented pictorially as \emph{Young diagrams}, e.g., see
\cite[Ch. I.1]{Macdonald1995}. While this way of representing signatures is
extremely useful in many contexts, we will employ another graphical
representation of signatures which works equally well for signatures having
negative parts.

Namely, associate to each $\mu\in \mathsf{Sign}_{N}$ a configuration of $N$
vertical arrows on $\mathbb{Z}$, with multiple arrows per site allowed, by
putting an arrow at each of the locations $\mu_1,\ldots,\mu_N\in \mathbb{Z} $.
In other words, write $\mu$ in multiplicative notation as
$\mu=\ldots(-1)^{m_{-1}}0^{m_0}1^{m_1}2^{m_2}\ldots  $, where $m_i:=\#\left\{
	j\colon \mu_j=i \right\}$, $i\in \mathbb{Z}$. Then put $m_i$ vertical arrows
at each site $i\in \mathbb{Z}$. Note that all but finitely many sites $i\in
	\mathbb{Z}$ will be empty. See \Cref{fig:arrow_signature}, left, for an
illustration.

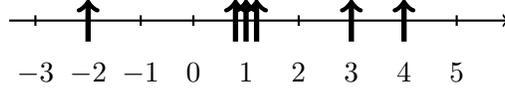
\begin{figure}[htpb]
	\centering
	\begin{tikzpicture}[scale=.7, thick]
		\draw[->] (-3.5,0)--++(9.5,0);
		\foreach \ii in {-3,...,5}
			{
				\draw (\ii,.1)--++(0,-.2) node[below,
					yshift=-10] {$\ii$};
			}
		\draw [line width=2,->] (4,-.4)--++(0,.8);
		\draw [line width=2,->] (3,-.4)--++(0,.8);
		\draw [line width=2,->] (1.2,-.4)--++(0,.8);
		\draw [line width=2,->] (1,-.4)--++(0,.8);
		\draw [line width=2,->] (.8,-.4)--++(0,.8);
		\draw [line width=2,->] (-2,-.4)--++(0,.8);
	\end{tikzpicture}
	\caption{Representing a signature $\mu=(4,3,1,1,1,-2)\in
			\mathsf{Sign}_6$ as a configuration of $6$ vertical arrows on $\mathbb{Z}$.}
	\label{fig:arrow_signature}
\end{figure}

\subsection{Definition of spin Hall-Littlewood functions}
\label{sub:spin_HL_definition}

Let us now recall the definitions of the symmetric rational functions
$F_{\lambda/\mu}$ and $G_{\lambda/\mu}^c$ introduced in
\cite{Borodin2014vertex}. 
Similar objects were also considered earlier as Bethe ansatz 
eigenfunctions, e.g., see \cite[Ch. VII]{QISM_book}, and also
\cite{Povolotsky2013}, \cite{BCPS2014} for more stochastic particle systems connections.

We begin by defining versions of the spin Hall-Littlewood functions
depending on one variable, the spectral parameter $u\in\mathbb{C}$.

\subsubsection{Functions $F_{\lambda/\mu}(u)$}
\label{ssub:F_definition}

Let a signature $\mu\in \mathsf{Sign}_{N-1}$ \emph{interlace} with a signature
$\lambda\in \mathsf{Sign}_{N}$ (notation: $\mu\prec\lambda$) which by
definition means that
\begin{equation}
	\label{interlacing_definition}
	\lambda_{N}\le \mu_{N-1}\le \lambda_{N-1}\le \ldots\le \lambda_2\le
	\mu_1\le \lambda_1 .
\end{equation}
There exists a unique configuration of arrows on the grid $\mathbb{Z}\times
	\left\{ -1,0,1	 \right\}$ connecting $\mu$ to $\lambda$ (see
\Cref{fig:connecting_interlacing}, left):
\begin{itemize}
	\item
	      vertical arrows $(\mu_i,-1)\to(\mu_i,0)$ entering from the
	      bottom;
	\item
	      vertical arrows $(\lambda_j,0)\to(\lambda_j,1)$ exiting at the
	      top;
	\item
	      horizontal arrows along $\mathbb{Z}\times \left\{0\right\}$ such
	      that the local configuration of arrows around each vertex of
	      $\mathbb{Z}\times\left\{0\right\}$ looks like one of the vertices in
	      \Cref{fig:vertex_weights}, and configurations of arrows at neighboring
	      vertices are compatible. There configuration of horizontal arrows is packed at
	      $-\infty$, and is empty at $+\infty$.
\end{itemize}

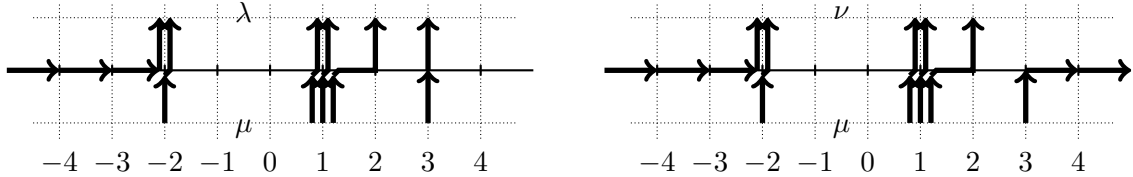
\begin{figure}[htpb]
	\centering
	\begin{tikzpicture}[scale=.7, thick]
		\draw (7.5,0)--++(9.5,0);
		\draw[densely dotted, line width=.4]
		(7.5,1)--++(9.2,0);
		\draw[densely dotted, line width=.4]
		(7.5,-1)--++(9.2,0);
		\foreach \ii in {-4,...,4}
			{
				\draw (\ii+12,.1)--++(0,-.2)
				node[below, yshift=-25] {$\ii$};
				\draw[densely dotted, line width=.4]
				(\ii+12,-1.3)--++(0,2.6);
			}
		\draw [line width=2,->] (15,-1)--++(0,1);
		\draw [line width=2,->] (15,0)--++(0,1);
		\draw [line width=2,->] (13.2,-1)--++(0,.9);
		\draw [line width=2,->]
		(13.2,-.1)--++(.1,.1)--++(.7,0)--++(0,1);
		\draw [line width=2,->] (13,-.1)--++(.1,.1)--++(0,1);
		\draw [line width=2,->]
		(12.8,-.1)--++(.1,.1)--++(0,1);
		\draw [line width=2,->] (13,-1)--++(0,.9);
		\draw [line width=2,->] (12.8,-1)--++(0,.9);
		\draw [line width=2,->] (10,-1)--++(0,.9);
		\draw [line width=2,->] (10,-.1)--++(.1,.1)--++(0,1);
		\draw [line width=2,->] (7,0)--++(1,0);
		\draw [line width=2,->] (8,0)--++(1,0);
		\draw [line width=2,->] (9,0)--++(.9,0);
		\draw [line width=2,->] (9.9,0)--++(0,1);
		\node at (11.5,1.1) {$\lambda$};
		\node at (11.5,-1.1) {$\mu$};
	\end{tikzpicture}
	\qquad
	\begin{tikzpicture}[scale=.7, thick]
		\draw (7.5,0)--++(9.5,0);
		\draw[densely dotted, line width=.4]
		(7.5,1)--++(9.2,0);
		\draw[densely dotted, line width=.4]
		(7.5,-1)--++(9.2,0);
		\foreach \ii in {-4,...,4}
			{
				\draw (\ii+12,.1)--++(0,-.2)
				node[below, yshift=-25] {$\ii$};
				\draw[densely dotted, line width=.4]
				(\ii+12,-1.3)--++(0,2.6);
			}
		\draw [line width=2,->] (15,-1)--++(0,1);
		\draw [line width=2,->] (15,0)--++(1,0);
		\draw [line width=2,->] (16,0)--++(1,0);
		\draw [line width=2,->] (13.2,-1)--++(0,.9);
		\draw [line width=2,->]
		(13.2,-.1)--++(.1,.1)--++(.7,0)--++(0,1);
		\draw [line width=2,->] (13,-.1)--++(.1,.1)--++(0,1);
		\draw [line width=2,->]
		(12.8,-.1)--++(.1,.1)--++(0,1);
		\draw [line width=2,->] (13,-1)--++(0,.9);
		\draw [line width=2,->] (12.8,-1)--++(0,.9);
		\draw [line width=2,->] (10,-1)--++(0,.9);
		\draw [line width=2,->] (10,-.1)--++(.1,.1)--++(0,1);
		\draw [line width=2,->] (7,0)--++(1,0);
		\draw [line width=2,->] (8,0)--++(1,0);
		\draw [line width=2,->] (9,0)--++(.9,0);
		\draw [line width=2,->] (9.9,0)--++(0,1);
		\node at (11.5,1.1) {$\nu$};
		\node at (11.5,-1.1) {$\mu$};
	\end{tikzpicture}
	\caption{
		Left: a configuration of horizontal arrows connecting
		$\mu=(3,1,1,1,-2)$ to $\lambda=(3,2,1,1,-2,-2)$, with $\mu\prec\lambda$.
		Right: a configuration of horizontal arrows connecting the same $\mu$ to
		$\nu=(2,1,1,-2,-2)$, with $\nu\mathop{\dot\prec}\mu$.
	}
	\label{fig:connecting_interlacing}
\end{figure}
For each $m\in \mathbb{Z}$, denote the numbers of incoming and outgoing
vertical and horizontal arrows at vertex $m\times\left\{ 0 \right\}$ by
$i_{1,2}(m)\in\mathbb{Z}_{\ge0}$ and $j_{1,2}(m)\in\left\{ 0,1 \right\}$,
respectively (this notation follows the beginning of
\Cref{sub:vertex_weights}).

Using this configuration of horizontal arrows connecting $\mu$ to $\lambda$,
define
\begin{equation}
	%\begin{noindent}
		\label{F_skew_one_variable_definition}
		F_{\lambda/\mu}(u):= \prod_{m=-\infty}^{-1}
		\frac{\Bigl[
			\scalebox{.6}{
				\begin{tikzpicture}[scale=1.5, ultra thick,
						baseline=-4.5pt]
					\node at (-1,0) {\LARGE$j_1(m)$};
					\node at (1,0) {\LARGE$j_2(m)$};
					\node at (0,.25) {\LARGE$i_1(m)$};
					\node at (0,-.25) {\LARGE$i_2(m)$};
				\end{tikzpicture}
			}
	\Bigr]_{u}}
	{\Bigl[
			\SVoioi{.6}{-4.5pt}
	\Bigr]_{u}}\;
	\prod_{m=0}^{\infty}
	\Bigl[
			\scalebox{.6}{
				\begin{tikzpicture}[scale=1.5, ultra thick,
						baseline=-4.5pt]
					\node at (-1,0) {\LARGE$j_1(m)$};
					\node at (1,0) {\LARGE$j_2(m)$};
					\node at (0,.25) {\LARGE$i_1(m)$};
					\node at (0,-.25) {\LARGE$i_2(m)$};
				\end{tikzpicture}
			}
	\Bigr]_{u}
	,
	%\end{noindent}
\end{equation}
where we use notation \eqref{vertex_weights} for the vertex weights depending
on the spectral parameter $u$. Observe that both products above are finite
since $i_{1,2}(-m)=i_{1,2}(m)=0$, $j_{1,2}(-m)=1$, $j_{1,2}(m)=0$ for all
sufficiently large $m$. If $\mu\not\prec\lambda$, set
$F_{\lambda/\mu}(u)\equiv 0$.

When $\mu,\lambda\in \mathsf{Sign}^+$, $F_{\lambda/\mu}$ defined by
\eqref{F_skew_one_variable_definition} coincides with the one given in
\cite{Borodin2014vertex}. Moreover, \eqref{F_skew_one_variable_definition}
extends the definition so that $F_{\lambda/\mu}$ for arbitrary signatures
$\mu\prec\lambda$ satisfies the following translation property:
\begin{equation}
	\label{F_shifting_property}
	F_{\lambda+(r^{N})/\mu+(r^{N-1})}(u)=\left( \frac{u-s}{1-su}
	\right)^{r} F_{\lambda/\mu}(u), \qquad \mu\in \mathsf{Sign}_{N-1}, \quad
	\lambda\in \mathsf{Sign}_N,
\end{equation}
where in the left-hand side we add arbitrary $r\in \mathbb{Z}$ to all parts of
both $\mu$ and $\lambda$.

\subsubsection{Functions $G_{\mu/\nu}^c(u)$}
\label{ssub:G_definition}

Let $\mu,\nu\in \mathsf{Sign}_N$. If these signatures satisfy
\begin{equation}
	\label{interlace2}
	\nu_N\le \mu_N\le \nu_{N-1}\ldots \mu_2\le \nu_2\le \mu_1,
\end{equation}
then we also say that $\nu$ and $\mu$ \emph{interlace}, but use a slightly
different notation $\nu\mathop{\dot\prec}\mu$ for this.

Let us connect $\mu$ to $\nu$ by a configuration of horizontal arrows in the
same sense as in \Cref{ssub:F_definition} above. Note that now the ``larger''
signature $\mu$ is placed at the \emph{bottom}. This implies that the
configuration of horizontal arrows connecting $\mu$ to $\nu$ contains
infinitely many horizontal arrows, both at $-\infty$ and at $+\infty$ (see
\Cref{fig:connecting_interlacing}, right).

Using this configuration of arrows connecting $\mu$ to $\nu$, define
\begin{equation}
	%\begin{noindent}
	\label{G_skew_one_variable_definition}
	G_{\mu/\nu}^{c}(u):=
	\prod_{m=-\infty}^{+\infty}
		\frac{\Bigl[
			\scalebox{.6}{
				\begin{tikzpicture}[scale=1.5, ultra thick,
						baseline=-4.5pt]
					\node at (-1,0) {\LARGE$j_1(m)$};
					\node at (1,0) {\LARGE$j_2(m)$};
					\node at (0,.25) {\LARGE$i_1(m)$};
					\node at (0,-.25) {\LARGE$i_2(m)$};
				\end{tikzpicture}
			}
	\Bigr]_{u^{-1}}}
	{\Bigl[
			\SVoioi{.6}{-4.5pt}
	\Bigr]_{u^{-1}}},
	%\end{noindent}
\end{equation}
where we used the same notation $i_{1,2}(m), j_{1,2}(m)$ for the numbers of
arrows at individual vertices of $\mathbb{Z}\times \left\{ 0 \right\}$ as in
\Cref{ssub:F_definition}. Again, observe that the product in
\eqref{G_skew_one_variable_definition} is actually finite. If
$\nu\mathop{\dot{\not\prec}}\mu$, set $G_{\mu/\nu}^{c}(u)\equiv 0$.

\begin{remark}
	\label{rmk:G_coincides_with_Bor17}
	Let us connect \eqref{G_skew_one_variable_definition} to the
	definition of $G_{\mu/\nu}^{c}$ given in \cite{Borodin2014vertex}. Denote
	\begin{equation*}
		\Bigl[
			\scalebox{.6}{
				\begin{tikzpicture}[scale=1.5, ultra thick,
						baseline=-4.5pt]
					\node at (-.4,0) {\LARGE$j_1$};
					\node at (.4,0) {\LARGE$j_2$};
					\node at (0,.25) {\LARGE$i_1$};
					\node at (0,-.25) {\LARGE$i_2$};
				\end{tikzpicture}
			}
			\Bigr]^{\bullet}_{u}
		:=
		\frac{\Bigl[
			\scalebox{.6}{
				\begin{tikzpicture}[scale=1.5, ultra thick,
						baseline=-4.5pt]
					\node at (-.4,0) {\LARGE$j_1$};
					\node at (.4,0) {\LARGE$j_2$};
					\node at (0,.25) {\LARGE$i_1$};
					\node at (0,-.25) {\LARGE$i_2$};
				\end{tikzpicture}
			}
			\Bigr]_{u^{-1}}}
		{\Bigl[
			\SVoioi{.6}{-4.5pt}
			\Bigr]_{u^{-1}}},
	\end{equation*}
	then from \eqref{vertex_weights} we have
	\begin{equation*}
		%\begin{noindent}
		\begin{array}{cll}
			&
			\Big[ \Voo{.6}gg{-4.5pt}
			\Big]_{u}^{\bullet}=
			\dfrac{u-st^g}{1-su},
			&
			\qquad
			\Big[
			\Voi{.6}g{g-1}{-4.5pt} \Big]_{u}^{\bullet}=
			\dfrac{1-s^2t^{g-1}}{1-su},
			\\[8pt]
			&
			\Big[
			\Vii{.6}gg{-4.5pt} \Big]_{u}^{\bullet} =
			\dfrac{1-st^gu}{1-su},
			&
			\qquad
			\Big[
				\Vio{.6}{g}{g+1}{-4.5pt} \Big]_{u}^{\bullet}
				=
			\dfrac{(1-t^{g+1})u}{1-su}.
		\end{array}
	%\end{noindent}
	\end{equation*}
	Observe that in the above graphical definition of $G_{\mu/\nu}^{c}$
	the ``larger'' signature $\mu$ is placed at the bottom. Replacing the
	right-pointing horizontal arrows by empty edges, and vice versa replacing
	empty edges by \emph{left-pointing} horizontal arrows leads to the conjugated
	vertex weights $w^{c}_u$ defined in \cite{Borodin2014vertex}:
	\begin{equation*}
		%\begin{noindent}
		\begin{array}{cll}
			&
			\Big[ \BVii{.6}gg{-4.5pt}
			\Big]_{u}^{c}=
			\dfrac{u-st^g}{1-su},
			&
			\qquad
			\Big[
			\BVio{.6}{g+1}{g}{-4.5pt} \Big]_{u}^{c}=
			\dfrac{1-s^2t^{g}}{1-su},
			\\[8pt]
			&
			\Big[
			\BVoo{.6}gg{-4.5pt} \Big]_{u}^{c} =
			\dfrac{1-st^gu}{1-su},
			&
			\qquad
			\Big[
				\BVoi{.6}{g-1}{g}{-4.5pt} \Big]_{u}^{c}
				=
			\dfrac{(1-t^{g})u}{1-su}.
		\end{array}
	%\end{noindent}
	\end{equation*}
	Then $G_{\mu/\nu}^c(u)$ is equal to the product of the conjugated
	weights $[\cdots]^{c}_u$ similar to \eqref{G_skew_one_variable_definition} but
	without the denominators (also with $\nu$ at the top and $\mu$ at the bottom).
	Note that \cite{Borodin2014vertex} also defines functions $G_{\mu/\nu}(u)$
	without the conjugation, but we do not use them in the present paper.
\end{remark}

\subsubsection{Multivariable functions $F$ and $G^c$}
\label{ssub:F_G_multivar_definition}

Using the single-variable functions \eqref{F_skew_one_variable_definition} and
\eqref{G_skew_one_variable_definition}, one can define the corresponding
multivariable functions $F$ and $G^c$.

Let $K\in \mathbb{Z}_{\ge1}$, $\lambda,\mu\in \mathsf{Sign}$, such that
$\ell(\lambda)=\ell(\mu)+K$, $\ell(\mu)=N\in \mathbb{Z}_{\ge0}$. Set
\begin{equation}
	\label{F_skew_multivariable}
	F_{\lambda/\mu}(u_1,\ldots,u_K ):= \sum_{ \{ \kappa^{(j)} \}}
	F_{\lambda/\kappa^{(K-1)}}(u_1) F_{\kappa^{(K-1)}/\kappa^{(K-2)}}(u_2)\ldots
	F_{\kappa^{(1)}/\mu}(u_K),
\end{equation}
where the sum runs over all $(K-1)$-tuples of signatures $\kappa^{(j)}\in
	\mathsf{Sign}_{N+j}$, $j=1,\ldots,K-1 $, such that $\mu\prec
	\kappa^{(1)}\prec\ldots\prec \kappa^{(K-1)}\prec\lambda$.
Equivalently, $F_{\lambda/\mu}(u_1,\ldots,u_K )$ can be thought of as the
partition function of a path configuration similar to the one in
\Cref{fig:connecting_interlacing}, left, but consisting of $K$ horizontal
layers. The signatures $\mu$ and $\lambda$ encode, respectively, the bottom
and the top boundary conditions, and there are additional $K$ paths entering
on the left.

The multivariable version of $G^{c}$ is defined in a similar way. Fix $K\in
	\mathbb{Z}_{\ge1}$, $N\in \mathbb{Z}_{\ge0}$, and let $\mu,\nu\in
	\mathsf{Sign}_{N}$. Set
\begin{equation}
	\label{G_skew_multivariable}
	G_{\lambda/\mu}^c(u_1,\ldots,u_K ):= \sum_{ \{\kappa^{(j)}\} }
	G_{\mu/\kappa^{(K-1)}}^{c}(u_1) G_{\kappa^{(K-1)}/\kappa^{(K-2)}}^{c}(u_2)
	\ldots G_{\kappa^{(1)}/\nu}(u_K),
\end{equation}
where the sum is taken over all $(K-1)$-tuples of signatures $\kappa^{(j)}\in
	\mathsf{Sign}_N$, $j=1,\ldots,K-1 $, satisfying
$\nu\mathop{\dot\prec}\kappa^{(1)}\mathop{\dot\prec}\ldots
	\mathop{\dot\prec}\kappa^{(K-1)}\mathop{\dot\prec}\mu$. Equivalently,
$G_{\lambda/\mu}^c(u_1,\ldots,u_K )$ is the partition function of path
configurations similar to the one in \Cref{fig:connecting_interlacing}, right,
but consisting of $K$ horizontal layers. The signatures $\mu$ and $\nu$
encode, respectively, the bottom and the top boundary conditions.

The Yang-Baxter equation for the vertex weights used to define the functions
$F_{\lambda/\mu}(u_1,\ldots, u_K)$ and $G_{\mu/\nu}^{c}(u_1,\ldots, u_K)$
readily implies that these functions are symmetric with respect to
permutations of the $u_j$'s. See \cite[Theorem 3.5]{Borodin2014vertex} for
details.

In special cases when the lower diagram is simple, the skew functions $F$ and
$G^c$ admit explicit formulas expressing them as sums over permutations. Let
us recall such a formula for $F_{\lambda/\varnothing}$. A formula for
$G^c_{\mu/(0,\ldots,0 )}$ (where the number of zeros is the same as the number
of components in $\mu$) is of similar nature but is more complicated, so we
omit it here and refer to \cite[Theorem 5.1]{Borodin2014vertex}, \cite[Theorem
	4.14]{BorodinPetrov2016inhom} for details on the statements and their proofs.
For the function $F_{\lambda/\varnothing}$ with $\lambda\in \mathsf{Sign}_N^+$
we have
\begin{equation}
	\label{F_symmetrization_formula}
	F_{\lambda/\varnothing}(u_1,\ldots,u_N ) =
	\frac{(1-t)^N}{\prod_{i=1}^{N}(1-su_i)} \sum_{\sigma\in S(N)} \prod_{1\le
		i<j\le N}\frac{u_{\sigma(i)}-t u_{\sigma(j)}}{u_{\sigma(i)}-u_{\sigma(j)}}
	\prod_{i=1}^{N} \biggl( \frac{u_{\sigma(i)}-s}{1-su_{\sigma(i)}}
	\biggr)^{\lambda_i}.
\end{equation}

\subsection{Cauchy summation identities}
\label{sub:Cauchy_identity}

One of the central properties of the functions $F$ and $G^c$ described in
\Cref{sub:spin_HL_definition} is that they satisfy summation identities of
Cauchy type \cite{Borodin2014vertex}. The most basic of these identities is
the one for the single-variable functions:

\begin{theorem}[{Single-variable skew Cauchy identity \cite[Theorem
					4.2]{Borodin2014vertex}}]
	\label{thm:skew_Cauchy_one}
	Let $u,v\in \mathbb{C}$ satisfy
	\begin{equation}
		\label{condition_on_convergence}
		\left|\frac{(u-s)(1-sv)}{(v-s)(1-su)}\right|<1
	\end{equation}
	Then for any $N\in \mathbb{Z}_{\ge0}$, $\lambda\in \mathsf{Sign}_{N}$,
	$\mu\in  \mathsf{Sign}_{N+1}$ we have (see \Cref{fig:skew_Cauchy} for a
	graphical illustration of both sides of the sum)
	\begin{equation}
		\label{skew_Cauchy_identity}
		\sum_{\kappa\in \mathsf{Sign}_{N}}
		G_{\lambda/\kappa}^{c}(v^{-1})F_{\mu/\kappa}(u) =
		\frac{v-u}{v-tu} \sum_{\nu\in \mathsf{Sign}_{N+1}}
		F_{\nu/\lambda}(u)\,G_{\nu/\mu}^{c}(v^{-1}).
	\end{equation}
\end{theorem}
\begin{remark}
	\label{rmk:skew_Cauchy}
	In \Cref{thm:skew_Cauchy_one} the sum over $\kappa$ in the left-hand
	side is finite, while the sum over $\nu$ in the right-hand side is infinite.
	Condition \eqref{condition_on_convergence} is needed to ensure the convergence
	of this infinite sum.
\end{remark}

In \Cref{sub:bijective_proof_skew_Cauchy} below we will present a new
bijective proof of the skew Cauchy identity of \Cref{thm:skew_Cauchy_one}
employing the forward and backward transition weights developed of
\Cref{sec:main_construction}. This bijective proof motivates a new version of
the skew Cauchy identity which we present in \Cref{sub:another_Cauchy}.

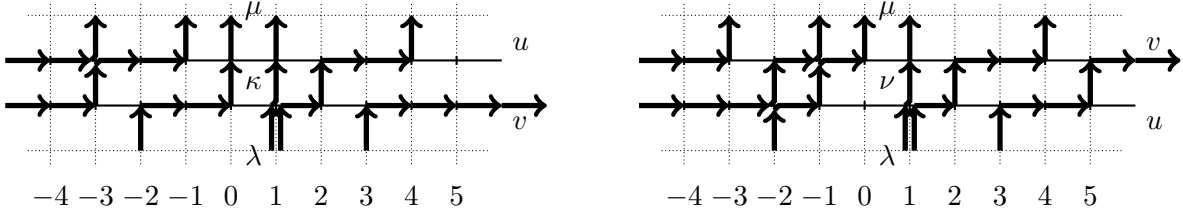
\begin{figure}[htpb]
\centering
\begin{tikzpicture}[scale=.6, thick]
%\begin{noindent}
		\draw (-4.5,0)--++(10.5,0) node[below right] {$v$};
		\draw (-4.5,1)--++(10.5,0) node[above right] {$u$};
		\draw[densely dotted, line width=.4]
		(-4.5,2)--++(10.2,0);
		\draw[densely dotted, line width=.4]
		(-4.5,-1)--++(10.2,0);
		\foreach \ii in {-4,...,5}
			{
				\draw (\ii,1.1)--++(0,-.2);
				\draw (\ii,.1)--++(0,-.2)
				node[below, yshift=-25] {$\ii$};
				\draw[densely dotted, line width=.4]
				(\ii,-1.3)--++(0,3.6);
			}
		\node at (0.5,-1.1) {$\lambda$};
		\draw [line width=2,->] (-2,-1)--++(0,1);
		\draw [line width=2,->] (.9,-1)--++(0,1);
		\draw [line width=2,->] (1.1,-1)--++(0,1);
		\draw [line width=2,->] (3,-1)--++(0,1);
		\node at (0.5,2.1) {$\mu$};
		\draw [line width=2,->] (-3.1,1)--++(.1,.1)--++(0,.9);
		\draw [line width=2,->] (-1,1)--++(0,1);
		\draw [line width=2,->] (0,1)--++(0,1);
		\draw [line width=2,->] (1,1)--++(0,1);
		\draw [line width=2,->] (4,1)--++(0,1);
		\node at (.5,0.5) {$\kappa$};
		\draw [line width=2,->] (-3,0)--++(0,.9);
		\draw [line width=2,->] (0,0)--++(0,1);
		\draw [line width=2,->] (2,0)--++(0,1);
		\draw [line width=2,->] (-5,1)--++(1,0);
		\draw [line width=2,->] (-4,1)--++(.9,0);
		\draw [line width=2,->] (-5,0)--++(1,0);
		\draw [line width=2,->] (-4,0)--++(1,0);
		\draw [line width=2,->] (-3,.9)--++(.1,.1)--++(.9,0);
		\draw [line width=2,->] (-2,1)--++(1,0);
		\draw [line width=2,->] (-2,0)--++(1,0);
		\draw [line width=2,->] (-1,0)--++(1,0);
		\draw [line width=2,->] (.9,0)--++(.1,.1)--++(0,.9);
		\draw [line width=2,->] (1.1,0)--++(.9,0);
		\draw [line width=2,->] (2,1)--++(1,0);
		\draw [line width=2,->] (3,1)--++(1,0);
		\draw [line width=2,->] (3,0)--++(1,0);
		\draw [line width=2,->] (4,0)--++(1,0);
		\draw [line width=2,->] (5,0)--++(1,0);
		\draw [line width=2,->] (6,0)--++(1,0);
	\end{tikzpicture}\qquad \quad
	\begin{tikzpicture}[scale=.6, thick]
		\draw (-4.5,0)--++(10.5,0) node[below right] {$u$};
		\draw (-4.5,1)--++(10.5,0) node[above right] {$v$};
		\draw[densely dotted, line width=.4]
		(-4.5,2)--++(10.2,0);
		\draw[densely dotted, line width=.4]
		(-4.5,-1)--++(10.2,0);
		\foreach \ii in {-4,...,5}
			{
				\draw (\ii,1.1)--++(0,-.2);
				\draw (\ii,.1)--++(0,-.2)
				node[below, yshift=-25] {$\ii$};
				\draw[densely dotted, line width=.4]
				(\ii,-1.3)--++(0,3.6);
			}
		\node at (0.5,-1.1) {$\lambda$};
		\draw [line width=2,->] (-2,-1)--++(0,.9);
		\draw [line width=2,->] (.9,-1)--++(0,1);
		\draw [line width=2,->] (1.1,-1)--++(0,1);
		\draw [line width=2,->] (3,-1)--++(0,1);
		\node at (0.5,2.1) {$\mu$};
		\draw [line width=2,->] (-3,1)--++(0,1);
		\draw [line width=2,->] (-1.1,1)--++(.1,.1)--++(0,.9);
		\draw [line width=2,->] (0,1)--++(0,1);
		\draw [line width=2,->] (1,1)--++(0,1);
		\draw [line width=2,->] (4,1)--++(0,1);
		\node at (.5,0.5) {$\nu$};
		\draw [line width=2,->] (-5,1)--++(1,0);
		\draw [line width=2,->] (-4,1)--++(1,0);
		\draw [line width=2,->] (-5,0)--++(1,0);
		\draw [line width=2,->] (-4,0)--++(1,0);
		\draw [line width=2,->] (-3,0)--++(.9,0);
		\draw [line width=2,->] (-2.1,0)--++(.1,.1)--++(0,.9);
		\draw [line width=2,->] (-2,-.1)--++(.1,.1)--++(.9,0);
		\draw [line width=2,->] (-2,1)--++(.9,0);
		\draw [line width=2,->] (-1,0)--++(0,.9);
		\draw [line width=2,->] (-1,.9)--++(.1,.1)--++(.9,0);
		\draw [line width=2,->] (.9,0)--++(.1,.1)--++(0,.9);
		\draw [line width=2,->] (1.1,0)--++(.9,0);
		\draw [line width=2,->] (2,0)--++(0,1);
		\draw [line width=2,->] (2,1)--++(1,0);
		\draw [line width=2,->] (3,1)--++(1,0);
		\draw [line width=2,->] (3,0)--++(1,0);
		\draw [line width=2,->] (4,0)--++(1,0);
		\draw [line width=2,->] (5,0)--++(0,1);
		\draw [line width=2,->] (5,1)--++(1,0);
		\draw [line width=2,->] (6,1)--++(1,0);
	\end{tikzpicture}
	\caption{
		Illustration of the sums in the skew Cauchy identity \eqref{skew_Cauchy_identity}
		with $\lambda=(3,1,1,-2)$, $\mu=(4,1,0,-1,-3)$, and $N=4$.
		Left: The (finite) sum runs over $\kappa\in\mathsf{Sign}_N$
		with $\lambda\mathop{\dot\succ}\kappa\prec \mu$.
		Right: The (infinite) sum runs over $\nu\in \mathsf{Sign}_{N+1}$ with
		$\lambda\prec \nu\mathop{\dot\succ}\mu$. Spectral parameters
		corresponding to the two horizontal layers are also indicated.
	}
	\label{fig:skew_Cauchy}
	%\end{noindent}
\end{figure}

Via iteration (cf. \eqref{F_skew_multivariable},
\eqref{G_skew_multivariable}), the skew Cauchy identity of
\Cref{thm:skew_Cauchy_one} implies the following multivariable identity:
\begin{corollary}[Multivariable skew Cauchy identity]
	\label{cor:skew_multi_Cauchy}
	Let $u_1,\ldots,u_K,v_1,\ldots,v_L\in \mathbb{C}$ be such that each
	pair $(u_i,v_j)$ satisfies \eqref{condition_on_convergence}. For any $N\in
		\mathbb{Z}_{\ge0}$, $\lambda\in \mathsf{Sign}_{N}$, and $\mu\in
		\mathsf{Sign}_{N+K}$, we have
	\begin{multline}
		\label{skew_multi_Cauchy}
		\sum_{\kappa\in\mathsf{Sign}_N}
		G_{\lambda/\kappa}^{c}(v_1^{-1},\ldots,v_L^{-1} )
		F_{\mu/\kappa}(u_1,\ldots,u_K ) \\=
		\prod_{i=1}^{K}\prod_{j=1}^{L}\frac{v_j-u_i}{v_j-tu_i} \sum_{\nu\in
			\mathsf{Sign}_{N+K}} F_{\nu/\lambda}(u_1,\ldots,u_K )\,
		G_{\nu/\mu}^{c}(v_1^{-1},\ldots,v_L^{-1} )
	\end{multline}
\end{corollary}
Next, setting $\lambda=\varnothing$ and $\mu=(0^K)$ in
\Cref{cor:skew_multi_Cauchy}, we get:
\begin{corollary}[Ordinary Cauchy identity]
	Let $u_1,\ldots,u_K,v_1,\ldots,v_L\in \mathbb{C}$ be such that each
	pair $(u_i,v_j)$ satisfies \eqref{condition_on_convergence}. Then we have
	\label{cor:nonskew_multi_Cauchy}
	\begin{equation}
		\label{nonskew_multi_Cauchy}
		\prod_{i=1}^{K}\frac{1-t^i}{1-su_i}
		=
		\prod_{i=1}^{K}\prod_{j=1}^{L}\frac{v_j-u_i}{v_j-tu_i}
		\sum_{\nu\in\mathsf{Sign}^+_{K}}
		F_{\nu/\varnothing}(u_1,\ldots,u_K )\,
		G_{\nu/(0^K)}^{c}(v_1^{-1},\ldots,v_L^{-1} ) .
	\end{equation}
	Note that here the sum runs over nonnegative signatures because all
	parts of $\mu=(0^K)$ are nonnegative.
\end{corollary}
\section{Transition probabilities $\mathsf{U}^{\mathrm{fwd}}$ and
$\mathsf{U}^{\mathrm{bwd}}$ on signatures}
\label{sec:local_transition_probabilities}

In this section, employing the vertex level forward and backward transition
probabilities from \Cref{sec:main_construction}, we define the transition
probabilities on signatures $\mathsf{U}_{v,u}^{\mathrm{fwd}}
	(
	\kappa\to \nu\mid \lambda,\mu )$ and $\mathsf{U}_{v,u}^{\mathrm{bwd}}(
	\nu\to\kappa\mid \lambda,\mu
	)$. The latter probabilities are in particular used to give a new
bijective proof of the skew Cauchy identity (\Cref{thm:skew_Cauchy_one}).

\subsection{Definition of transition probabilities on signatures}
\label{sub:definition_of_three_signature_probabilities}

Throughout this section we assume that our parameters satisfy
\begin{equation}
	\label{weights_nonnegativity_region_v_u}
	0\le t<1,\qquad  -1<s\le 0, \qquad  0\le u< v<1,
\end{equation}
so that the probabilities $P_{v,u}^{\mathrm{fwd}}$ and
$P_{v,u}^{\mathrm{bwd}}$ of the local Yang-Baxter moves are nonnegative
(thanks to \Cref{prop:nonnegative_transition_weights}). In particular, this
implies the convergence condition \eqref{condition_on_convergence} in Cauchy
identities. Note that we need a strict inequality in
\eqref{condition_on_convergence}, and for that we require $u<v$.

The condition $v<1$ (hence $u<1$) included in
\eqref{weights_nonnegativity_region_v_u} ensures that the vertex weights
\eqref{vertex_weights} with spectral parameters $u$ and $v$ are nonnegative.
This property will be essential in \Cref{sec:YB_field} below.

% and here we need it for some remarks, but let us keep this for uniformity of notation

\begin{remark}
	\label{rmk:swap_u_v}
	In this and the following sections (in comparison with
	\Cref{sec:main_construction}) we swap the parameters
	$(u,v)\leftrightarrow(v,u)$ it probabilities of the local Yang-Baxter moves.
	The swapped parameters (corresponding to $P_{v,u}^{\mathrm{fwd}}$ and
	$P_{v,u}^{\mathrm{bwd}}$) match the skew Cauchy identities of
	\Cref{sub:Cauchy_identity} (cf. \Cref{fig:skew_Cauchy}).
\end{remark}

Fix $N\in \mathbb{Z}_{\ge0}$, and let $\kappa,\lambda\in
	\mathsf{Sign}_N$ and $\mu\in \mathsf{Sign}_{N+1}$ such that
$\lambda\mathop{\dot\succ}\kappa\prec \mu$ be fixed. For each $\nu\in
	\mathsf{Sign}_{N+1}$ we define the forward transition probability
$\mathsf{U}_{v,u}^{\mathrm{fwd}}\left(\kappa\to \nu\mid \lambda,\mu
	\right)$ by constructing a random signature $\nu$ as follows.

Consider the two-layer arrow configuration as in \Cref{fig:skew_Cauchy}, left,
with signatures $\lambda,\kappa,\mu$ appearing from bottom to top. Observe
that this configuration has boundary conditions $\scalebox{.6}{
		\begin{tikzpicture} [scale=1.5, ultra thick, baseline=6pt]
			\draw[->] (-.5,0.1)--++(.4,0);
			\draw[->] (-.5,.4)--++(.4,0);
		\end{tikzpicture}
	}$ on the far left and $\scalebox{.6}{
		\begin{tikzpicture} [scale=1.5, ultra thick, baseline=6pt]
			\draw[->] (-.5,0.1)--++(.4,0);
			\draw[dotted] (-.5,.4)--++(.4,0);
		\end{tikzpicture}
	}$ on the far right, and, moreover, cannot contain vertical arrows to
the left of $\mu_{N+1}$ and to the right of $\lambda_1$. Add the cross vertex
$\iiYBii{.25}{3.5}{10.5pt}$ to the left of an arbitrary location $M\le
	\mu_{N+1}$. Then for each $r=M,M+1,\ldots $ perform the forward randomized
Yang-Baxter move which drags the cross to the right through the column number
$r$. Let these forward Yang-Baxter moves have probabilities
$P_{v,u}^{\mathrm{fwd}}$ given in \Cref{fig:fwd_YB}. This sequence of forward
Yang-Baxter moves will not affect the signatures $\lambda,\mu$, and will
randomly change $\kappa$, cf. \Cref{fig:U_transition_probabilities}.

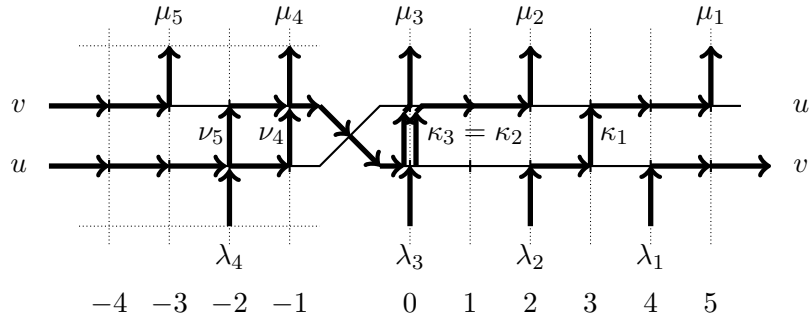
\begin{figure}[htpb]
\centering
\begin{tikzpicture}[scale=.8, thick]
%\begin{noindent}
	\draw (-4.5,0)--++(4,0)--++(1,1)--++(6,0);
	\draw (-4.5,1)--++(4,0)--++(1,-1)-++(6,0);
	\node at (-5.5,0) {$u$};
	\node at (-5.5,1) {$v$};
	\node at (7.5,-0) {$v$};
	\node at (7.5,1) {$u$};
		\draw[densely dotted, line width=.4]
		(-4.5,2)--++(4,0);
		\draw[densely dotted, line width=.4]
		(-4.5,-1)--++(4,0);
		\foreach \ii in {-4,...,-1}
			{
				\draw (\ii,1.1)--++(0,-.2);
				\draw (\ii,.1)--++(0,-.2)
				node[below, yshift=-42] {$\ii$};
				\draw[densely dotted, line width=.4]
				(\ii,-1.3)--++(0,3.6);
			}
		\foreach \ii in {0,...,5}
			{
				\draw (\ii+1,1.1)--++(0,-.2);
				\draw (\ii+1,.1)--++(0,-.2)
				node[below, yshift=-42] {$\ii$};
				\draw[densely dotted, line width=.4]
				(\ii+1,-1.3)--++(0,3.6);
			}
		\node at (-2,-1.5) {$\lambda_4$};
		\node at (1,-1.5) {$\lambda_3$};
		\node at (3,-1.5) {$\lambda_2$};
		\node at (5,-1.5) {$\lambda_1$};
		\draw [line width=2,->] (-2,-1)--++(0,.95);
		\node at (-3,2.5) {$\mu_5$};
		\node at (-1,2.5) {$\mu_4$};
		\node at (1,2.5) {$\mu_3$};
		\node at (3,2.5) {$\mu_2$};
		\node at (6,2.5) {$\mu_1$};
		\node at (-2.3,0.5) {$\nu_5$};
		\node at (-1.3,0.5) {$\nu_4$};
		\node[anchor=west] at (1.1,0.5) {$\kappa_3=\kappa_2$};
		\node at (4.4,0.5) {$\kappa_1$};
		\draw [line width=2,->] (6,0)--++(1,0);
		\draw [line width=2,->] (-5,1)--++(1,0);
		\draw [line width=2,->] (-4,1)--++(1,0);
		\draw [line width=2,->] (-3,1)--++(0,1);
		\draw [line width=2,->] (-1,1)--++(0,1);
		\draw [line width=2,->] (-5,0)--++(1,0);
		\draw [line width=2,->] (-4,0)--++(1,0);
		\draw [line width=2,->] (-3,0)--++(.95,0);
		\draw [line width=2,->] (-2,0)--++(0,1);
		\draw [line width=2,->] (-2,1)--++(.95,0);
		\draw [line width=2,->] (-2,0)--++(1,0);
		\draw [line width=2,->] (-1,0)--++(0,.95);
		\draw [line width=2,->] (-1,1)--++(.5,0);
		\draw [line width=2,->] (-.5,1)--++(.5,-.5);
		\draw [line width=2,->] (0,.5)--++(.5,-.5);
		\draw [line width=2,->] (.5,0)--++(.4,0);
		\draw [line width=2,->] (.9,0)--++(0,.9);
		\draw [line width=2,->] (.9,.9)--++(.1,.1)--++(0,1);
		\draw [line width=2,->] (1.1,0)--++(0,.9);
		\draw [line width=2,->] (1.1,.9)--++(.1,.1)--++(.9,0);
		\draw [line width=2,->] (1,-1)--++(0,1);
		\draw [line width=2,->] (2,1)--++(1,0);
		\draw [line width=2,->] (3,1)--++(0,1);
		\draw [line width=2,->] (3,-1)--++(0,1);
		\draw [line width=2,->] (3,0)--++(1,0);
		\draw [line width=2,->] (4,0)--++(0,1);
		\draw [line width=2,->] (4,1)--++(1,0);
		\draw [line width=2,->] (5,1)--++(1,0);
		\draw [line width=2,->] (6,1)--++(0,1);
		\draw [line width=2,->] (5,-1)--++(0,1);
		\draw [line width=2,->] (5,0)--++(1,0);
	\end{tikzpicture}
	\caption{Performing randomized Yang-Baxter moves
		to sample $\nu$ given $\kappa$ under
		$\mathsf{U}_{v,u}^{\mathrm{fwd}}$ (dragging the cross to the right)
		or $\kappa$ given $\nu$ under
		$\mathsf{U}_{v,u}^{\mathrm{bwd}}$ (the cross is dragged the left).
	}
	\label{fig:U_transition_probabilities}
	%\end{noindent}
\end{figure}

\begin{lemma}
	\label{lemma:cross_far_to_the_right}
	As $r\to+\infty$, the state of the cross vertex stabilizes at
	$\oiYBio{.25}{3.5}{10.5pt}$\,.
\end{lemma}
\begin{proof}
	Once the cross vertex passes to the right of $\lambda_1$ it can only
	be in one of two states, $\oiYBio{.25}{3.5}{10.5pt}$ or
	$\ioYBio{.25}{3.5}{10.5pt}$\,, since the boundary conditions far to the right
	are $\scalebox{.6}{
			\begin{tikzpicture} [scale=1.5, ultra thick,
					baseline=6pt]
				\draw[->] (-.5,0.1)--++(.4,0);
				\draw[dotted] (-.5,.4)--++(.4,0);
			\end{tikzpicture}
		}$. From the table in
	\Cref{fig:fwd_YB} we see that $P_{v,u}^{\mathrm{fwd}}\left(
		\oiYBio{.25}{3.5}{10.5pt},\oiYBio{.25}{3.5}{10.5pt} \right)=1$. Moreover,
	since there are no vertical arrows to the right of $\lambda_1$, we have
	$P_{v,u}^{\mathrm{fwd}}\left(
		\ioYBio{.25}{3.5}{10.5pt},\oiYBio{.25}{3.5}{10.5pt} \right)   =
		1-\frac{(u-s)(1-sv)}{(v-s)(1-su)}$, which is strictly positive by
	\eqref{condition_on_convergence}. Therefore, the state
	$\ioYBio{.25}{3.5}{10.5pt}$ of the cross vertex eventually turns into
	$\oiYBio{.25}{3.5}{10.5pt}$ with probability 1 (which in fact corresponds to
	choosing $\nu_1$ somewhere to the right of $\lambda_1$), and the latter state
	is preserved forever.
\end{proof}

We see that the process of (randomized) dragging of the cross vertex to the
right essentially terminates. Cutting cross vertex $\oiYBio{.25}{3.5}{10.5pt}$
which has stabilized far on the right, we obtain the final two-layer arrow
configuration which looks as in \Cref{fig:skew_Cauchy}, right. That is, the
boundary conditions are now $\scalebox{.6}{
		\begin{tikzpicture} [scale=1.5, ultra thick, baseline=6pt]
			\draw[->] (-.5,0.1)--++(.4,0);
			\draw[->] (-.5,.4)--++(.4,0);
		\end{tikzpicture}
	}$ on the far left and $\scalebox{.6}{
		\begin{tikzpicture} [scale=1.5, ultra thick, baseline=6pt]
			\draw[dotted] (-.5,0.1)--++(.4,0);
			\draw[->] (-.5,.4)--++(.4,0);
		\end{tikzpicture}
	}$ on the far right, while the fixed signature
$\kappa\in\mathsf{Sign}_N$ in the middle has been replaced by a \emph{random}
signature $\nu\in \mathsf{Sign}_{N+1}$. Moreover, this new signature satisfies
$\lambda\prec
	\nu\mathop{\dot\succ}\mu$ because in the final two-layer configuration
there can be at most one horizontal arrow per edge.

\begin{definition}
	\label{def:Ufwd}
	The law of the random signature $\nu\in \mathsf{Sign}_{N+1}$ described
	above will be denoted by $\mathsf{U}_{v,u}^{\mathrm{fwd}}(\kappa\to\nu\mid
		\lambda,\mu)$. We will call $\mathsf{U}_{v,u}^{\mathrm{fwd}}$ the
	\emph{forward transition probabilities} (\textit{on signatures}).
\end{definition}

The backward transition probabilities $\mathsf{U}_{v,u}^{\mathrm{bwd}}
	\left(
	\nu\to \kappa\mid \lambda,\mu \right)$ (where the signatures
$\lambda\in\mathsf{Sign}_N$, $\nu,\mu\in\mathsf{Sign}_{N+1}$ with
$\lambda\prec
	\nu\mathop{\dot\succ}\mu$ are given) are defined in a similar way, but
now the cross vertex $\oiYBio{.25}{3.5}{10.5pt}$ is added to the right of
$\nu_1$ and is dragged to the left using the backward Yang-Baxter moves having
probabilities $P_{v,u}^{\mathrm{bwd}}$ given in \Cref{fig:bwd_YB}. The process
of dragging the cross vertex to the left terminates at $\mu_{N+1}$ when the
cross vertex has the state $\iiYBii{.25}{3.5}{10.5pt}$\,. This process does
not affect the signatures $\lambda$ and $\mu$, and turns the fixed signature
$\nu\in \mathsf{Sign}_{N+1}$ in the middle into a \emph{random} signature
$\kappa\in \mathsf{Sign}_N$.

\begin{definition}
	\label{def:Ubwd}
	The law of the random signature $\kappa\in \mathsf{Sign}_{N}$ just
	described will be denoted by $\mathsf{U}_{v,u}^{\mathrm{bwd}}(\nu\to\kappa\mid
		\lambda,\mu)$. We will call $\mathsf{U}_{v,u}^{\mathrm{bwd}}$ the
	\emph{backward transition probabilities} (\emph{on signatures}).
\end{definition}

Clearly, by the very construction,
\begin{equation}
	\label{transition_probabilities_U_sum_to_one}
	\begin{split}
		\sum_{\nu\in \mathsf{Sign}_{N+1}}
		\mathsf{U}_{v,u}^{\mathrm{fwd}}
		\left(
		\kappa\to \nu\mid \lambda,\mu \right)&=1,\qquad
		\textnormal{for every $\lambda,\kappa,\mu$ with
			$\lambda\mathop{\dot\succ}\kappa\prec \mu$};\\
		\sum_{\kappa\in \mathsf{Sign}_{N}}
		\mathsf{U}_{v,u}^{\mathrm{bwd}}
		\left(
		\nu\to \kappa\mid \lambda,\mu \right)&=1,\qquad
		\textnormal{for every $\lambda,\nu,\mu$ with $\lambda\prec
				\nu\mathop{\dot\succ}\mu$}.
	\end{split}
\end{equation}
The first of these sums is infinite and converges due to
\eqref{weights_nonnegativity_region_v_u}. The second of the sums is finite.

\begin{proposition}
	\label{prop:U_are_products_of_P}
	Let $N\in \mathbb{Z}_{\ge0}$, $\lambda,\kappa\in\mathsf{Sign}_N$, and
	$\mu,\nu\in\mathsf{Sign}_{N+1}$ be fixed. The forward transition probability
	$\mathsf{U}_{v,u}^{\mathrm{fwd}}(\kappa\to\nu\mid \lambda,\mu)$ on signatures
	is equal to the product of finitely many local forward transition
	probabilities $P_{v,u}^{\mathrm{fwd}}$ over columns with numbers from
	$\mu_{N+1}$ to $\nu_1$. Similarly,
	$\mathsf{U}_{v,u}^{\mathrm{bwd}}(\nu\to\kappa\mid \lambda,\mu)$ is the product
	of finitely many local backward transition probabilities
	$P_{v,u}^{\mathrm{bwd}}$ over columns from $\mu_{N+1}$ to $\nu_1$.
\end{proposition}
\begin{proof}
	We argue only about forward transition probabilities, the case of the
	backward ones is analogous. Let the multiplicative notations of the signatures
	$\lambda,\kappa,\nu,\mu$ be
	$\lambda=\ldots(-1)^{\ell_{-1}}0^{\ell_0}1^{\ell_1}2^{\ell_2}\ldots  $,
	$\kappa=\ldots(-1)^{k_{-1}}0^{k_0}\ldots  $,
	$\nu=\ldots(-1)^{n_{-1}}0^{n_0}\ldots  $, and
	$\mu=\ldots(-1)^{m_{-1}}0^{m_0}\ldots  $. Consider the situation in the
	definition of $\mathsf{U}_{v,u}^{\mathrm{fwd}}$ when the cross vertex is moved
	through the column number $r$ (for example, $r=0$ in
	\Cref{fig:U_transition_probabilities}). Assume that the following data is
	known before the move of the cross vertex:
	\begin{itemize}
		\item
		      The state of the cross vertex (i.e., one of six states
		      as in \eqref{cross_vertex_weights});
		\item
		      The numbers $\ell_r,k_r,m_r$ of vertical arrows at the
		      $r$-th column before the move of the cross vertex;
		\item
		      The numbers $\ell_r,n_r,m_r$ of vertical arrows at the
		      $r$-th column after the move of the cross vertex;
		\item
		      The numbers of horizontal arrows in both layers of the
		      arrow configuration as in \Cref{fig:U_transition_probabilities} between the
		      $(r-1)$-st column and the cross vertex, as well as between the $r$-th and the
		      $(r+1)$-st columns.
	\end{itemize}
	One readily sees that the state of the cross vertex after the forward
	randomized Yang-Baxter move (placing the cross vertex one step to the right)
	is completely determined by the above data.

	The state of the cross vertex and all the above data at the far left
	is known. Therefore, by induction all the intermediate states of the cross
	vertex in the definition of $\mathsf{U}_{v,u}^{\mathrm{fwd}}(\kappa\to\nu\mid
		\lambda,\mu)$ are completely determined by the four signatures
	$\lambda,\kappa,\nu,\mu$. This implies that the transition probability
	$\mathsf{U}_{v,u}^{\mathrm{fwd}}(\kappa\to\nu\mid \lambda,\mu)$ on signatures
	is indeed equal to the product of the local transition probabilities
	$P_{v,u}^{\mathrm{fwd}}$ depending on these intermediate cross vertex states.
	This completes the proof.
\end{proof}

\subsection{Bijective proof of the skew Cauchy identity}
\label{sub:bijective_proof_skew_Cauchy}

The key observation leading to our bijective proof of
\Cref{thm:skew_Cauchy_one} is the following
\begin{proposition}[Reversibility on signatures]
	\label{prop:reversibility_on_signatures}
	Fix arbitrary $N\in \mathbb{Z}_{\ge0}$,
	$\lambda,\kappa\in\mathsf{Sign}_N$, and $\mu,\nu\in\mathsf{Sign}_{N+1}$. We
	have for any $(u,v)$ satisfying \eqref{weights_nonnegativity_region_v_u}:
	\begin{equation}
		\label{reversibility_on_signatures}
		\left[ \iiYBii{.35}{3}{14.5pt} \right]_{v,u}
		G_{\lambda/\kappa}^{c}(v^{-1})F_{\mu/\kappa}(u)
		\mathsf{U}_{v,u}^{\mathrm{fwd}}(\kappa\to\nu\mid \lambda,\mu) = \left[
			\oiYBio{.35}{3}{14.5pt} \right]_{v,u}
		F_{\nu/\lambda}(u)\,G_{\nu/\mu}^{c}(v^{-1})
		\mathsf{U}_{v,u}^{\mathrm{bwd}}(\nu\to\kappa\mid \lambda,\mu),
	\end{equation}
	where the weights of the cross vertices are given in
	\eqref{cross_vertex_weights} (modulo the swap, cf. \Cref{rmk:swap_u_v}).
\end{proposition}
\begin{remark}
	\label{rmk:reversibility_on_signatures}
	Both sides of \eqref{reversibility_on_signatures} are nonzero only if
	$\lambda\mathop{\dot\succ}\kappa\prec \mu$ and $\lambda\prec
		\nu\mathop{\dot\succ}\mu$. Indeed, if, say, the condition $\kappa\prec \mu$ is
	violated, then $F_{\mu/\kappa}(u)$ is zero by the very definition. At the same
	time $\mathsf{U}_{v,u}^{\mathrm{bwd}}(\nu\to\kappa\mid \lambda,\mu)$ also
	vanishes because $\kappa\not\prec\mu$ implies that $\kappa$ cannot arise as
	the middle signature in the two-layer arrow configuration after dragging the
	cross vertex from far right to the left.
\end{remark}
\begin{proof}[Proof of \Cref{prop:reversibility_on_signatures}]
	By \eqref{F_skew_one_variable_definition},
	\eqref{G_skew_one_variable_definition} and \Cref{prop:U_are_products_of_P},
	the skew functions $F, G^c$ as well as the transition probabilities
	$\mathsf{U}_{v,u}$ in both sides of \eqref{reversibility_on_signatures} can be
	expressed as products over the columns in the two-layer arrow configurations
	as in \Cref{fig:skew_Cauchy}. The desired identity
	\eqref{reversibility_on_signatures} then follows by repeatedly applying the
	local reversibility condition at each column for the probabilities of the
	Yang-Baxter moves $P_{v,u}^{\mathrm{fwd}}$ and $P_{v,u}^{\mathrm{bwd}}$. The
	local reversibility condition is satisfied by the very construction of the
	latter probabilities, see \Cref{def:bijectivisation} and
	\Cref{sub:YB_bijectivisation_new_label}. The quantities $\left[
			\iiYBii{.25}{3.5}{10.5pt} \right]_{v,u}
		G_{\lambda/\kappa}^{c}(v^{-1})F_{\mu/\kappa}(u) $ and $\left[
			\oiYBio{.25}{3.5}{10.5pt} \right]_{v,u}
		F_{\nu/\lambda}(u)\,G_{\nu/\mu}^{c}(v^{-1})$ collect the weights entering the
	local reversibility conditions, while the probabilities
	$\mathsf{U}_{v,u}^{\mathrm{fwd}},\mathsf{U}_{v,u}^{\mathrm{bwd}}$ collect the
	local probabilities $P_{v,u}^{\mathrm{fwd}},P_{v,u}^{\mathrm{bwd}}$. This
	implies \eqref{reversibility_on_signatures}.
\end{proof}

\begin{proof}[Proof of \Cref{thm:skew_Cauchy_one}]
	Summing \eqref{reversibility_on_signatures} over both
	$\kappa\in\mathsf{Sign}_{N}$ and $\nu\in \mathsf{Sign}_{N+1}$ and recalling
	that $\left[\iiYBii{.25}{3.5}{10.5pt} \right]_{v,u}=1$ and $\left[
			\oiYBio{.25}{3.5}{10.5pt} \right]_{v,u}=\frac{v-u}{v-tu}$, we have
	\begin{multline}
		\label{skew_Cauchy_identity_bijective_proof}
		\sum_{\kappa\in \mathsf{Sign}_N}
		G_{\lambda/\kappa}^{c}(v^{-1})F_{\mu/\kappa}(u)
		\Bigg(\sum_{\nu\in \mathsf{Sign}_{N+1}}
		\mathsf{U}_{v,u}^{\mathrm{fwd}}(\kappa\to\nu\mid \lambda,\mu) \Bigg) \\ =
		\frac{v-u}{v-tu} \sum_{\nu\in \mathsf{Sign}_{N+1}}
		F_{\nu/\lambda}(u)\,G_{\nu/\mu}^{c}(v^{-1}) \Bigg( \sum_{\kappa\in
			\mathsf{Sign}_N} \mathsf{U}_{v,u}^{\mathrm{bwd}}(\nu\to\kappa\mid \lambda,\mu)
		\Bigg).
	\end{multline}
	By \eqref{transition_probabilities_U_sum_to_one}, the sums in the
	parentheses in both sides are equal to $1$, which implies the desired identity
	\eqref{skew_Cauchy_identity}.
\end{proof}

We call the above proof of the skew Cauchy identity
\eqref{skew_Cauchy_identity} \emph{bijective} because
\eqref{skew_Cauchy_identity_bijective_proof} provides a \emph{refinement} of
\eqref{skew_Cauchy_identity}
(involving summation over $\kappa,\nu$ in both sides), in which the terms in
both sides are bijectively identified with each other with the help of the
reversibility condition \eqref{reversibility_on_signatures}. Thus, the
transition probabilities $\mathsf{U}_{v,u}^{\mathrm{fwd}}$ and
$\mathsf{U}_{v,u}^{\mathrm{bwd}}$ show how to split terms in both sides of the
original identity \eqref{skew_Cauchy_identity} into smaller ones, such that
these smaller terms are identified with each other.

\subsection{Markov projection of the forward transition onto first columns}
\label{sub:properties_of_global_transitions}

For notational convenience, in this subsection we assume that both $\lambda$
and $\mu$ are nonnegative signatures (i.e., whose parts are all nonnegative).
Then the signatures $\kappa,\nu$ entering
$\mathsf{U}_{v,u}^{\mathrm{fwd}}(\kappa\to \nu\mid
	\lambda,\mu)$ (as well as
$\mathsf{U}_{v,u}^{\mathrm{bwd}}(\nu\to\kappa\mid
	\lambda,\mu)$) should also be nonnegative, otherwise these transition
probabilities vanish for interlacing reasons. Fix any $h\in
	\mathbb{Z}_{\ge1}$. For any nonnegative signature $\rho$ having multiplicative
notation $\rho=0^{r_0}1^{r_1}2^{r_2}\ldots $, let
$\rho^{[<h]}:=(r_0,r_1,\ldots,r_{h-1} )\in \mathbb{Z}_{\ge0}^{h}$ and
$\rho^{[\ge h]}:=(r_h,r_{h+1},\ldots )$ be the corresponding configurations of
arrows in the first $h$ columns and in the rest of the nonnegative integer
lattice. Using the fact that the forward transition probabilities
$\mathsf{U}^{\mathrm{fwd}}_{v,u}$ were defined in
\Cref{sub:definition_of_three_signature_probabilities} in a sequential way
(from left to right columns), we can express them as follows (for every fixed
$h\ge1$):
\begin{multline}
	\label{U_fwd_Markov_representation}
	\mathsf{U}_{v,u}^{\mathrm{fwd}}\left( \kappa\to\nu\mid \lambda,\mu
	\right)
	=
	\mathsf{U}^{[<h],\mathrm{fwd}}_{v,u}(\kappa^{[<h]}\to\nu^{[<h]}\mid
	\lambda^{[<h]},\mu^{[<h]})
	\\\times
	\mathsf{U}^{[\ge h],\mathrm{fwd}}_{v,u}
	(\kappa^{[\ge h]}\to\nu^{[\ge h]}\mid
	\lambda,\mu,\kappa^{[<h]},\nu^{[<h]}).
\end{multline}
A crucial property in \eqref{U_fwd_Markov_representation} is that
$\mathsf{U}^{[<h],\mathrm{fwd}}_{v,u}$, the transition probability describing
the evolution of the first $h$ columns, does not depend on configurations of
arrows the in columns $h,h+1,\ldots $. In other words, in the transition
$\kappa\to\nu$ under $\mathsf{U}_{v,u}^{\mathrm{fwd}}$, the first $h$ columns
are (randomly) transformed in a \emph{marginally Markovian way}. We will say
that the forward transition probabilities $\mathsf{U}_{v,u}^{\mathrm{fwd}}$ on
signatures \emph{admit Markov projections onto the first $h$ columns} for
every $h\ge1$. In \eqref{U_fwd_Markov_representation} this Markov projection
is denoted by $\mathsf{U}^{[<h],\mathrm{fwd}}_{v,u}$.

Representation \eqref{U_fwd_Markov_representation} is possible because the
forward transition probabilities are defined via dragging the cross vertex
from left to right. A similar representation for the backward transition
probabilities based on their definition via dragging the cross vertex from
right to left would show that in the transition $\nu\to\kappa$ under
$\mathsf{U}_{v,u}^{\mathrm{bwd}}$ the columns $h,h+1,\ldots $ evolve in a
marginally Markovian way. Since this Markov projection of the backward
probabilities always involves infinitely many columns, we will not focus on
this right-to-left Markov property in the present paper.

Let us now consider the case $h=1$. For shorter notation in this case we will
write $[0]$ instead of $[<1]$ in the superscripts. Let us write down the
Markov projection $\mathsf{U}_{v,u}^{[0],\mathrm{fwd}}$ of
$\mathsf{U}_{v,u}^{\mathrm{fwd}}$ onto the column number $0$. In this case the
quantity $\rho^{[0]}$ for any nonnegative signature $\rho$ is simply the
number of zero parts in $\rho$. There are six possible types of transitions in
the first column which can be read off the last row of the table in
\Cref{fig:fwd_YB} (recall that we swap the parameters
$u$ and $v$, cf. \Cref{rmk:swap_u_v}):
\begin{equation}
	%\begin{noindent}
	\label{U_0_dynamic_S6V_transitions}
	\begin{split}
		&
		\mathsf{U}_{v,u}^{[0]}(g\to g\mid g-1,g+1)=1
		,
		\\&
		\mathsf{U}_{v,u}^{[0]}(g\to g\mid g,g+1)=
				\dfrac{(1-t)v}{v-t
				u}\dfrac{u-st^g}{v-st^g}
		,\qquad
		\mathsf{U}_{v,u}^{[0]}(g\to g+1\mid g,g+1)=
			\dfrac{v-u}{v-tu}
			\dfrac{v-st^{g+1}}{v-st^g}
		,
		\\&
		\mathsf{U}_{v,u}^{[0]}(g\to g\mid g-1,g)=
			\dfrac{(1-t)u}{v-tu}
			\dfrac{v-st^{g}}{u-st^{g}}
		,\qquad
		\mathsf{U}_{v,u}^{[0]}(g\to g-1\mid g-1,g)=
			\dfrac{t(v-u)}{v-tu}
			\dfrac{u-st^{g-1}}{u-st^{g}}
		,
		\\&
		\mathsf{U}_{v,u}^{[0]}(g\to g\mid g,g)=1
		.
	\end{split}
	%\end{noindent}
\end{equation}
These transitions depend on arbitrary $g\in \mathbb{Z}_{\ge0}$ with the
understanding that $g\ge1$ in the first and the third lines in
\eqref{U_0_dynamic_S6V_transitions}.

\section{Yang-Baxter field}
\label{sec:YB_field}

In this section we introduce our main stochastic object, the \emph{spin
	Hall-Littlewood Yang-Baxter random field}
(called simply the \emph{Yang-Baxter field} throughout the paper), and discuss
its main properties.

\subsection{Spin Hall-Littlewood measures and processes}
\label{sub:spin_HL_measures_processes}

Fix $(x,y)\in \mathbb{Z}_{\ge0}^2$, and let $v_1,\ldots,v_x $ and $u_1,\ldots
	,u_y$ be spectral parameters such that $0\le u_i<v_j<1$ for all $i,j$. As in
\Cref{sec:local_transition_probabilities}, we continue to assume that $0\le
	t<1$ and $-1<s\le0$. Define the following probability measure on the set of
nonnegative signatures of length $y$:
\begin{equation}
	\label{spin_HL_measure}
	\mathscr{H}_{x,y}(\lambda):=
	\frac{1}{\Pi_{x,y}}\,
	G^c_{\lambda/(0^y)}(v_1^{-1},\ldots,v_x^{-1}
	)F_{\lambda/\varnothing}(u_1,\ldots,u_y ) ,\qquad \lambda\in
	\mathsf{Sign}^{+}_{y}.
\end{equation}
The weights under $\mathscr{H}_{x,y}$ are nonnegative and their sum over
$\lambda\in
	\mathsf{Sign}^{+}_y$ converges thanks to our conditions on parameters.
The normalization constant in \eqref{spin_HL_measure} has the following
product form due to the Cauchy identity of \Cref{cor:nonskew_multi_Cauchy}:
\begin{equation}
	\label{spin_HL_normalization}
	\Pi_{x,y}
	=
	\prod_{i=1}^{y}
	\biggl(\frac{1-t^i}{1-su_i}
	\prod_{j=1}^{x}\frac{v_j-tu_i}{v_j-u_i}
	\biggr).
\end{equation}
We call the measures $\mathscr{H}_{x,y}$ \eqref{spin_HL_measure} the
\emph{spin Hall-Littlewood measures} by analogy with the Macdonald measures
\cite{BorodinCorwin2011Macdonald}
(and their several degenerations, most notably, the
Schur measures \cite{okounkov2001infinite}). As in the Macdonald setting, skew
Cauchy identities allow to extend the measures \eqref{spin_HL_measure} to spin
Hall-Littlewood processes which are probability measures on certain sequences
of nonnegative signatures. For simplicity, we will only consider a particular
case of spin Hall-Littlewood processes suitable for our needs.

Fix $k\in \mathbb{Z}_{\ge1}$ and sequences
\begin{equation}
	\label{spin_HL_down_right_path_sequences}
	\vec{x}:=(0=x_1\le x_2\le \ldots\le x_k),\qquad
	\vec{y}:=(y_1\ge y_2\ge \ldots\ge y_{k-1}\ge y_k=0).
\end{equation}
Consider the following down-right path in $\mathbb{Z}_{\ge0}^2$ corresponding
to these sequences:
\begin{equation}
	\label{spin_HL_down_right_path}
	\mathcal{P}_{\vec{x},\vec{y}}:=
	\left\{
	(x_1,y_1),(x_2,y_1),(x_2,y_2),(x_3,y_2),\ldots,(x_k,y_{k-1}),(x_k,y_k)
	\right\}.
\end{equation}
Let $v_1,\ldots,v_{x_k}$ and $u_1,\ldots,u_{y_1} $ be spectral parameters
satisfying the same conditions as for the measures \eqref{spin_HL_measure}.
The \emph{spin Hall-Littlewood process} $\mathscr{HP}_{\vec{x},\vec{y}}$
indexed by the down-right path $\mathcal{P}_{\vec{x},\vec{y}}$ depending on
these spectral parameters is a probability measure on sequences of nonnegative
signatures $\lambda^p$, $p\in \mathcal{P}_{\vec{x},\vec{y}}$, with
$\lambda^{(0,y_1)}=(0^{y_1})$ and $\lambda^{(x_k,0)}=\varnothing$, defined as
\begin{multline}
	\label{spin_HL_process}
	\mathscr{HP}_{\vec{x},\vec{y}}
	(\lambda^p\colon p\in \mathcal{P}_{\vec{x},\vec{y}})
	\\:=\frac{1}{\Pi_{\vec{x},\vec{y}}}\,
	\prod_{i=1}^{k-1}
	G^c_{\lambda^{(x_{i+1},y_i)}/\lambda^{(x_i,y_i)}}
	(v_{x_i+1}^{-1},\ldots,v_{x_{i+1}}^{-1} ) \prod_{i=2}^{k}
	F_{\lambda^{(x_{i},y_{i-1})}/\lambda^{(x_{i},y_{i})}}(u_{y_{i}+1},\ldots,u_{y_
		{i-1}} ).
\end{multline}
Here $\lambda^{(x,y)}\in \mathsf{Sign}_y^+$, and the normalization constant in
\eqref{spin_HL_process} can be read off the skew Cauchy identities (see
\Cref{sub:Cauchy_identity}):
\begin{equation}
	\label{spin_HL_process_normalization}
	\Pi_{\vec{x},\vec{y}}=
	\Biggl(\prod_{i=1}^{y}
	\frac{1-t^i}{1-su_i}
	\Biggr)
	\prod_{\substack{(i,j)\in \mathbb{Z}_{\ge1}^{2}\colon
	\\\text{box $(i,j)$ is below $\mathcal{P}_{\vec{x},\vec{y}}$}}}
	\frac{v_j-tu_i}{v_j-u_i}.
\end{equation}
A graphical illustration of a spin Hall-Littlewood process is given in
\Cref{fig:spin_HL_process}.

\begin{figure}[htpb]
	%\begin{noindent}
	\centering
	\begin{tikzpicture}
	[scale=1, very thick]
	\draw[->] (-.5,0)--++(7,0);
	\draw[->] (0,-.5)--++(0,4);
		\foreach \ii in {1,2,3,4,5,6}
		{
			\node at (\ii-.5,-.4) {$v_\ii$};
			\draw[dotted, thick] (\ii,-.5)--++(0,3.75);
		}
		\foreach \jj in {1,2,3}
		{
			\node at (-.4,\jj-.5) {$u_\jj$};
			\draw[dotted, thick] (-.5,\jj)--++(6.75,0);
		}
		\draw [red, line width=2] (0,3)--++(3,0)--++(0,-1)--++(1,0)--++(0,-1)--++(2,0)--++(0,-1);
		\foreach \p in {(0,3),(3,3),(3,2),(4,2),(4,1),(6,1),(6,0)}
		{
			\draw[red,fill] \p circle(3pt);
		}
		\node at (0,-.9) {$x_1$};
		\node at (3,-.9) {$x_2$};
		\node at (4,-.9) {$x_3$};
		\node at (6,-.9) {$x_4$};
		\node at (-1,0) {$y_4$};
		\node at (-1,1) {$y_3$};
		\node at (-1,2) {$y_2$};
		\node at (-1,3) {$y_1$};
		\node at (1.5,3.35) {$G^c$};
		\node at (3.7,2.35) {$G^c$};
		\node at (5.4,1.35) {$G^c$};
		\node at (2.7,2.5) {$F$};
		\node at (3.7,1.5) {$F$};
		\node at (5.7,0.5) {$F$};
		\node at (0,3) (p1) {};
		\node at (3,3) (p2) {};
		\node at (3,2) (p3) {};
		\node at (4,2) (p4) {};
		\node at (4,1) (p5) {};
		\node at (6,1) (p6) {};
		\node at (6,0) (p7) {};
		\node[anchor=east] (lp1) at (-.5,4) {$\lambda^{(0,3)}=(0^{3})$};
		\node (lp2) at (3.6,3.3) {$\lambda^{(3,3)}$};
		\node (lp3) at (2.6,1.7) {$\lambda^{(3,2)}$};
		\node (lp4) at (4.6,2.3) {$\lambda^{(4,2)}$};
		\node (lp5) at (4.6,1.3) {$\lambda^{(4,1)}$};
		\node[anchor=west] (lp6) at (6.5,2) {$\lambda^{(6,1)}$};
		\node[anchor=west] (lp7) at (6.5,1) {$\lambda^{(6,0)}=\varnothing$};
		\draw[line width=.7,dashed] (p1)--(lp1);
		\draw[line width=.7,dashed] (p6)--(lp6);
		\draw[line width=.7,dashed] (p7)--(lp7);
	\end{tikzpicture}
	\caption{An illustration of the spin Hall-Littlewood process
	indexed by sequences $\vec{x}=(0,3,4,6)$ and $\vec{y}=(3,2,1,0)$.
	The second product (over in $(i,j)$) in \eqref{spin_HL_process_normalization}
	runs over all boxes inside the region bounded by the
	down-right path $\mathcal{P}_{\vec{x},\vec{y}}$.
	For this particular path the product contains 13 terms.}
	\label{fig:spin_HL_process}
	%\end{noindent}
\end{figure}
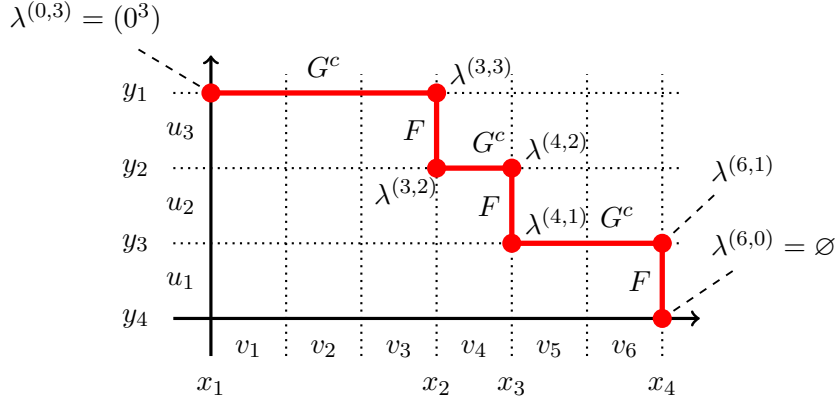

One of the properties of spin Hall-Littlewood processes is that the marginal
distribution of each single signature $\lambda^{(x,y)}\in\mathsf{Sign}^+_y$,
$(x,y)\in \mathcal{P}_{\vec{x},\vec{y}}$, under
$\mathscr{HP}_{\vec{x},\vec{y}}$ \eqref{spin_HL_process} is given by the spin
Hall-Littlewood measure $\mathscr{H}_{x,y}$ \eqref{spin_HL_measure}. More
generally, take any subpath $\mathcal{Q}$ of $\mathcal{P}_{\vec{x},\vec{y}}$
such that $\mathcal{Q}$ is itself a down-right path. Then the marginal
distribution of the signatures $\{\lambda^{q}\colon q\in \mathcal{Q}\}$ under
the original spin Hall-Littlewood process $\mathscr{HP}_{\vec{x},\vec{y}}$
\eqref{spin_HL_process} is itself a spin Hall-Littlewood process corresponding
to the path $\mathcal{Q}$.

\subsection{Yang-Baxter field}
\label{sub:YB_random_field}

Let us now introduce the Yang-Baxter field with the help of the forward
transition probabilities on signatures discussed in
\Cref{sec:local_transition_probabilities}. The field depends on $t\in[0,1)$,
		$s\in(-1,0]$, and two sequences of spectral parameters $v_1,v_2,\ldots $,
$u_1,u_2,\ldots $ such that $0\le u_i<v_j<1$ for all $i,j$. The Yang-Baxter
field is a probability distribution on the space of nonnegative signatures
$\lambda^{(x,y)}$ indexed by points of the quadrant $(x,y)\in
	\mathbb{Z}_{\ge0}^2$ such that $\lambda^{(x,y)}\in
	\mathsf{Sign}_{y}^{+}$, the signatures interlace as (see
\Cref{sub:signatures} for notation)
\begin{equation*}
	\lambda^{(x,y)}\prec\lambda^{(x,y+1)},\qquad
	\lambda^{(x,y)}\mathop{\dot\prec}\lambda^{(x+1,y)},\qquad
	(x,y)\in \mathbb{Z}_{\ge0}^2,
\end{equation*}
and satisfy the boundary conditions $\lambda^{(x,0)}\equiv \varnothing$,
$\lambda^{(0,y)}=(0^y)$.

\begin{definition}
	\label{def:YB_field}
	We construct the \emph{Yang-Baxter field} $\boldsymbol	\Lambda:=\{
		\boldsymbol\lambda^{(x,y)} \}_{x,y\ge0}$ inductively. Initialize the boundary
	values in the following nonrandom way:
	$\boldsymbol\lambda^{(x,0)}=\varnothing$, $\boldsymbol\lambda^{(0,y)}=(0^y)$
	for all $x,y\ge0$. Now, for some $n\ge1$, let the field be already defined for
	all $(x',y')\in \mathbb{Z}_{\ge0}^{2}$ such that $x'+y'\le n$. Conditioned on
	$\{\boldsymbol\lambda^{(x',y')}\}_{x'+y'\le n}$, independently sample the
	random signatures $\boldsymbol\lambda^{(x,y)}$ with $x+y=n+1$, $x,y\ge1$,
	according to
	\begin{equation}
		\label{YB_field_definition_forward}
		\mathrm{Prob}\bigl( \boldsymbol\lambda^{(x,y)}=\nu
		\mid \{\boldsymbol\lambda^{(x',y')}\}_{x'+y'\le n} \bigr)=
		\mathsf{U}^{\mathrm{fwd}}_{v_x,u_y}
		\bigl( \boldsymbol\lambda^{(x-1,y-1)}\to \nu
		\mid \boldsymbol\lambda^{(x,y-1)},\boldsymbol\lambda^{(x-1,y)}
		\bigr).
	\end{equation}
	This defines the Yang-Baxter field for $(x,y)$ with $x+y\le n+1$, and
	the induction step completes the definition of the field for all $(x,y)\in
		\mathbb{Z}_{\ge0}^{2}$. See \Cref{fig:YB_field} for an illustration.
\end{definition}

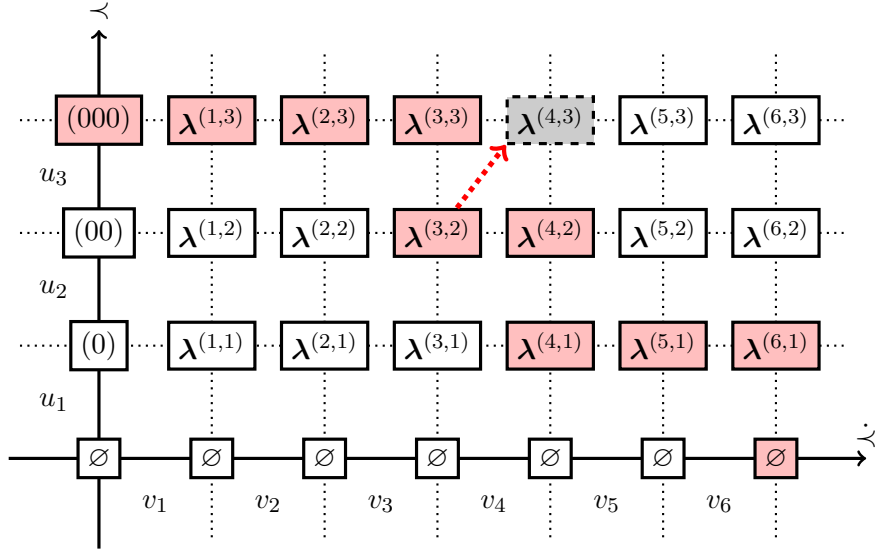
\begin{figure}[htpb]
	%\begin{noindent}
	\centering
	\begin{tikzpicture}
		[scale=1.5, very thick]
		\draw[->] (-.8,0)--++(7.6,0) node[above] {$\mathop{\dot\prec}$};
		\draw[->] (0,-.8)--++(0,4.6) node[above] {$\prec$};
		\foreach \ii in {1,2,3,4,5,6}
		{
			\node at (\ii-.5,-.4) {$v_\ii$};
			\draw[dotted, thick] (\ii,-.7)--++(0,4.3);
		}
		\foreach \jj in {1,2,3}
		{
			\node at (-.4,\jj-.5) {$u_\jj$};
			\draw[dotted, thick] (-.7,\jj)--++(7.3,0);
		}
		\foreach \ii in {0,...,5}
		{\node[rectangle,draw,fill=white] at (\ii,0) {$\varnothing$};}
		\node[rectangle,draw,fill=white] at (0,1) {$(0)$};
		\node[rectangle,draw,fill=white] at (0,2) {$(00)$};
		\node[rectangle,draw,fill=pink] at (0,3) {$(000)$};
		\node[rectangle,draw,fill=pink] at (6,0) {$\varnothing$};
		\node[rectangle,draw,fill=white] at (1,1) {$\boldsymbol\lambda^{(1,1)}$};
		\node[rectangle,draw,fill=white] at (2,1) {$\boldsymbol\lambda^{(2,1)}$};
		\node[rectangle,draw,fill=white] at (3,1) {$\boldsymbol\lambda^{(3,1)}$};
		\node[rectangle,draw,fill=pink] at (4,1) {$\boldsymbol\lambda^{(4,1)}$};
		\node[rectangle,draw,fill=pink] at (5,1) {$\boldsymbol\lambda^{(5,1)}$};
		\node[rectangle,draw,fill=pink] at (6,1) {$\boldsymbol\lambda^{(6,1)}$};
		\node[rectangle,draw,fill=white] at (1,2) {$\boldsymbol\lambda^{(1,2)}$};
		\node[rectangle,draw,fill=white] at (2,2) {$\boldsymbol\lambda^{(2,2)}$};
		\node[rectangle,draw,fill=pink] (kappa) at (3,2) {$\boldsymbol\lambda^{(3,2)}$};
		\node[rectangle,draw,fill=pink] at (4,2) {$\boldsymbol\lambda^{(4,2)}$};
		\node[rectangle,draw,fill=white] at (5,2) {$\boldsymbol\lambda^{(5,2)}$};
		\node[rectangle,draw,fill=white] at (6,2) {$\boldsymbol\lambda^{(6,2)}$};
		\node[rectangle,draw,fill=pink] at (1,3) {$\boldsymbol\lambda^{(1,3)}$};
		\node[rectangle,draw,fill=pink] at (2,3) {$\boldsymbol\lambda^{(2,3)}$};
		\node[rectangle,draw,fill=pink] at (3,3) {$\boldsymbol\lambda^{(3,3)}$};
		\node[rectangle,draw,dashed,fill=gray!40!white] (nu) at (4,3) {$\boldsymbol\lambda^{(4,3)}$};
		\node[rectangle,draw,fill=white] at (5,3) {$\boldsymbol\lambda^{(5,3)}$};
		\node[rectangle,draw,fill=white] at (6,3) {$\boldsymbol\lambda^{(6,3)}$};
		\draw[->,dotted,red,line width=2] (kappa)--(nu.south west);
	\end{tikzpicture}
	\caption{Yang-Baxter random field. Signatures along a down-right path
	(the extension of the path in \Cref{fig:spin_HL_process}) are highlighted in red.
	The signature $\boldsymbol\lambda^{(3,2)}$ is replaced by $\boldsymbol\lambda^{(4,3)}$
	in this path with the help of the forward transition probability, cf. the proof of
	\Cref{thm:YB_field_spin_HL_process}.}
	\label{fig:YB_field}
	%\end{noindent}
\end{figure}

The discussion in \Cref{sub:properties_of_global_transitions} readily implies
the following Markov projection property of the Yang-Baxter field:
\begin{proposition}
	\label{prop:YB_Markov_projections}
	Fix any $h\in \mathbb{Z}_{\ge1}$. Under the Yang-Baxter field, the
	first $h$ columns of the signatures $\boldsymbol\lambda^{(x,y)}$ evolve in a
	marginally Markovian way (i.e., independently of the columns $h+1,h+2,\ldots
	$).
\end{proposition}
This evolution of the first $h$ columns defines a random field indexed by
$\mathbb{Z}_{\ge0}^{2}$ with values in $\mathbb{Z}_{\ge0}^{h}$ which can be
regarded as an $h$-layer stochastic vertex model. In
\Cref{sec:dynamicS6V,sec:degenerations} we discuss the case $h=1$ in detail.
Details on the two-layer case for $s=0$ may be found in \cite[Section
	4.4]{BufetovMatveev2017}.

The next theorem states a key property of the Yang-Baxter field
$\boldsymbol\Lambda$:
\begin{theorem}
	\label{thm:YB_field_spin_HL_process}
	Under the Yang-Baxter field, for any down-right path
	$\mathcal{P}_{\vec{x},\vec{y}}$ as in
	\eqref{spin_HL_down_right_path_sequences}--\eqref{spin_HL_down_right_path},
	the joint distribution of the signatures $\{\boldsymbol\lambda^p\colon p\in
		\mathcal{P}_{\vec{x},\vec{y}}\}$ is given by the spin Hall-Littlewood
	process $\mathscr{HP}_{\vec{x},\vec{y}}$ \eqref{spin_HL_process}.
\end{theorem}
\begin{proof}
	Extend the path $\mathcal{P}_{\vec{x},\vec{y}}$ by adding to it all
	the intermediate vertices, so that the distance between each two consecutive
	vertices along the extended path is equal to $1$ (cf. \Cref{fig:YB_field}).
	Let us also add vertices $(0,y_1+1)$ and $(x_k+1,0)$ in the beginning and the
	end of the path, respectively. If we establish the claim for such extended
	paths, then the original claim will follow, cf. the remark in the end of
	\Cref{sub:spin_HL_measures_processes}.

	Using the inductive definition of $\boldsymbol\Lambda$, we establish
	the modified claim by induction on the down-right path. The base of the
	induction is the case when the path goes along the coordinate axes, i.e., has
	the form $\left\{    (0,y_1+1),(0,y_1),\ldots,
		(0,1),(0,0),(1,0),\ldots,(x_k+1,0)\right\}$. In this case the random
	signatures along this path are in fact deterministic, and coincide with the
	corresponding signatures under the spin Hall-Littlewood process corresponding
	to this path.

	In the induction step, we replace one down-right corner of the form
	$\{(x,y+1),(x,y),(x+1,y)\}$ by the right-down corner
	$\{(x,y+1),(x+1,y+1),(x+1,y)\}$ (see an example in \Cref{fig:YB_field} where
	$(x,y)=(3,2)$). Denote the old and the new paths by $\mathcal{P}$ and
	$\mathcal{P}'$, respectively. For shorter notation, set
	\begin{equation*}
		\kappa:=\boldsymbol\lambda^{(x,y)},\qquad
		\mu:=\boldsymbol\lambda^{(x,y+1)},
		\qquad
		\lambda:=\boldsymbol\lambda^{(x+1,y)},
		\qquad
		\nu:=\boldsymbol\lambda^{(x+1,y+1)}.
	\end{equation*}
	Assume that the joint distribution of the signatures along
	$\mathcal{P}$ is given by the corresponding spin Hall-Littlewood process. The
	joint distribution along $\mathcal{P}'$ can be obtained from the joint
	distribution along $\mathcal{P}$ with the help of the conditional distribution
	of $\nu$ given $\lambda,\kappa,\mu$. By \Cref{def:YB_field}, the latter
	conditional distribution is given by the forward transition probability. Thus,
	we see that the joint distribution of all four signatures
	$\lambda,\kappa,\mu,\nu$ is proportional to the left-hand side of
	\eqref{reversibility_on_signatures} (with $u=u_{y+1}$, $v=v_{x+1}$). Using
	this identity and summing over $\kappa$, we see from the right-hand side of
	\eqref{reversibility_on_signatures} that the joint distribution of
	$\lambda,\nu,\mu$ is proportional to
	$F_{\nu/\lambda}(u_{y+1})\,G_{\nu/\mu}^{c}(v^{-1}_{x+1})$, as it should be
	under the spin Hall-Littlewood process corresponding to the path
	$\mathcal{P}'$. This completes the induction step and the proof of the
	proposition.
\end{proof}

\Cref{thm:YB_field_spin_HL_process} and
\Cref{prop:reversibility_on_signatures} readily
imply a backward version of the conditional distribution
\eqref{YB_field_definition_forward} in the Yang-Baxter field:
\begin{corollary}
	\label{cor:YB_field_bwd_conditional_distr}
	Under the Yang-Baxter field, for any $(x,y)\in	\mathbb{Z}_{\ge0}$ the
	conditional distribution of $\boldsymbol\lambda^{(x,y)}$ given the signatures
	to the right and above it is equal to the backward transition probability:
	\begin{equation*}
		\mathrm{Prob}(\boldsymbol\lambda^{(x,y)}=\kappa\mid
		\boldsymbol\lambda^{(x+1,y)},\boldsymbol\lambda^{(x,y+1)},
		\boldsymbol\lambda^{(x+1,y+1)})
		=
		\mathsf{U}^{\mathrm{bwd}}_{v_{x+1},u_{y+1}}
		\bigl(
		\boldsymbol\lambda^{(x+1,y+1)}\to\kappa
		\mid
		\boldsymbol\lambda^{(x+1,y)},
		\boldsymbol\lambda^{(x,y+1)}
		\bigr).
	\end{equation*}
\end{corollary}

\section{A dynamic stochastic six vertex model}
\label{sec:dynamicS6V}

Here we consider the Markov projection of the Yang-Baxter field onto the
column number zero. This produces a new dynamic version of the stochastic six
vertex model. The original stochastic six vertex model was introduced in
\cite{GwaSpohn1992}, and its asymptotic behavior was studied in various
regimes in, e.g., \cite{BCG6V}, \cite{AmolBorodin2016Phase},
\cite{Amol2016Stationary}. We recall this model in \Cref{sub:degen_HL} below.

\subsection{Dynamic vertex weights}
\label{sub:dynamicS6V_subsection}

Let $\boldsymbol\Lambda=\{\boldsymbol\lambda^{(x,y)}\}_{x,y\ge0}$ be the
Yang-Baxter field constructed in \Cref{sec:YB_field}. Recall that each
$\boldsymbol\lambda^{(x,y)}$ is a random nonnegative signature (of length
$y$). For each $(x,y)\in \mathbb{Z}_{\ge0}^2$, let $\boldsymbol\ell^{(x,y)}:=
	(\boldsymbol\lambda^{(x,y)})^{[0]}\in	\mathbb{Z}_{\ge0}$ denote the
number of arrows in the zeroth column of the arrow configuration encoded by
the signature $\boldsymbol\lambda^{(x,y)}$.\footnote{Equivalently,
	$(\boldsymbol\lambda^{(x,y)})^{[0]}$ is the number of zero parts in the
	signature $\boldsymbol\lambda^{(x,y)}$.} Since $\boldsymbol\lambda^{(x,y)}\in
	\mathsf{Sign}_y$, we have $\boldsymbol\ell^{(x,y)}\le y$.
\Cref{prop:YB_Markov_projections} implies that the scalar random field
$\mathbf{L}:=\{\boldsymbol\ell^{(x,y)}\}_{x,y\ge0}$ \emph{does not depend} on
the rest of the Yang-Baxter field (i.e., of the numbers of arrows in
$\boldsymbol\lambda^{(x,y)}$ in columns $\ge1$). In this way we say that
$\mathbf{L}$ is a marginally Markovian projection of the Yang-Baxter field
$\boldsymbol\Lambda$ onto the column number zero.

Let us now present an independent description of $\mathbf{L}$. From the
definition of the Yang-Baxter field via conditional probabilities
\eqref{YB_field_definition_forward} it follows that for each $(x,y)\in
	\mathbb{Z}_{\ge0}^{2}$ the value of $\boldsymbol\ell^{(x+1,y+1)}$ is
randomly determined using $\boldsymbol\ell^{(x+1,y)},\boldsymbol\ell^{(x,y)}$,
and $\boldsymbol\ell^{(x,y+1)}$, and the corresponding conditional
probabilities can be read from \eqref{U_0_dynamic_S6V_transitions}. In the
language of values of the field $\mathbf{L}$ these conditional probabilities
are given in \Cref{fig:L_dynamicS6V_probabilities}. The nature of the six
possible configurations of the values of $\mathbf{L}$ at $2\times 2$ squares
allow to \emph{identify} $\mathbf{L}$ with the height function in a dynamic
version of the stochastic six vertex model. Let us describe this model in more
detail.

\begin{figure}[htpb]
	%\begin{noindent}
	\centering
	\begin{tabular}{c|c|c|c|c|c}
	\scalebox{.9}{\begin{tikzpicture}
		[scale=1.2, very thick]
		\draw[dashed] (0,-.6)--++(0,1.2);
		\draw[dashed] (-.6,0)--++(1.2,0);
		\node[anchor=north east] at (-.1,-.1) {$\ell$};
		\node[anchor=south east] at (-.1,.1) {$\ell+1$};
		\node[anchor=north west] at (.1,-.1) {$\ell-1$};
		\node[anchor=south west] at (.1,.1) {$\ell$};
		\draw[ultra thick,->] (-.6,0)--++(.55,0);
		\draw[ultra thick,->] (0,-.6)--++(0,.55);
		\draw[ultra thick,->] (0,0)--++(.55,0);
		\draw[ultra thick,->] (0,0)--++(0,.55);
	\end{tikzpicture}}&
	\scalebox{.9}{\begin{tikzpicture}
		[scale=1.2, very thick]
		\draw[dashed] (0,-.6)--++(0,1.2);
		\draw[dashed] (-.6,0)--++(1.2,0);
		\node[anchor=north east] at (-.1,-.1) {$\ell$};
		\node[anchor=south east] at (-.1,.1) {$\ell+1$};
		\node[anchor=north west] at (.1,-.1) {$\ell$};
		\node[anchor=south west] at (.1,.1) {$\ell$};
		\draw[ultra thick,->] (-.6,0)--++(.6,0);
		\draw[ultra thick,->] (0,0)--++(0,.6);
	\end{tikzpicture}}&
	\scalebox{.9}{\begin{tikzpicture}
		[scale=1.2, very thick]
		\draw[dashed] (0,-.6)--++(0,1.2);
		\draw[dashed] (-.6,0)--++(1.2,0);
		\node[anchor=north east] at (-.1,-.1) {$\ell$};
		\node[anchor=south east] at (-.1,.1) {$\ell+1$};
		\node[anchor=north west] at (.1,-.1) {$\ell$};
		\node[anchor=south west] at (.1,.1) {$\ell+1$};
		\draw[ultra thick,->] (-.6,0)--++(.6,0);
		\draw[ultra thick,->] (0,0)--++(.6,0);
	\end{tikzpicture}}&
	\scalebox{.9}{\begin{tikzpicture}
		[scale=1.2, very thick]
		\draw[dashed] (0,-.6)--++(0,1.2);
		\draw[dashed] (-.6,0)--++(1.2,0);
		\node[anchor=north east] at (-.1,-.1) {$\ell$};
		\node[anchor=south east] at (-.1,.1) {$\ell$};
		\node[anchor=north west] at (.1,-.1) {$\ell-1$};
		\node[anchor=south west] at (.1,.1) {$\ell$};
		\draw[ultra thick,->] (0,-.6)--++(0,.6);
		\draw[ultra thick,->] (0,0)--++(.6,0);
	\end{tikzpicture}}&
	\scalebox{.9}{\begin{tikzpicture}
		[scale=1.2, very thick]
		\draw[dashed] (0,-.6)--++(0,1.2);
		\draw[dashed] (-.6,0)--++(1.2,0);
		\node[anchor=north east] at (-.1,-.1) {$\ell$};
		\node[anchor=south east] at (-.1,.1) {$\ell$};
		\node[anchor=north west] at (.1,-.1) {$\ell-1$};
		\node[anchor=south west] at (.1,.1) {$\ell-1$};
		\draw[ultra thick,->] (0,-.6)--++(0,.6);
		\draw[ultra thick,->] (0,0)--++(0,.6);
	\end{tikzpicture}}&
	\scalebox{.9}{\begin{tikzpicture}
		[scale=1.2, very thick]
		\draw[dashed] (0,-.6)--++(0,1.2);
		\draw[dashed] (-.6,0)--++(1.2,0);
		\node[anchor=north east] at (-.1,-.1) {$\ell$};
		\node[anchor=south east] at (-.1,.1) {$\ell$};
		\node[anchor=north west] at (.1,-.1) {$\ell$};
		\node[anchor=south west] at (.1,.1) {$\ell$};
	\end{tikzpicture}}
	\\\hline
	\scalebox{.9}{1}
	&
		\scalebox{.9}{$\dfrac{(1-t)v}{v-t
		u}\dfrac{u-st^\ell}{v-st^\ell}$}
	&
			\scalebox{.9}{$\dfrac{v-u}{v-tu}
			\dfrac{v-st^{\ell+1}}{v-st^\ell}$}
	&
			\scalebox{.9}{$\dfrac{(1-t)u}{v-tu}
			\dfrac{v-st^{\ell}}{u-st^{\ell}}$}
	&
			\scalebox{.9}{$\dfrac{t(v-u)}{v-tu}
			\dfrac{u-st^{\ell-1}}{u-st^{\ell}}$}
	&
	\scalebox{.9}{1}
	\Bigg.
	\end{tabular}
	\caption{%
		Conditional probabilities in the random field
		$\mathbf{L}$ on $\mathbb{Z}_{\ge0}^2$.  In the top row all possible values
		of the field in the square $\left\{ x,x+1 \right\}\times\left\{ y,y+1
		\right\}$ are listed, where $\ell\in\mathbb{Z}_{\ge0}$ (and $\ell\ge1$ in
		the first, fourth, and fifth pictures).  The bottom row contains the
		corresponding conditional probabilities to sample the top right value
		$\boldsymbol\ell^{(x+1,y+1)}$ of the field given the three other values.
		The spectral parameters are $v=v_{x+1}$ and $u=u_{y+1}$.  The arrows
		represent identification with the six vertex configurations.
	}
	%\end{noindent}
	\label{fig:L_dynamicS6V_probabilities}
\end{figure}

First we define the space of configurations in our dynamic stochastic six
vertex model. Consider an ensemble of infinite up-right paths in the positive
integer quadrant with the following properties:
\begin{itemize}
	\item
	      Paths go along edges of the shifted lattice $\left(
		      \mathbb{Z}_{\ge0}+\frac{1}{2} \right)^2$;
	\item
	      Each edge of $\left( \mathbb{Z}_{\ge0}+\frac{1}{2} \right)^2$ is
	      occupied by at most one path;
	\item
	      Paths can touch each other at a vertex but cannot cross each
	      other;
	\item
	      On the boundary of the quadrant no paths enter from below, and
	      at each height $n+\frac{1}{2}$, $n\ge0$, a new path enters through the left
	      part of the boundary;
\end{itemize}
Fix such a configuration of up-right paths. At each $(x,y)$ in the original
non-shifted lattice $\mathbb{Z}_{\ge0}^2$ define the value of the \emph{height
	function}, $\mathfrak{h}(x,y)$, to be the number of paths passing below
$(x,y)$. See \Cref{fig:dyn_stoch6V} for an illustration.

\begin{definition}[DS6V]
	\label{def:dyn_S6V}
	The \emph{dynamic stochastic six vertex model} (\emph{DS6V} for short)
	is a probability distribution on ensembles of up-right paths (depending on the
	parameters $t\in[0,1)$, $s\in(-1,0]$, and two sequences $v_1,v_2,\ldots $ and
	$u_1,u_2,\ldots $ such that $0\le u_i<v_j<1$ for all $i,j$) defined
	inductively as follows. Suppose that the path configuration below the line
	$x+y\le n$ (for some $n\in \mathbb{Z}_{\ge1}$) is sampled. Thus, at each
	vertex
	$(n-\frac{1}{2},\frac{1}{2}),(n-\frac{3}{2},\frac{1}{2}),\ldots,(\frac{1}{2},n
		-\frac{1}{2})$ we know the incoming configuration of paths. We also know
	the values of the height function at each point $(x,y)\in
		\mathbb{Z}_{\ge0}^{2}$ with $x+y\le n$. Using the probabilities in
	\Cref{fig:L_dynamicS6V_probabilities}, sample the outgoing configuration of
	paths at each vertex
	$(n-\frac{1}{2},\frac{1}{2}),\ldots,(\frac{1}{2},n-\frac{1}{2})$
	independently, and then proceed by induction.
\end{definition}

The weights in \Cref{fig:L_dynamicS6V_probabilities} together with our
conditions on the parameters of the model imply that under the dynamic
stochastic six vertex model for each $y$ there almost surely exists $x$ such
that $\mathfrak{h}(x',y)=0$ for all $x'\ge x$. In other words, each path
almost surely reaches arbitrarily large vertical coordinates.

\begin{figure}[htpb]
	%\begin{noindent}
	\centering
	\scalebox{.9}{
		\begin{tikzpicture}
		[scale=1.3, very thick]
		\draw[->] (-.35,0)--++(6.85,0) node [right] {$x$};
		\draw[->] (0,-.35)--++(0,4.85) node [left] {$y$};
			\foreach \ii in {1,2,3,4,5,6}
			{
				\node at (\ii-.5,-1) {$v_\ii$};
				\draw[dotted, thick] (\ii,-.3)--++(0,4.65);
				\node at (\ii,-.6) {$\ii$};
			}
			\foreach \jj in {1,2,3,4}
			{
				\node at (-1,\jj-.5) {$u_\jj$};
				\draw[dotted, thick] (-.3,\jj)--++(6.65,0);
				\node at (-.6,\jj) {$\jj$};
			}
			\node at (-.6,0) {$0$};
			\node at (0,-.6) {$0$};
			\foreach \zz in {(0,0),(1,0),(2,0),(3,0),(4,0),(5,0),(6,0),(5,1),(6,1),(5,2),(6,2),(6,3)}
			{\node[rectangle,draw,fill=white] at \zz {0};}
			\foreach \zz in {(0,1),(1,1),(2,1),(3,1),(4,1),(3,2),(4,2),(5,3),(5,4),(6,4)}
			{\node[rectangle,draw,fill=white] at \zz {1};}
			\foreach \zz in {(0,2),(1,2),(2,2),(1,3),(2,3),(3,3),(4,3),(4,4)}
			{\node[rectangle,draw,fill=white] at \zz {2};}
			\foreach \zz in {(0,3),(1,4),(2,4),(3,4)}
			{\node[rectangle,draw,fill=white] at \zz {3};}
			\foreach \zz in {(0,4)}
			{\node[rectangle,draw,fill=white] at \zz {4};}
			\draw[line width=2.7,->] (-.65, .5)--(-.5,.5)--++(1,0)--++(1,0)--++(1,0)--++(1,0)--++(1,0)--++(0,1)--++(0,.92)--++(.08,.08)--++(.92,0)--++(0,1)--++(1,0);
			\draw[line width=2.7,->] (-.65,1.5)--(-.5,1.5)--++(1,0)--++(2,0)--++(0,1)--++(1,0)--++(0.92,0)--++(.08,.08)--++(0,.92)--++(0,1);
			\draw[line width=2.7,->] (-.65,2.5)--(-.5,2.5)--++(1,0)--++(0,.92)--++(.08,.08)--++(.92,0)--++(1,0)--++(1,0)--++(0,1);
			\draw[line width=2.7,->] (-.65,3.5)--(-.5,3.5)--++(0.92,0)--++(.08,.08)--++(0,.92);
		\end{tikzpicture}
	}
	\caption{Path configuration of six vertex type in a quadrant together with its height function.}
	%\end{noindent}
	\label{fig:dyn_stoch6V}
\end{figure}
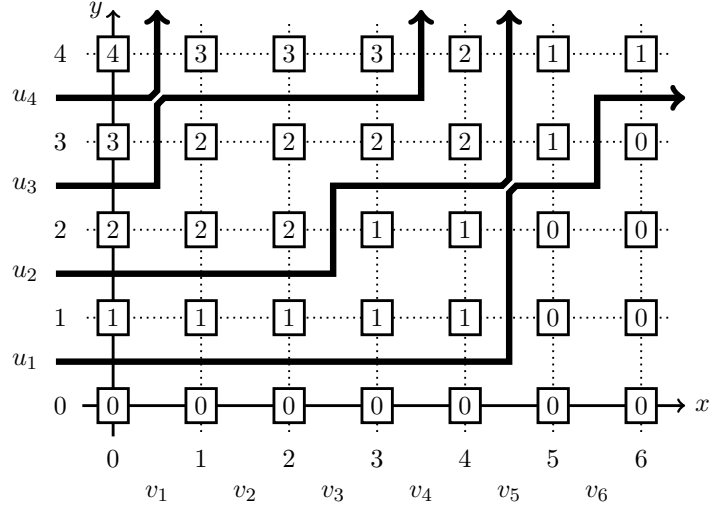

\begin{remark}
	\label{rmk:dynS6V_not_the_same}
	The vertex model introduced in \Cref{def:dyn_S6V} differs from the
	dynamic stochastic six vertex model presented recently in
	\cite{borodin2017elliptic} as a degeneration of the stochastic
	Interaction-Round-a-Face model (introduced in the same work). A higher spin
	model following the approach of the latter paper was then developed in
	\cite{aggarwal2017dynamical}. All these dynamic stochastic vertex models are
	closely related to versions of the Yang-Baxter equation with dynamic
	parameters (see \Cref{sub:dynamic_YB} below for our dynamic Yang-Baxter
	exuation which seems to be simpler than the one in
	\cite{borodin2017elliptic}). Therefore, we regard the model from
	\Cref{def:dyn_S6V} as another dynamic version of the stochastic six vertex
	model, different from the ones in \cite{borodin2017elliptic},
	\cite{aggarwal2017dynamical}.
\end{remark}

\begin{proposition}
	\label{prop:dyn6V_is_YB_field}
	Let $\mathfrak{H}:=\{\mathfrak{h}(x,y)\}_{x,y\ge0}$ be the random
	field of values of the height function of DS6V (\Cref{def:dyn_S6V}). Let
	$\mathbf{L}=\{\boldsymbol\ell^{(x,y)}\}_{x,y\ge0}$ be the random field
	obtained as the projection of the Yang-Baxter random field of
	\Cref{def:YB_field} onto the column number zero. Then these random fields
	$\mathfrak{H}$ and $\mathbf{L}$ have the same distribution.
\end{proposition}
\begin{proof}
	Straightforward from the identification of weights in $\mathfrak{H}$
	and $\mathbf{L}$ in \Cref{fig:L_dynamicS6V_probabilities} together with the
	identification of the boundary conditions.
\end{proof}

From \Cref{thm:YB_field_spin_HL_process} and \Cref{prop:dyn6V_is_YB_field} we
immediately get the following interpretation of the distribution of the height
function in DS6V:
\begin{corollary}
	\label{cor:dyn6V_spin_HL_process}
	Fix a down-right path $\mathcal{P}_{\vec{x},\vec{y}}$ as in
	\eqref{spin_HL_down_right_path_sequences}--\eqref{spin_HL_down_right_path}.
	The joint distribution of the random variables $\{\mathfrak{h}(p)\colon
		p\in\mathcal{P}_{\vec{x},\vec{y}}\}$ (i.e., the values of the height function
	of the dynamic stochastic six vertex model along this down-right path),
	coincides with the joint distribution of $\bigl\{(\lambda^{(p)})^{[0]}\colon
		p\in   \mathcal{P}_{\vec{x},\vec{y}}\bigr\}$, the numbers of zero parts in
	the signatures $\lambda^{(p)}$ governed by the spin Hall-Littlewood process
	$\mathscr{HP}_{\vec{x},\vec{y}}$ corresponding to the down-right path
	$\mathcal{P}_{\vec{x},\vec{y}}$.
\end{corollary}

\subsection{A dynamic Yang-Baxter equation}
\label{sub:dynamic_YB}

The probabilities of vertex configurations in DS6V (given in
\Cref{fig:L_dynamicS6V_probabilities}) satisfy a dynamic version of the
Yang-Baxter equation. It is convenient to formulate it in terms of the values
of the height function since the corresponding arrow configurations can be
readily recovered as in \Cref{fig:L_dynamicS6V_probabilities}. Consider two
three-line configurations as in \Cref{fig:dyn_YBE}. Fix the six boundary
values $\ell_0,\ell_1,\ell_2,\ell_1',\ell_2',\ell_3\in \mathbb{Z}_{\ge0}$ of
the height function. Clearly, these values can be arbitrary provided that they
satisfy
\begin{equation}
	\label{conditions_on_ell}
	\ell_1-\ell_0,\ell_2-\ell_1,\ell_3-\ell_2 \in \left\{ 0,1 \right\},
	\qquad
	\ell_1'-\ell_0,\ell_2'-\ell_1',\ell_3-\ell_2'\in \left\{ 0,1 \right\}.
\end{equation}
Also fix spectral parameters $\mathsf{u}_1,\mathsf{u}_2,\mathsf{v}$. For the
dynamic Yang-Baxter equation in \Cref{thm:dynamic_YB} below these parameters
do not have to satisfy any conditions as in \Cref{def:dyn_S6V}. However, if
$0\le \mathsf{u}_2<\mathsf{u}_1<\mathsf{v}<1$ and $0\le t<1$, $-1<s\le 0$,
then all the individual vertex weights entering the dynamic Yang-Baxter
equation belong to $[0,1]$.

\begin{figure}[htpb]
	%\begin{noindent}
	\centering
	\begin{tikzpicture}
		[scale=1,very thick]
		\draw[densely dashed] (0,0) -- (3,1.7);
		\draw[densely dashed] (0,1) -- (3,-.7);
		\draw[densely dashed] (2.2,-1.2)--++(0,3.4);
		\node at (2.6,-1) {$\ell_0$};
		\node at (1.3,-.2) {$\ell_1$};
		\node at (0,.5) {$\ell_2$};
		\node at (2.6,0.5) {$\ell_1'$};
		\node at (2.6,2) {$\ell_2'$};
		\node at (1.3,1.3) {$\ell_3$};
		\node at (1.75,.5) {$?$};
		\node at (-.3,-.2) {$\mathsf{u}_1$};
		\node at (-.3,1.2) {$\mathsf{u}_2$};
		\node at (2.2,-1.5) {$\mathsf{v}$};
		\begin{scope}[shift={(6,0)}]
			\draw[densely dashed] (0,1.7) -- (3,0);
			\draw[densely dashed] (0,-.7) -- (3,1);
			\draw[densely dashed] (.8,-1.2)--++(0,3.4);
			\node at (1.8,-.2) {$\ell_0$};
			\node at (.4,-1) {$\ell_1$};
			\node at (.4,.5) {$\ell_2$};
			\node at (3,0.5) {$\ell_1'$};
			\node at (1.8,1.3) {$\ell_2'$};
			\node at (.4,2) {$\ell_3$};
			\node at (1.25,.5) {$?$};
			\node at (-.3,-.9) {$\mathsf{u}_1$};
			\node at (-.3,1.9) {$\mathsf{u}_2$};
			\node at (.8,-1.5) {$\mathsf{v}$};
		\end{scope}
	\end{tikzpicture}
	\caption{%
		The dynamic Yang-Baxter equation for the dynamic stochastic six vertex weights in
		\Cref{fig:L_dynamicS6V_probabilities}.%
	}
	\label{fig:dyn_YBE}
	%\end{noindent}
\end{figure}
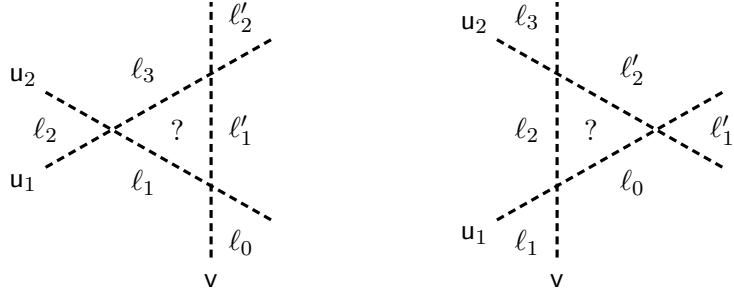

\begin{theorem}[Dynamic Yang-Baxter equation]
	\label{thm:dynamic_YB}
	Form two partition functions corresponding to the left and the right
	three-line configurations in \Cref{fig:dyn_YBE}. In both partition functions,
	the same boundary conditions satisfying \eqref{conditions_on_ell} are fixed,
	and the summation is over all possible values (in fact, no more than two) of
	the height function ``$?$'' inside the triangle. The spectral parameters
	$\mathsf{u}_1,\mathsf{u}_2,\mathsf{v}$ are attached to the three lines, and at
	each intersection the corresponding ``horizontal'' and ``vertical'' parameters
	replace $u$ and $v$, respectively, in the weights in
	\Cref{fig:L_dynamicS6V_probabilities}.

	Then these two partition functions are equal to each other.
\end{theorem}
\begin{proof}
	There are totally 20 types of identities corresponding to various
	choices of the boundary conditions satisfying \eqref{conditions_on_ell}, and
	depending on one particular value of the height function, say, $\ell_0=\ell$.
	Each of these identities is readily verified by hand. For illustration, let us
	present one such identity:
	\begin{equation*}
		\left[\scalebox{.7}{
				\begin{tikzpicture}[scale=1.1, thick,
						baseline=12pt]
					\draw[densely dashed] (0,0) --
					(3,1.7);
					\draw[densely dashed] (0,1) --
					(3,-.7);
					\draw[densely dashed]
					(2.2,-1.2)--++(0,3.4);
					\draw[->, line width=2]
					(0,1)--(0.882353,.5);
					\draw[->, line width=2]
					(0.882353,.5)--(2.2,1.24667);
					\draw[->, line width=2]
					(2.2,1.24667)--(3,1.7);
					\node at (2.6,-1) {$\ell$};
					\node at (1.3,-.2) {$\ell$};
					\node at (0,.5) {$\ell$};
					\node at (2.6,0.5) {$\ell$};
					\node at (2.65,2) {$\ell+1$};
					\node at (1.3,1.3) {$\ell+1$};
					\node at (1.75,.5) {$\ell$};
				\end{tikzpicture}
			}\right]
		+
		\left[\scalebox{.7}{
				\begin{tikzpicture}[scale=1.1, thick,
						baseline=12pt]
					\draw[densely dashed] (0,0) --
					(3,1.7);
					\draw[densely dashed] (0,1) --
					(3,-.7);
					\draw[densely dashed]
					(2.2,-1.2)--++(0,3.4);
					\draw[->, line width=2]
					(0,1)--(0.882353,.5);
					\draw[->, line width=2]
					(0.882353,.5)--(2.2,-0.246667);
					\draw[->, line width=2]
					(2.2,-0.246667)--(2.2,1.24667);
					\draw[->, line width=2]
					(2.2,1.24667)--(3,1.7);
					\node at (2.6,-1) {$\ell$};
					\node at (1.3,-.2) {$\ell$};
					\node at (0,.5) {$\ell$};
					\node at (2.6,0.5) {$\ell$};
					\node at (2.65,2) {$\ell+1$};
					\node at (1.3,1.3) {$\ell+1$};
					\node at (1.75,.5) {$\ell+1$};
				\end{tikzpicture}
			}\right]
		=
		\left[\scalebox{.7}{
				\begin{tikzpicture}[scale=1.1,
						thick,baseline=12pt]
					\draw[densely dashed] (0,1.7) --
					(3,0);
					\draw[densely dashed] (0,-.7) --
					(3,1);
					\draw[densely dashed]
					(.8,-1.2)--++(0,3.4);
					\draw[->, line width=2]
					(0,1.7)--(0.8,1.246667);
					\draw[->, line width=2]
					(.8,1.24667)--(2.11765,.5);
					\draw[->, line width=2]
					(2.11765,.5)--(3,1);
					\node at (1.8,-.2) {$\ell$};
					\node at (.35,-1) {$\ell$};
					\node at (.35,.5) {$\ell$};
					\node at (3,0.5) {$\ell$};
					\node at (1.8,1.3) {$\ell+1$};
					\node at (.4,2) {$\ell+1$};
					\node at (1.25,.5) {$\ell$};
				\end{tikzpicture}
			}\right].
	\end{equation*}
	This translates into the following identity between rational functions
	\begin{multline*}
		\tfrac{(1-t) \mathsf{u}_1 (\mathsf{u}_2-s
			t^\ell)}{(\mathsf{u}_1-t \mathsf{u}_2) (\mathsf{u}_1-s
			t^\ell)}
		\tfrac{(\mathsf{v}-\mathsf{u}_1) (s
			t^{\ell+1}-\mathsf{v})}{(\mathsf{v}-t \mathsf{u}_1) (s
			t^\ell-\mathsf{v})}
		+
		\tfrac{(\mathsf{u}_1-\mathsf{u}_2) (\mathsf{u}_1-s
			t^{\ell+1})}{(\mathsf{u}_1-t \mathsf{u}_2)
			(\mathsf{u}_1-s t^\ell)}
		\tfrac{(1-t) \mathsf{v} (\mathsf{u}_2-s t^\ell)}{(\mathsf{v}-t
			\mathsf{u}_2) (\mathsf{v}-s t^\ell)}
		\tfrac{(1-t) \mathsf{u}_1 (\mathsf{v}-s
			t^{\ell+1})}{(\mathsf{v}-t \mathsf{u}_1)
			(\mathsf{u}_1-s t^{\ell+1})}
		\\=
		\tfrac{(\mathsf{v}-\mathsf{u}_2) (s
			t^{\ell+1}-\mathsf{v})}{(\mathsf{v}-t \mathsf{u}_2) (s
			t^\ell-\mathsf{v})}
		\tfrac{(1-t) \mathsf{u}_1 (\mathsf{u}_2-s
			t^\ell)}{(\mathsf{u}_1-t \mathsf{u}_2) (\mathsf{u}_1-s
			t^\ell)},
	\end{multline*}
	which is readily verified by hand. The remaining 19 identities
	comprising the dynamic Yang-Baxter equation are checked in a similar way, and
	the theorem follows.
\end{proof}

\begin{remark}
	\label{rmk:dyn_YB_from_static}
	The dynamic Yang-Baxter equation of \Cref{thm:dynamic_YB} can in fact
	be reduced to the usual Yang-Baxter equation for the stochastic six vertex
	model, but we do not use this fact here.
\end{remark}

The dynamic Yang-Baxter equation of \Cref{thm:dynamic_YB} satisfied by the
probabilities in the dynamic stochastic six vertex model hints at the model's
integrability (i.e., that certain observables of this model are computable in
explicit form). We do not discuss these problems in the present work, though
in \Cref{sec:degenerations} below we consider degenerations of DS6V for which
certain observables indeed can be computed in explicit form.

\appendix

\section{Degenerations and limits}
\label{sec:degenerations}

%\begin{noindent}
Here we discuss a number of degenerations of 
the dynamic stochastic six vertex model (DS6V) 
and its properties
stated in \Cref{thm:YB_field_spin_HL_process} and
\Cref{cor:dyn6V_spin_HL_process}. Some of these degenerations
correspond to degenerations of the spin Hall-Littlewood symmetric functions
outlined in \cite[Section 8]{Borodin2014vertex}. The tables in
\Cref{fig:dynamic_vertex_weights_degen,fig:dynamic_vertex_weights_degen_half_cont} 
list various degenerations of the DS6V weights considered in 
Appendices \ref{sub:degen_HL} to \ref{sub:hc_degen_t0s0}.
Additional (less direct) degenerations are discussed in 
\Cref{sub:degen_rational,sub:degen_ASEP,sub:degen_finite_spin}.
We also discuss two degenerations of the full Yang-Baxter field in \Cref{sub:degen_HL,sub:hc_degen_Schur},
and compare them to known systems.
\begin{remark}
	\label{rmk:first_k_columns}
	Every degeneration of the DS6V model we consider 
	can be lifted to a $k$-layer model, where $k\ge2$ is arbitrary.
	Indeed, such a model would arise by taking the corresponding degeneration
	of the full Yang-Baxter field, and looking at its Markov projection
	onto the first $k$ columns 
	as in \Cref{sub:properties_of_global_transitions}.
	Such multilayer models for $s=0$ were explicitly written down in 
	\cite{BufetovMatveev2017}.
	For shortness, we will not address multilayer
	extensions in the present work.
\end{remark}

For simplicity we assume that the spectral parameters are
constant, $u_i\equiv u$ and $v_j\equiv v$, but most constructions (except the ASEP type limit in \Cref{sub:degen_ASEP}) 
work for the inhomogeneous parameters $u_i,v_j$, too.
%\end{noindent}

\begin{figure}[htpb]
	%\begin{noindent}
	\centering
	\begin{tabular}{c|c|c|c|c|c}
		&&
	\scalebox{.9}{\begin{tikzpicture}
		[scale=1.2, very thick]
		\draw[dashed] (0,-.6)--++(0,1.2);
		\draw[dashed] (-.6,0)--++(1.2,0);
		\node[anchor=north east] at (-.1,-.1) {$\ell$};
		\node[anchor=south east] at (-.1,.1) {$\ell+1$};
		\node[anchor=north west] at (.1,-.1) {$\ell$};
		\node[anchor=south west] at (.1,.1) {$\ell$};
		\draw[ultra thick,->] (-.6,0)--++(.6,0);
		\draw[ultra thick,->] (0,0)--++(0,.6);
	\end{tikzpicture}}&
	\scalebox{.9}{\begin{tikzpicture}
		[scale=1.2, very thick]
		\draw[dashed] (0,-.6)--++(0,1.2);
		\draw[dashed] (-.6,0)--++(1.2,0);
		\node[anchor=north east] at (-.1,-.1) {$\ell$};
		\node[anchor=south east] at (-.1,.1) {$\ell+1$};
		\node[anchor=north west] at (.1,-.1) {$\ell$};
		\node[anchor=south west] at (.1,.1) {$\ell+1$};
		\draw[ultra thick,->] (-.6,0)--++(.6,0);
		\draw[ultra thick,->] (0,0)--++(.6,0);
	\end{tikzpicture}}&
	\scalebox{.9}{\begin{tikzpicture}
		[scale=1.2, very thick]
		\draw[dashed] (0,-.6)--++(0,1.2);
		\draw[dashed] (-.6,0)--++(1.2,0);
		\node[anchor=north east] at (-.1,-.1) {$\ell$};
		\node[anchor=south east] at (-.1,.1) {$\ell$};
		\node[anchor=north west] at (.1,-.1) {$\ell-1$};
		\node[anchor=south west] at (.1,.1) {$\ell$};
		\draw[ultra thick,->] (0,-.6)--++(0,.6);
		\draw[ultra thick,->] (0,0)--++(.6,0);
	\end{tikzpicture}}&
	\scalebox{.9}{\begin{tikzpicture}
		[scale=1.2, very thick]
		\draw[dashed] (0,-.6)--++(0,1.2);
		\draw[dashed] (-.6,0)--++(1.2,0);
		\node[anchor=north east] at (-.1,-.1) {$\ell$};
		\node[anchor=south east] at (-.1,.1) {$\ell$};
		\node[anchor=north west] at (.1,-.1) {$\ell-1$};
		\node[anchor=south west] at (.1,.1) {$\ell-1$};
		\draw[ultra thick,->] (0,-.6)--++(0,.6);
		\draw[ultra thick,->] (0,0)--++(0,.6);
	\end{tikzpicture}}
	\\\hline
	&\scalebox{.9}{\parbox{.1\textwidth}{Original weights}}
	&
		\scalebox{.9}{$\dfrac{(1-t)v}{v-t
		u}\dfrac{u-st^\ell}{v-st^\ell}$}
	&
			\scalebox{.9}{$\dfrac{v-u}{v-tu}
			\dfrac{v-st^{\ell+1}}{v-st^\ell}$}
	&
			\scalebox{.9}{$\dfrac{(1-t)u}{v-tu}
			\dfrac{v-st^{\ell}}{u-st^{\ell}}$}
	&
			\scalebox{.9}{$\dfrac{t(v-u)}{v-tu}
			\dfrac{u-st^{\ell-1}}{u-st^{\ell}}$}
	\bigg.
	\\\hline
	(a)&\scalebox{.9}{\Cref{sub:degen_HL}}
	&
	\scalebox{.9}{$\dfrac{(1-t)u/v}{1-tu/v}$}
	&
	\scalebox{.9}{$\dfrac{1-u/v}{1-tu/v}$}
	&
	\scalebox{.9}{$\dfrac{1-t}{1-tu/v}$}
	&
	\scalebox{.9}{$\dfrac{t(1-u/v)}{1-tu/v}$}
	\bigg.
	\\\hline
	(b)&\scalebox{.9}{\Cref{sub:degen_t0}}
	&
		\scalebox{.9}{$\dfrac{u-s\mathbf{1}_{\ell=0}}{v-s\mathbf{1}_{\ell=0}}$}
	&
			\scalebox{.9}{$
			\dfrac{v-u}{v-s\mathbf{1}_{\ell=0}}$}
	&
	\scalebox{.9}{$1$}
	&
	\scalebox{.9}{$0$}
	\bigg.
	\\\hline
	(c)&\scalebox{.9}{\Cref{sub:degen_Schur}}
	&
	\scalebox{.9}{$u/v$}
	&
	\scalebox{.9}{$1-u/v$}
	&
	\scalebox{.9}{$1$}
	&
	\scalebox{.9}{$0$}
	\big.
	\\\hline
	(d)&\scalebox{.9}{\Cref{sub:degen_IHL}}
	&
		\scalebox{.9}{$\dfrac{(1-t)v}{v-t
		u}\dfrac{u+t^\ell}{v+t^\ell}$}
	&
			\scalebox{.9}{$\dfrac{v-u}{v-tu}
			\dfrac{v+t^{\ell+1}}{v+t^\ell}$}
	&
			\scalebox{.9}{$\dfrac{(1-t)u}{v-tu}
			\dfrac{v+t^{\ell}}{u+t^{\ell}}$}
	&
			\scalebox{.9}{$\dfrac{t(v-u)}{v-tu}
			\dfrac{u+t^{\ell-1}}{u+t^{\ell}}$}
	\bigg.
	\\\hline
	(e)&\scalebox{.9}{\Cref{sub:degen_t0s0}}
	&
	\scalebox{.9}{$\dfrac{u+\mathbf{1}_{\ell=0}}{v+\mathbf{1}_{\ell=0}}$}
	&
	\scalebox{.9}{$\dfrac{v-u}{v+\mathbf{1}_{\ell=0}}$}
	&
	\scalebox{.9}{$1$}
	&
	\scalebox{.9}{$0$}
	\bigg.
	\end{tabular}
	\caption{%
		Direct degenerations of the dynamic stochastic six vertex weights 
		from \Cref{sec:dynamicS6V} considered in the first part of \Cref{sec:degenerations}.
		Here $\ell\in \mathbb{Z}_{\ge0}$ (and $\ell\ge1$ in the last two cases)
		is the parameter corresponding to the height function, and $\mathbf{1}_{\cdots}$
		denotes the indicator of an event.
		The vertices $(1,1;1,1)$ and $(0,0;0,0)$ always having 
		weight $1$ are not shown.%
	}
	\label{fig:dynamic_vertex_weights_degen}
	%\end{noindent}
\end{figure}

\subsection{Hall-Littlewood degeneration and stochastic six vertex model}
\label{sub:degen_HL}

Setting $s=0$ and keeping other parameters makes the DS6V weights independent
of the height function. Moreover, in this degeneration the weights depend only
on the ratio $u/v$ and not on the individual parameters $u,v$. See
\Cref{fig:dynamic_vertex_weights_degen}(a). Thus, in this limit the DS6V turns
into the usual stochastic six vertex model introduced in \cite{GwaSpohn1992}
and studied in Integrable Probability since \cite{BCG6V}. The spin
Hall-Littlewood symmetric functions $F$ and $G^c$ turn (up to simple factors)
into the Hall-Littlewood symmetric polynomials \cite[Ch. III]{Macdonald1995}.
The correspondence between the stochastic six vertex model and Hall-Littlewood
processes following from \Cref{cor:dyn6V_spin_HL_process} was obtained earlier
in \cite{borodin2016stochastic_MM} (at the level of formulas),
\cite{BorodinBufetovWheeler2016} (for a half-continuous degeneration, cf.
\Cref{sub:hc_degen_HL} below), and in full form in \cite{BufetovMatveev2017}.

The Yang-Baxter field $\boldsymbol \Lambda:=\{
	\boldsymbol\lambda^{(x,y)} \}_{x,y\ge0}$ for $s=0$ becomes a certain
field of random Young diagrams indexed by $\mathbb{Z}_{\ge0}^{2}$ related to
Hall-Littlewood measures and processes. This random field \emph{differs} from
the Hall-Littlewood RSK field introduced in \cite{BufetovMatveev2017}, despite
that:
\begin{itemize}
	\item
	      In both fields, joint distributions along down-right paths are
	      the same and are given by the Hall-Littlewood processes as in
	      \Cref{cor:dyn6V_spin_HL_process}.
	\item
	      The projection onto the first column in both fields produces the
	      stochastic six vertex model.
\end{itemize}
The existence of two different random fields with these properties might seem
surprising, but such non-uniqueness of 2-dimensional stochastic dynamics was
observed before, e.g., in \cite{BorodinPetrov2013NN} or \cite[Section
	4]{BorodinPetrov2013Lect}. The fact that the $s=0$ Yang-Baxter field and the
Hall-Littlewood RSK field are indeed different will be evident in
\Cref{sub:hc_degen_Schur} when we take further degenerations and obtain
different objects.

\begin{remark}
	\label{rmk:spinHL_not_RSK}
	The Hall-Littlewood RSK field of \cite{BufetovMatveev2017} has an
	additional structure coming from the fact that the skew Hall-Littlewood
	symmetric functions in one variable are proportional to a power of the
	variable. Using this fact, analogues of the probabilities
	$\mathsf{U}^{\mathrm{fwd}}_{v,u}$ and $\mathsf{U}^{\mathrm{bwd}}_{v,u}$ for
	the Hall-Littlewood RSK field lead to \emph{randomized RSK correspondences}:
	having Young diagrams $\mu,\kappa,\lambda$, and an integer $r\in
		\mathbb{Z}_{\ge0}$ (corresponding to the power of $u/v$), the randomized RSK
	produces a random output Young diagram $\nu$. See \cite[Section
		3.6]{BufetovMatveev2017} for details on this reduction of a random field of
	Young diagrams to randomized RSK correspondences with input.

	However, for $s\ne 0$ the skew spin Hall-Littlewood symmetric
	functions in one variable are not simply proportional to powers of the
	variables. This presents a clear obstacle to a possible reduction of the
	Yang-Baxter field or another such random field of signatures to a randomized
	correspondence with integer input. Therefore, we do not address this issue in
	the present work.
\end{remark}

Observables of Hall-Littlewood processes pertaining to the projection onto
first columns can be extracted using the action of Hall-Littlewood versions of
Macdonald difference operators (e.g., see \cite{dimitrov2016kpz}). Thus, the
connection between the stochastic six vertex model and Hall-Littlewood
processes produces tools for the analysis of the former model alternative to
the original approach of \cite{BCG6V}. See, e.g.,
\cite{borodin2016stochastic_MM} for an analysis via Hall-Littlewood measures.

\subsection{Schur degeneration and modified discrete time PushTASEP}
\label{sub:degen_t0}

Setting $t=0$ and keeping all other parameters makes the DS6V weights look as
in \Cref{fig:dynamic_vertex_weights_degen}(b). These weights are still dynamic
in the sense that they retain dependence on the height function. However, this
dependence only singles out the bottommost path: the behavior of all other
paths follows the same weights.

As noted in \cite[Section 8.3]{Borodin2014vertex}, the spin Hall-Littlewood
functions $F$ and $G^c$ for $t=0$ turn into certain determinants generalizing
Schur polynomials, thus making the spin Hall-Littlewood measures and processes 
in this degeneration potentially more tractable.

Let us reinterpret the $t=0$ degeneration of DS6V as a discrete time particle
system by regarding the horizontal direction as time (a similar interpretation
is valid for the general DS6V model, too, only the corresponding particle
system becomes more complicated.)

\begin{definition}
	\label{def:discrete_modified_push}
	Consider a discrete time particle system living on half infinite
	particle configurations $x_1(\mathsf{t})<x_2(\mathsf{t})<\ldots $,
	$\mathsf{t}\in \mathbb{Z}_{\ge0}$, on $\mathbb{Z}$. Identify this system with
	the $t=0$ degeneration of the DS6V model as follows:
	\begin{equation}
		\label{PushTASEP_interpretation_of_particle_sys}
		x_i(\mathsf{t})=k \qquad \Longleftrightarrow \qquad
		\mathfrak{h}(\mathsf{t},k-1)=i-1 \quad\textnormal{and}\quad
		\mathfrak{h}(\mathsf{t},k)=i,
	\end{equation}
	cf. \Cref{fig:dyn_stoch6V} and the definition of the height function
	in \Cref{sub:dynamicS6V_subsection}. The boundary condition with arrows on the
	right in DS6V translates into the \emph{step initial condition} $x_i(0)=i$,
	$i\in \mathbb{Z}_{\ge1}$.

	The particle system on $\mathbb{Z}$ thus defined evolves as follows.
	In discrete time, particles jump to the right by one or stay put (indeed, this
	is because the weight of the vertex $(1,0;1,0)$ is zero). At each time step
	$\mathsf{t}\to \mathsf{t}+1$, the first particle flips a coin with the
	probability of success $(u-s)/(v-s)$, and each of the other particles flip
	independent coins with probability of success $u/v$. Then in the order from
	left to right, each particle $x_i$, $i=1,2,\ldots $ jumps to the right by one
	if either
	\begin{itemize}
		\item
		      the coin of $x_i$ is a success,
		\item
		      or if
		      $x_i(\mathsf{t})=x_{i-1}(\mathsf{t})+1=x_{i-1}(\mathsf{t}+1)$. In other words,
		      if the particle $x_{i-1}$ is moving to the right by one and its destination is
		      occupied by $x_i$, then $x_i$ is pushed to the right by one (and then the coin
		      of $x_i$ does not matter). If the destination of $x_i$ is also occupied, the
		      pushing propagates further to the right to $x_{i+1}$, and so on.
	\end{itemize}
	At each time step almost surely the update eventually terminates after
	a final push to the right by one of the infinite densely packed configuration.
\end{definition}

The particle system of \Cref{def:discrete_modified_push} is a \emph{modified
	discrete time PushTASEP} with a special behavior of the first particle (the
original discrete time PushTASEP is discussed in in \Cref{sub:degen_Schur}
next). To the best of the authors' knowledge, this modified PushTASEP was not
studied before by methods of integrable probability.

\subsection{Discrete time PushTASEP and Schur measures}
\label{sub:degen_Schur}

Setting $s=t=0$ in DS6V turns it into the discrete time PushTASEP (pushing
Totally Asymmetric Simple Exclusion Process). That is, interpreting the vertex
model as a particle system as in
\eqref{PushTASEP_interpretation_of_particle_sys}, we get the following
evolution. Initially $x_i(0)=i$, $i\in \mathbb{Z}_{\ge1}$. At each discrete
time step $\mathsf{t}\to\mathsf{t}+1$, each particle $x_1,x_2,\ldots $ (in
this order) independently jumps to the right by one with probability $u/v$,
following the pushing mechanism described in
\Cref{def:discrete_modified_push}.

When $s=t=0$, the spin Hall-Littlewood symmetric functions turn into the Schur
symmetric polynomials \cite[Ch. II.3]{Macdonald1995}, and the measures and
processes from \Cref{sub:spin_HL_measures_processes} turn into the Schur
measures and processes, which are determinantal with explicit double contour
integral kernels \cite{okounkov2001infinite}, \cite{okounkov2003correlation}.

The discrete time PushTASEP just described is a known particle system
associated with Schur measures and processes.\footnote{This discrete time
	PushTASEP is known as the Bernoulli one. There is also geometric PushTASEP in
	which particles jump to the right by arbitrary distance according to some
	distribution. These processes can be read off from, e.g.,
	\cite{BorFerr2008DF}; concise discrete time definitions are also obtained by
	setting $q=0$ in \cite[Sections 5.2 and 6.3]{MatveevPetrov2014}. The
	continuous time version of the PushTASEP (which is a suitable limit of both
	the Bernoulli and the geometric PushTASEPs) is discussed in
	\Cref{sub:hc_degen_Schur}.} However, its relation to the Schur measures
following from our \Cref{cor:dyn6V_spin_HL_process} \emph{differs} from the
one in \cite{BorFerr2008DF}. A connection similar to the latter one was
employed in \cite{BorFerr08push} for asymptotic analysis. Let us compare these
two connections in the single-point case (though both of them can be lifted to
suitable multipoint statements).

\begin{proposition}[\cite{BorFerr2008DF}]
	\label{prop:PushTASEP_known}
	For the discrete time PushTASEP with step initial condition and
	probability of jump $u/v$, we have the following equality in distribution for
	all $N\ge1$, $\mathsf{t}\ge0$:
	\begin{equation*}
		x_N(\mathsf{t})\stackrel{d}{=}\lambda_N+N,
	\end{equation*}
	where $\lambda=(\lambda_1,\ldots,\lambda_N)\in \mathsf{Sign}_N^+$ is a
	random signature distributed according to the Schur measure
	\begin{equation*}
		\mathrm{Prob}(\lambda)
		=
		\frac{1}{Z}\,s_\lambda(\underbrace{1,\ldots,1}_N)
		s_{\lambda'}(\underbrace{\tfrac uv,\ldots,\tfrac uv
		}_{\mathsf{t}}).
	\end{equation*}
	Here $Z$ is the normalizing constant and $\lambda'$ denotes the
	transposition (in the language of Young diagrams) of $\lambda$.
\end{proposition}

Recall that via the identification
\eqref{PushTASEP_interpretation_of_particle_sys}, the vertex model height
function can be interpreted as $\mathfrak{h}(\mathsf{t},x):=\#\left\{
	\textnormal{particles at time $\mathsf{t}$ which are $\le x$}
	\right\}$, which is natural to view as the height function of the
PushTASEP.

\begin{proposition}[$t=s=0$ in \Cref{cor:dyn6V_spin_HL_process}]
	\label{prop:PushTASEP_new}
	For the discrete time PushTASEP as above we have for all $N\ge1$ and
	$\mathsf{t}\ge0$:
	\begin{equation*}
		\mathfrak{h}(\mathsf{t},N)\stackrel{d}{=}\mu^{[0]},
	\end{equation*}
	where $\mu=(\mu_1,\ldots,\mu_N )\in \mathsf{Sign}_N^+$ is distributed
	according to the Schur measure
	\begin{equation}
		\label{Schur_for_discrete_PushTASEP}
		\mathrm{Prob}(\mu)=\frac{1}{Z}\,
		s_{\mu}(\underbrace{v^{-1},\ldots,v^{-1}
		}_{\mathsf{t}})s_{\mu}(\underbrace{u,\ldots,u }_{N}).
	\end{equation}
	Here $Z$ is the normalizing constant, and $\mu^{[0]}$ denotes the
	number of zero parts in the signature $\mu$ (in other words,
	$\mu^{[0]}=N-\mu_1'$).
\end{proposition}
Note that when $\mathsf{t}<N$, $\mathrm{Prob}(\mu)$ given by
\eqref{Schur_for_discrete_PushTASEP} automatically vanishes if
$\mu_{\mathsf{t}+1}>0$, as it should be. Indeed, after time $\mathsf{t}<N$
there are at least $N-\mathsf{t}$ particles in the PushTASEP at locations $\le
	N$, so the value of the height function $N-\mu_1'$ must be at least
$N-\mathsf{t}$.

These two connections between PushTASEP and Schur measures admit different
deformations along the hierarchy of symmetric functions. Namely,
\Cref{prop:PushTASEP_known} can be generalized by inserting the $q$-Whittaker
parameter $q\in(0,1)$, which gives rise to $q$-PushTASEP connected with
$q$-Whittaker measures and processes \cite{BorodinPetrov2013NN},
\cite{CorwinPetrov2013}, \cite{MatveevPetrov2014}. On the other hand,
\Cref{prop:PushTASEP_new} is generalized to our
\Cref{cor:dyn6V_spin_HL_process}, and thus the PushTASEP is lifted to the
dynamic stochastic six vertex model depending on two additional parameters
$t\in(0,1)$ and $s\in (-1,0)$ and related to the spin Hall-Littlewood measures
and processes.

Moreover, \Cref{prop:PushTASEP_known} can be generalized to PushTASEP with
particle-dependent jumping probabilities while \Cref{prop:PushTASEP_new} can
be extended to PushTASEP in inhomogeneous space. In the latter version of
PushTASEP, the jumping probability of a particle depends on the current
location of the particle, and not on the particle itself. The asymptotics of
PushTASEP in inhomogeneous space are studied in the forthcoming work
\cite{Petrov2017push}.

We postpone the discussion of the $t=s=0$ degeneration of the Yang-Baxter
field to \Cref{sub:hc_degen_Schur} where a half-continuous rescaling further
simplifies the object.

\subsection{Hall-Littlewood degeneration with rescaling}
\label{sub:degen_IHL}

Renaming $(u,v)$ to $(-su,-sv)$ makes the DS6V weights independent of $s$, cf.
\Cref{fig:dynamic_vertex_weights_degen}(d). The new degenerate weights still
contain the dynamic dependence on the value of the height function. They are
nonnegative for $0\le t<1$ and $0\le u<v<1$.

Taking variables $-su_i$ and $-sv_j$ in the spin Hall-Littlewood functions $F$
and $G^c$, respectively, we can then send $s\to0$. This limit requires a
rescaling of the functions themselves, but the spin Hall-Littlewood measures
and processes have $s\to0$ limits without any rescaling. The symmetric
functions $F$ and $G^c$ in this $s\to0$ limit become polynomials in $u_i$ or
$v_j$, respectively, whose top degree homogeneous components are the classical
Hall-Littlewood symmetric polynomials \cite[Section 8.2]{Borodin2014vertex}.
The functions $F_{\lambda/\varnothing}$ under this degeneration can also be
viewed as eigenfunctions of the stochastic $q$-Boson particle system
\cite{BorodinCorwinPetrovSasamoto2013}.

This $s\to0$ limit with rescaling of the spin Hall-Littlewood measures could
be easier to analyze (to the point of asymptotics) than the measures
\eqref{spin_HL_measure} before the limit. Via \Cref{cor:dyn6V_spin_HL_process}
this would give tools for asymptotic analysis of a dynamic model with the
weights given in \Cref{fig:dynamic_vertex_weights_degen}(d).

\subsection{Schur degeneration with rescaling}
\label{sub:degen_t0s0}

Further setting $t=0$ in the model of \Cref{sub:degen_IHL} produces a model
with vertex weights in \Cref{fig:dynamic_vertex_weights_degen}(e) which are
very similar to the ones considered in \Cref{sub:degen_t0}. Interpreting the
vertex model as a discrete time particle system as in
\Cref{def:discrete_modified_push} produces another version of the discrete
time PushTASEP with a special behavior of the first particle.

\subsection{Half-continuous dynamic stochastic six vertex model}
\label{sub:degen_half_continuous_DS6V}

%\begin{noindent}
In
\Cref{sub:degen_half_continuous_DS6V,sub:hc_degen_HL,sub:hc_degen_t0,sub:hc_degen_Schur,sub:hc_degen_IHL,sub:hc_degen_t0s0}
we discuss the rescaling of DS6V to the continuous horizontal direction,
beginning with the half-continuous DS6V model itself. The degenerations of the
half-continuous DS6V model considered in
\Cref{sub:hc_degen_HL,sub:hc_degen_t0,sub:hc_degen_Schur,sub:hc_degen_IHL,sub:hc_degen_t0s0}
are summarized in \Cref{fig:dynamic_vertex_weights_degen_half_cont}.
%\end{noindent}

\begin{figure}[htpb]
	\centering
	\begin{tabular}{c|c|c|c|c|c}
		    &                                          &
		\scalebox{.9}{
			\begin{tikzpicture}[scale=1.2, very thick]
				\draw[dashed] (0,-.6)--++(0,1.2);
				\draw[dashed] (-.6,0)--++(1.2,0);
				\node[anchor=north east] at (-.1,-.1)
				{$\ell$};
				\node[anchor=south east] at (-.1,.1)
				{$\ell+1$};
				\node[anchor=north west] at (.1,-.1) {$\ell$};
				\node[anchor=south west] at (.1,.1) {$\ell$};
				\draw[ultra thick,->] (-.6,0)--++(.6,0);
				\draw[ultra thick,->] (0,0)--++(0,.6);
				\node at (0,-.9) {Rate};
			\end{tikzpicture}
		}   &
		\scalebox{.9}{
			\begin{tikzpicture}[scale=1.2, very thick]
				\draw[dashed] (0,-.6)--++(0,1.2);
				\draw[dashed] (-.6,0)--++(1.2,0);
				\node[anchor=north east] at (-.1,-.1)
				{$\ell$};
				\node[anchor=south east] at (-.1,.1)
				{$\ell+1$};
				\node[anchor=north west] at (.1,-.1) {$\ell$};
				\node[anchor=south west] at (.1,.1)
				{$\ell+1$};
				\draw[ultra thick,->] (-.6,0)--++(.6,0);
				\draw[ultra thick,->] (0,0)--++(.6,0);
				\node at (0,-.9) {Probability};
			\end{tikzpicture}
		}   &
		\scalebox{.9}{
			\begin{tikzpicture}[scale=1.2, very thick]
				\draw[dashed] (0,-.6)--++(0,1.2);
				\draw[dashed] (-.6,0)--++(1.2,0);
				\node[anchor=north east] at (-.1,-.1)
				{$\ell$};
				\node[anchor=south east] at (-.1,.1) {$\ell$};
				\node[anchor=north west] at (.1,-.1)
				{$\ell-1$};
				\node[anchor=south west] at (.1,.1) {$\ell$};
				\draw[ultra thick,->] (0,-.6)--++(0,.6);
				\draw[ultra thick,->] (0,0)--++(.6,0);
				\node at (0,-.9) {Probability};
			\end{tikzpicture}
		}   &
		\scalebox{.9}{
			\begin{tikzpicture}[scale=1.2, very thick]
				\draw[dashed] (0,-.6)--++(0,1.2);
				\draw[dashed] (-.6,0)--++(1.2,0);
				\node[anchor=north east] at (-.1,-.1)
				{$\ell$};
				\node[anchor=south east] at (-.1,.1) {$\ell$};
				\node[anchor=north west] at (.1,-.1)
				{$\ell-1$};
				\node[anchor=south west] at (.1,.1)
				{$\ell-1$};
				\draw[ultra thick,->] (0,-.6)--++(0,.6);
				\draw[ultra thick,->] (0,0)--++(0,.6);
				\node at (0,-.9) {Probability};
			\end{tikzpicture}
		}
		\\\hline
		(a)
		    &
		\scalebox{.9}{\Cref{sub:degen_half_continuous_DS6V}}
		    &
		\scalebox{.9}{$(1-t)(u-st^\ell)$}
		    &
		\scalebox{.9}{$1-O(v^{-1})$}
		    &
		\scalebox{.9}{$\dfrac{(1-t)u}{u-st^\ell}$}
		    &
		\scalebox{.9}{$\dfrac{tu-st^\ell}{u-st^\ell}$}
		\bigg.
		\\\hline
		(b) & \scalebox{.9}{\Cref{sub:hc_degen_HL}}
		    &
		\scalebox{.9}{$(1-t)u$}
		    &
		\scalebox{.9}{$1-O(v^{-1})$}
		    &
		\scalebox{.9}{$1-t$}
		    &
		\scalebox{.9}{$t$}
		\bigg.
		\\\hline
		(c) & \scalebox{.9}{\Cref{sub:hc_degen_t0}}
		    &
		\scalebox{.9}{$u-s\mathbf{1}_{\ell=0}$}
		    &
		\scalebox{.9}{$1-O(v^{-1})$}
		    &
		\scalebox{.9}{$1$}
		    &
		\scalebox{.9}{$0$}
		\bigg.
		\\\hline
		(d) & \scalebox{.9}{\Cref{sub:hc_degen_Schur}}
		    &
		\scalebox{.9}{$u$}
		    &
		\scalebox{.9}{$1-O(v^{-1})$}
		    &
		\scalebox{.9}{$1$}
		    &
		\scalebox{.9}{$0$}
		\bigg.
		\\\hline
		(e) & \scalebox{.9}{\Cref{sub:hc_degen_IHL}}
		    &
		\scalebox{.9}{$(1-t)(u+t^\ell)$}
		    &
		\scalebox{.9}{$1-O(v^{-1})$}
		    &
		\scalebox{.9}{$\dfrac{(1-t)u}{u+t^\ell}$}
		    &
		\scalebox{.9}{$\dfrac{tu-st^\ell}{u+t^\ell}$}
		\bigg.
		\\\hline
		(f) & \scalebox{.9}{\Cref{sub:hc_degen_t0s0}}
		    &
		\scalebox{.9}{$u+\mathbf{1}_{\ell=0}$}
		    &
		\scalebox{.9}{$1-O(v^{-1})$}
		    &
		\scalebox{.9}{$1$}
		    &
		\scalebox{.9}{$0$}
		\bigg.
	\end{tabular}
	\caption{The half-continuous DS6V model and its various degenerations.
		The vertices $(1,1;1,1)$ and $(0,0;0,0)$ always having weight $1$ are not
		shown.}
	\label{fig:dynamic_vertex_weights_degen_half_cont}
\end{figure}

Taking the expansion as $v\to+\infty$ of the DS6V vertex weights in
\Cref{fig:dyn_stoch6V}, we see that
\begin{align*}
	\Biggl[
		\scalebox{.8}{
			\begin{tikzpicture}[scale=1.2, very thick,
					baseline=-3pt]
				\draw[dashed] (0,-.6)--++(0,1.2);
				\draw[dashed] (-.6,0)--++(1.2,0);
				\node[anchor=north east] at (-.1,-.1)
				{$\ell$};
				\node[anchor=south east] at (-.1,.1)
				{$\ell+1$};
				\node[anchor=north west] at (.1,-.1) {$\ell$};
				\node[anchor=south west] at (.1,.1) {$\ell$};
				\draw[ultra thick,->] (-.6,0)--++(.6,0);
				\draw[ultra thick,->] (0,0)--++(0,.6);
			\end{tikzpicture}
		}
	\Biggr] & =
	v^{-1}(1-t)(u-st^{\ell})+O(v^{-2})
	,
	\quad
	\Biggl[
		\scalebox{.8}{
			\begin{tikzpicture}[scale=1.2, very thick,
					baseline=-3pt]
				\draw[dashed] (0,-.6)--++(0,1.2);
				\draw[dashed] (-.6,0)--++(1.2,0);
				\node[anchor=north east] at (-.1,-.1)
				{$\ell$};
				\node[anchor=south east] at (-.1,.1)
				{$\ell+1$};
				\node[anchor=north west] at (.1,-.1) {$\ell$};
				\node[anchor=south west] at (.1,.1)
				{$\ell+1$};
				\draw[ultra thick,->] (-.6,0)--++(.6,0);
				\draw[ultra thick,->] (0,0)--++(.6,0);
			\end{tikzpicture}
		}
		\Biggr]
	=
	1-O(v^{-1})
	,           \\
	\Biggl[\scalebox{.8}{
			\begin{tikzpicture}[scale=1.2, very thick,
					baseline=-3pt]
				\draw[dashed] (0,-.6)--++(0,1.2);
				\draw[dashed] (-.6,0)--++(1.2,0);
				\node[anchor=north east] at (-.1,-.1)
				{$\ell$};
				\node[anchor=south east] at (-.1,.1) {$\ell$};
				\node[anchor=north west] at (.1,-.1)
				{$\ell-1$};
				\node[anchor=south west] at (.1,.1) {$\ell$};
				\draw[ultra thick,->] (0,-.6)--++(0,.6);
				\draw[ultra thick,->] (0,0)--++(.6,0);
			\end{tikzpicture}
		}
		\Biggr]
	        & =
	\frac{(1-t)u}{u-st^{\ell}}+O(v^{-2})
	,\qquad \qquad \qquad
	\Biggl[\scalebox{.8}{
			\begin{tikzpicture}[scale=1.2, very thick,
					baseline=-3pt]
				\draw[dashed] (0,-.6)--++(0,1.2);
				\draw[dashed] (-.6,0)--++(1.2,0);
				\node[anchor=north east] at (-.1,-.1)
				{$\ell$};
				\node[anchor=south east] at (-.1,.1) {$\ell$};
				\node[anchor=north west] at (.1,-.1)
				{$\ell-1$};
				\node[anchor=south west] at (.1,.1)
				{$\ell-1$};
				\draw[ultra thick,->] (0,-.6)--++(0,.6);
				\draw[ultra thick,->] (0,0)--++(0,.6);
			\end{tikzpicture}
		}
		\Biggr]
	=
	\frac{tu-st^{\ell}}{u-st^{\ell}}+O(v^{-2})
	,
	\\&\hspace{50pt}
	\Biggl[
		\scalebox{.8}{
			\begin{tikzpicture}[scale=1.2, very thick,
					baseline=-3pt]
				\draw[dashed] (0,-.6)--++(0,1.2);
				\draw[dashed] (-.6,0)--++(1.2,0);
				\node[anchor=north east] at (-.1,-.1)
				{$\ell$};
				\node[anchor=south east] at (-.1,.1)
				{$\ell+1$};
				\node[anchor=north west] at (.1,-.1)
				{$\ell-1$};
				\node[anchor=south west] at (.1,.1) {$\ell$};
				\draw[ultra thick,->] (-.6,0)--++(.55,0);
				\draw[ultra thick,->] (0,-.6)--++(0,.55);
				\draw[ultra thick,->] (0,0)--++(.55,0);
				\draw[ultra thick,->] (0,0)--++(0,.55);
			\end{tikzpicture}
		}
		\Biggr]
	=
	\Biggl[
		\scalebox{.8}{
			\begin{tikzpicture}[scale=1.2, very thick,
					baseline=-3pt]
				\draw[dashed] (0,-.6)--++(0,1.2);
				\draw[dashed] (-.6,0)--++(1.2,0);
				\node[anchor=north east] at (-.1,-.1)
				{$\ell$};
				\node[anchor=south east] at (-.1,.1) {$\ell$};
				\node[anchor=north west] at (.1,-.1) {$\ell$};
				\node[anchor=south west] at (.1,.1) {$\ell$};
			\end{tikzpicture}
		}
		\Biggr]
	=1.
\end{align*}
Thus, for $v\gg1$, taking into account the DS6V boundary conditions, we see
that all up-right paths will go to the right most of the steps. Occasionally
with probability proportional to $v^{-1}$, a path might turn up using the
vertex $(0,1;1,0)$, move some random distance up using several vertices
$(1,0;1,0)$, and either turn right using $(1,0;0,1)$, or hit a neighboring
path above it using $(1,1;1,1)$ (recall that paths can touch each other at a
vertex but cannot cross each other). In the latter case, this neighboring path
now faces up, and in turn should make a number of upward steps and either
eventually turn right, or hit the next path, and so on. The update in this
vertical slice eventually terminates after some path decides to turn right, or
after the infinite densely packed cluster of paths is pushed up by one.

In the limit as $v\to+\infty$ we thus obtain a probability distribution on
up-right paths in the half-continuous quadrant $\mathbb{R}_{\ge0}\times
	(\mathbb{Z}_{\ge0}+\tfrac12)$. All paths enter through the left
boundary, and nothing enters from below. Each $\ell$-th path from below,
$\ell\in
	\mathbb{Z}_{\ge1}$, carries an independent Poisson process of rate
$(1-t)(u-st^{\ell-1})$. Outside arrivals of these Poisson
processes\footnote{To rigorously define the system note that the behavior of
	the paths up to vertical coordinate $M$ does not depend on the behavior of the
	system above $M$, for any $M\ge1$. Thus, the evolution of any finite part of
	the system with vertical coordinate $\le M$ is well-defined, and for different
	$M$ these processes are compatible, thus defining the measure on the full
	half-continuous quadrant.} all paths go to the right. When there is an arrival
in the $\ell$-th Poisson process, the corresponding path turns up, and then
behaves as explained in the previous paragraph using probabilities of the
vertices $(1,0;1,0)$, $(1,0;0,1)$, and $(1,1;1,1)$.

Similarly to \Cref{def:discrete_modified_push}, one can interpret this
half-continuous DS6V model as a continuous time particle system
$x_1(\tau)<x_2(\tau)<\ldots $, $\tau\in \mathbb{R}_{\ge0}$, started from the
step initial configuration $x_i(0)=i$, $i\ge1$. Namely, in continuous time
each particle $x_i(\tau)$ wakes up at rate $(1-t)(u-st^{i-1})$ and
instantaneously moves to the right by a random number of steps according to
the probabilities in \Cref{fig:dynamic_vertex_weights_degen_half_cont}(a). If
the particle $x_{i+1}$ is in the way of $x_i$, then $x_i$ stops at where
$x_{i+1}$ was before. At the same time moment, $x_{i+1}$ is pushed to the
right by one, wakes up, and can instantaneously move further to the right, and
so on.

The height function of the half-continuous DS6V is identified (via a limit of
\Cref{cor:dyn6V_spin_HL_process}) with an observable of a limit of the spin
Hall-Littlewood measure \eqref{spin_HL_measure} as $v\to+\infty$ and the
number of the variables $v^{-1}$ in $G^{c}$ grows as $\tau v$. (This
identification can also be extended to multipoint observables.) Such limits of
the spin Hall-Littlewood measures and processes exist and can be constructed
via the corresponding half-continuous rescaling of the Yang-Baxter field. We
will not discuss the half-continuous Yang-Baxter field in the full generality
of parameters, and instead in \Cref{sub:hc_degen_Schur} below focus on the
simpler $s=t=0$ case which can be readily compared to existing
$(2+1)$-dimensional dynamics associated with Schur processes.

\subsection{Half-continuous stochastic six vertex model}
\label{sub:hc_degen_HL}

Setting $s=0$ in the half-continuous DS6V model turns the rates and
probabilities in this model into the ones in
\Cref{fig:dynamic_vertex_weights_degen_half_cont}(b). The vertex weights stop
being dynamic (i.e., they no longer depend on the value $\ell$ of the height
function), and the model becomes a half-continuous version of the stochastic
six vertex model. This model and its connection to Hall-Littlewood measures
and processes was considered in \cite{BorodinBufetovWheeler2016}.

\subsection{Continuous time modified PushTASEP}
\label{sub:hc_degen_t0}

Setting $t=0$ in the half-continuous DS6V model but keeping the parameter
$s\in(-1,0]$, and identifying the vertex model with a continuous time particle
system $x_1(\tau)<x_2(\tau)<\ldots $ yields the following system. Initially
$x_i(0)=i$, $i\ge 1$. Each particle has an independent exponential clock,
$x_1$ with a higher rate $u-s$, and each of the other ones with rate $u$. When
the clock of $x_i$ rings, it jumps to the right by one. If the destination is
occupied, and, more generally, if there is a packed cluster of particles
immediately to the right of $x_i$ (i.e., $x_i=x_{i+1}-1=\ldots
	=x_{i+k-1}-k+1=x_{i+k}-k$ before the jump), then each of the particles
$x_{i+1}, \ldots,x_{i+k} $ in this cluster is pushed to the right by one.

\subsection{Continuous time PushTASEP and $(2+1)$-dimensional Yang-Baxter
	dynamics}
\label{sub:hc_degen_Schur}

Setting $t=0$ in the half-continuous stochastic six vertex model of
\Cref{sub:hc_degen_HL}, or, which is the same, setting $s=0$ in the model of
\Cref{sub:hc_degen_t0}, leads to the usual continuous time \emph{PushTASEP}.
In this continuous time particle system on $\mathbb{Z}$, each particle
independently jumps to the right by one at rate $u$, and pushes to the right
the particles which are in the way.

The spin Hall-Littlewood measures turn into the Schur measures, and the limit
$v\to+\infty$ in the specialization $(v^{-1},\ldots,v^{-1} )$ ($v^{-1}$ is
repeated $\tau v$ times) corresponds to the so-called Plancherel
specialization of symmetric functions. In this way both
\Cref{prop:PushTASEP_known,prop:PushTASEP_new} readily lead to corresponding
statements for the continuous time PushTASEP.

Let us address what happens to the Yang-Baxter field under this
half-continuous $s=t=0$ degeneration, and compare it with other known
$(2+1)$-dimensional continuous time dynamics associated with Schur measures
and processes. Let us first introduce a suitable framework. Fix any $M\in
	\mathbb{Z}_{\ge1}$. A collection of signatures $\lambda^{(1)}\prec
	\lambda^{(2)}\prec \ldots\prec \lambda^{(M)}$, $\lambda^{(i)}\in
	\mathsf{Sign}_i^+$ (see \eqref{interlacing_definition} for notation),
is called an \emph{interlacing array} of depth $M$.\footnote{Also (often in
	connection with representation theory) referred to as a Gelfand--Tsetlin
	pattern.} We interpret the integers $\lambda^{(k)}_i$, $1\le i\le k\le M$, as
coordinates of particles in the space $\mathbb{Z}_{\ge0}\times \left\{
	1,\ldots,M  \right\}$. There are exactly $k$ particles on each level
$k=1,\ldots, M$.

We will consider a class of continuous time stochastic dynamics on interlacing
arrays called \emph{sequential update dynamics} introduced in
\cite{BorodinPetrov2013NN}. They evolve as follows:
\begin{itemize}
	\item
	      (\emph{independent jumps and blocking by particles below})
	      Each particle $\lambda^{(k)}_i$ has an independent exponential
	      clock of rate $w^{(k)}_i\ge0$ which may depend on the whole array. When the
	      clock rings, the particle $\lambda^{(k)}_i$ tries to jump to the right by one
	      (i.e., the coordinate $\lambda^{(j)}_i$ wishes to increase by one). If this
	      particle is blocked by the lower left neighbor, i.e.,
	      $\lambda^{(k)}_i=\lambda^{(k-1)}_{i-1}$ before the jump, then the jump of
	      $\lambda^{(k)}_{i}$ is forbidden.
	\item
	      (\emph{move propagation})
	      Denote the signature after the jump at level $k$ by $\nu^{(k)}$.
	      After a jump at level $k$, the update $\lambda^{(k)}\to\nu^{(k)}$ may initiate
	      a sequential cascade of instantaneous updates on all the upper levels,
	      $\lambda^{(k+1)}\to \nu^{(k+1)},\ldots,\lambda^{(M)}\to\nu^{(M)} $, according
	      to the transition probabilities $U_j(\lambda^{(j)}\to\nu^{(j)}\mid
		      \lambda^{(j-1)}\to\nu^{(j-1)})$. Here each $\nu^{(j)}$ differs from
	      $\lambda^{(j)}$ by a move of at most one particle to the right by one.
	\item
	      (\emph{mandatory pushing to preserve interlacing})
	      In order to preserve interlacing, the probabilities $U_j$ must
	      be equal to one in the case when
	      $\nu^{(j-1)}_{i}=\lambda^{(j-1)}_i+1=\lambda^{(j)}_i+1$ and
	      $\nu^{(j)}_i=\lambda^{(j)}_i+1$. In words, if a particle $\lambda^{(j-1)}_i$
	      moves and this breaks the interlacing with level $j$, then an instantaneous
	      move of $\lambda^{(j)}_i$ must be made to restore the interlacing.
\end{itemize}
Under certain conditions on $w^{(k)}_{i}$ and the transition probabilities
$U_k$ the sequential update dynamics acts nicely on Schur
processes,\footnote{That is, joint distributions in the dynamics started from
	the packed initial configuration $\lambda^{(k)}_j(0)\equiv 0$ are given by
	Schur processes along down-right paths as in our
	\Cref{thm:YB_field_spin_HL_process}.} see \cite{BorodinPetrov2013NN}. These
conditions might be interpreted as providing a bijectivisation of the skew
Cauchy identity for Schur polynomials when one of the specializations is
Plancherel.

The connection between the framework of interlacing arrays and the
half-continuous rescaling of the Yang-Baxter field is the following. Under the
rescaling of the horizontal coordinate $x$ to continuum, the Yang-Baxter field
(or any of its degenerations considered in
\Cref{sub:degen_HL,sub:degen_t0,sub:degen_Schur,sub:degen_IHL,sub:degen_t0s0})
$\{\boldsymbol\lambda^{(x,y)}\}$ indexed by $(x,y)\in
	\mathbb{Z}_{\ge0}$ turns into a field
$\{\boldsymbol\lambda^{(\tau,y)}\}$ indexed by $\tau\in
	\mathbb{R}_{\ge0}$, $y\in \mathbb{Z}_{\ge0}$. The boundary conditions
are $\boldsymbol\lambda^{(0,y)}=(0^y)$, $y\in \mathbb{Z}_{\ge0}$, and
$\boldsymbol\lambda^{(\tau,0)}=\varnothing$, $\tau\in \mathbb{R}_{\ge0}$. We
interpret the first $M$ rows $\{\boldsymbol\lambda^{(\tau,k)}\}_{k=1,\ldots,M
	}$ of the half-continuous random field of signatures as a continuous
time Markov dynamics (where $\tau$ is time) on interlacing arrays of depth $M$
via $\lambda^{(k)}(\tau)=\boldsymbol\lambda^{(\tau,k)}$, with initial
condition $\lambda^{(k)}(0)=(0^k)$.

In sequential update dynamics we describe below the quantities $w^{(k)}_i$ and
$U_k$ are essentially independent of $k$, i.e., the jumping and move
propagation mechanisms are the same at all levels of the interlacing array.
Thus, to describe such a dynamics let us fix $k$ and denote
\begin{equation}
	\label{lambda_kappa_seq_update_dynamics}
	\kappa=\lambda^{(k-1)}(\tau-),\qquad  \mu=\lambda^{(k)}(\tau-),\qquad
	\lambda=\lambda^{(k-1)}(\tau),\qquad  \nu=\lambda^{(k)}(\tau)
\end{equation}
(i.e., these are
signatures at levels $k-1$ and $k$ before and after the jump at time $\tau$,
respectively). The jump rates $w_i\equiv w^{(k)}_i$, $i=1,\ldots,k $,
correspond to independent jumps of the particles of $\mu$ when
$\lambda=\kappa$, and the transition probabilities $U(\mu\to\nu\mid
	\kappa\to\lambda)$ describe how the move at level $k-1$ propagates to level
$k$.

We are now in a position to describe the half-continuous $t=s=0$ degeneration
of the Yang-Baxter field:
\begin{definition}
	\label{def:Schur_dyn_YB}
	The Yang-Baxter continuous time dynamics on interlacing arrays looks
	as follows at each pair of consecutive levels $(k-1,k)$ (using notation
	\eqref{lambda_kappa_seq_update_dynamics}). When $\lambda=\kappa$, the rate of
	independent jump of each particle $\mu_i$ is in general equal to $u$, except:
	\begin{itemize}
		\item
		      (blocking from below) The rate of independent jump of
		      $\mu_i$
		      is zero if $\mu_i=\lambda_{i-1}$
		\item
		      (special blocking)
		      The rate of independent jump of $\mu_i$ is also zero if
		      \begin{equation}
			      \label{YB_Schur_cont_field_exotic_jump}
			      \mu_i=\lambda_i=\mu_{i+1}=\lambda_{i+1}=\ldots=
			      \mu_{i+m}=\lambda_{i+m}>\mu_{i+m+1}\quad
			      \text{for some $m\ge0$}.
		      \end{equation}
	\end{itemize}
	When $\lambda\ne \kappa$ and the difference is only in
	$\lambda_i=\kappa_i+1$, the transition probability $U(\mu\to\nu\mid
		\kappa\to\lambda)$ is in general equal to $\mathbf{1}_{\nu=\mu}$ (no move
	propagation), except:
	\begin{itemize}
		\item
		      (mandatory pushing to restore interlacing)
		      If $\lambda_i=\mu_i+1=\kappa_i+1$, this leads to
		      $\nu_i=\mu_i+1$ with probability $1$;
		\item
		      (special pushing)
		      If
		      \begin{equation}
			      \label{YB_Schur_cont_field_exotic_push}
			      \mu_i=\lambda_i=\mu_{i-1}=\lambda_{i-1}
			      =\ldots=\mu_{i-m}=\lambda_{i-m}<\lambda
			      _{i-m-1}\quad\text{for some $m\ge0$},
		      \end{equation}
		      then together with $\lambda_i=\kappa_i+1$ this leads to
		      $\nu_{i-m}=\mu_{i-m}+1$ with probability $1$.
	\end{itemize}
	In particular, in this dynamics the difference between $\lambda$ and
	$\kappa$, as well as between $\nu$ and $\mu$, is in the move of at most one
	particle to the right by one (in the language of Young diagrams, in adding one
	box).
\end{definition}

\begin{proposition}
	\label{prop:YB_Schur_cont_field}
	The half-continuous $t=s=0$ Yang-Baxter field is identified with the
	dynamics in \Cref{def:Schur_dyn_YB}.
\end{proposition}
\begin{proof}
	The Yang-Baxter field is determined using the forward transition
	probabilities $\mathsf{U}^{\mathrm{fwd}}_{v,u}$, which in turn are products of
	the local probabilities $P^{\mathrm{fwd}}_{v,u}$, cf.
	\Cref{def:Ufwd,def:YB_field}. Setting $s=t=0$ and expanding the latter as
	$v\to+\infty$ we get the quantities given in the table in
	\Cref{fig:fwd_Schur_cont}. Note that these quantities do not depend on the
	multiplicity $g$ of arrows in the middle as was the case for $s,t\ne 0$.
	Because of this, we can assume without loss of generality that all the
	multiplicities in the middle are $0$ or $1$. It remains to match the
	corresponding expansions as $v\to+\infty$ of $\mathsf{U}^{\mathrm{fwd}}_{v,u}$
	to rates $w_j$ and update probabilities $U(\mu\to\nu\mid \kappa\to\lambda)$
	given in \Cref{def:Schur_dyn_YB}. We do this in two steps, for
	$\lambda=\kappa$ (considering jump rates) and $\lambda\ne\kappa$ (dealing with
	move propagation).

	\begin{figure}[htpb]
		%\begin{noindent}
		\centering
		\scalebox{.9}{
			$
				\begin{array}{c||c||c|c||c|c||c}
					P^{\mathrm{fwd}}_{v,u}& \ooYBoo{.3}{3}{13.5pt}
						& \ioYBio{.3}{3}{13.5pt}
						& \oiYBio{.3}{3}{13.5pt}
						& \oiYBoi{.3}{3}{13.5pt}
						& \ioYBoi{.3}{3}{13.5pt}
						& \iiYBii{.3}{3}{13.5pt}
					\phantom{\bigg.}
					\\\hline
					\ooYBoo{.3}{3}{13.5pt}
					\phantom{\Big.}
						& 1
						& v^{-1}u
						&
					\coli
					1-O(v^{-1})
						&1
						&\colx 0
						& 1
					\\\hline
					\ioYBio{.3}{3}{13.5pt}
					\phantom{\Big.}
						& 1
						&
						v^{-1}u
						&
					\coli
					1-O(v^{-1})
						& 1
						& \colx0
						& 1
					\\\hline
					\oiYBio{.3}{3}{13.5pt}
					\phantom{\Big.}
						& \colx1
						& \colx0
						& 1
						&
					\colx
					1
						&
					\colxx
					0
						& \colx
						1
					\\\hline
					\oiYBoi{.3}{3}{13.5pt}
					\phantom{\Big.}
						& 1
						& 1
						& \coli0
						& 1
						& \colx0
						& 1
					\\\hline
					\ioYBoi{.3}{3}{13.5pt}
					\phantom{\Big.}
						& \coli1
						& \coli0
						& \colii1
						& \coli0
						& 1
						& \coli1
					\\\hline
					\iiYBii{.3}{3}{13.5pt}
					\phantom{\Big.}
						& 1
						& 
						v^{-1}u
						&
						\coli 1-O(v^{-1})
						&
						1
						&
						\colx 0
						& 1
				\end{array}
			$
		}
		\caption{Behavior of the forward local 
			Yang-Baxter transition probabilities for $s=t=0$
			as $v\to+\infty$. The coloring of the table cells is explained
			in \Cref{fig:fwd_YB}.
			Note that the parameters $u,v$
			are swapped compared to \Cref{fig:fwd_YB}, cf. \Cref{rmk:swap_u_v}.}
		\label{fig:fwd_Schur_cont}
		%\end{noindent}
	\end{figure}

	\smallskip\noindent
	\textit{Jump rates.} First consider the case $\lambda=\kappa$. Then
	the arrow configuration $\lambda\mathop{\dot\succ}\kappa\prec \mu$
	(cf. \Cref{fig:skew_Cauchy,fig:U_transition_probabilities}) looks as in
	\Cref{fig:Schur_YB_proof}(a), and we need to drag the cross vertex through
	this configuration from left to right. The nonnegative integer line
	$\mathbb{Z}_{\ge0}$ is divided into segments of two types: type~I segments
	$[\lambda_i,\mu_i)$ and type~II segments $[\mu_{i+1},\lambda_i)$. When
	$\lambda_i,\mu_i$ are sufficiently apart, these types of segments interlace,
	but it can also happen that segments of the same type can be neighbors.

	The cross vertex starts in state $\iiYBii{.25}{3.5}{10.5pt}$ in type~I
	segment, and this state cannot change thoughout type~I segment. Observe that
	on the boundary from type~I to type~II segment (say, corresponding to the
	arrow at $\mu_i$), if the length of type~II segment is positive, the cross
	vertex transforms (while moving to the right) as:
	\begin{itemize}
		\item
		      $\iiYBii{.25}{3.5}{10.5pt}
			      \rightsquigarrow\ioYBio{.25}{3.5}{10.5pt}
			      \rightsquigarrow\oiYBio{.25}{3.5}{10.5pt}$ with
		      probability $v^{-1}u+O(v^{-2})$ (i.e., at rate $u$) if
		      the length of the type~II segment is greater than $1$, or if the length of the
		      type~II segment is $1$ and the following type~I segment has zero length;
		\item
		      $\iiYBii{.25}{3.5}{10.5pt}
			      \rightsquigarrow\ioYBio{.25}{3.5}{10.5pt}
			      \rightsquigarrow\iiYBii{.25}{3.5}{10.5pt}$ with
		      probability $v^{-1}u+O(v^{-2})$ (i.e., at rate $u$) if
		      the length of the type~II segment is equal to $1$ and the next segment is
		      type~I of positive length;
		\item
		      $\iiYBii{.25}{3.5}{10.5pt}
			      \rightsquigarrow\oiYBio{.25}{3.5}{10.5pt}$ with
		      probability $1-O(v^{-1})$.
	\end{itemize}
	In the first two cases this move places an arrow in the middle at
	$\nu_i=\mu_i+1$, and in the second case an arrow is placed at $\nu_i=\mu_i$.
	As $v\to+\infty$, the move leading to $\nu_i=\mu_i+1$ can occur only once in
	the process of dragging the cross vertex, which proves the claim that in
	general the rates $w_i$ are equal to $u$.

	The cross vertex does not change throughout type~II segments and
	leaves such a segment as $\oiYBio{.25}{3.5}{10.5pt}$ (unless event of
	probability $v^{-1}u$ occurs and the length of type~II segment is $1$, but
	this is already considered above). When entering type~I segment of positive
	length (at, say, the boundary corresponding to $\lambda_j=\kappa_j$), the
	cross vertex transforms as $\oiYBio{.25}{3.5}{10.5pt}
		\rightsquigarrow\iiYBii{.25}{3.5}{10.5pt}$, and this removes an arrow in the
	middle at $\kappa_j$.

	A type~II segment of zero length corresponds to
	$\mu_{i+1}=\kappa_i=\lambda_i$ for some $i$, which blocks the independent jump
	of $\mu_{i+1}$. A type~I segment of zero length does not change the state of
	the cross from $\oiYBio{.25}{3.5}{10.5pt}$ which it has while traveling
	through type~II segment; this behavior corresponds to the special case
	\eqref{YB_Schur_cont_field_exotic_jump} in which the jump rate is zero. This
	establishes the claim about the jump rates.

	\begin{figure}[htpb]
		%\begin{noindent}
		\centering
			\begin{tikzpicture}
				[scale=.9]
				\def\h{.7}
				\draw[dashed] (0,0)--++(8.5,0);
				\draw[dashed] (0,\h)--++(8.5,0);
				\draw[->, ultra thick] (0,0)--++(2,0)--++(0,\h)--++(1,0)--++(0,.5);
				\draw[->, ultra thick] (0,\h)--++(1,0)--++(0,.5);
				\draw[->, ultra thick] (2,-.5)--++(0,.5)--++(1.5,0)--++(0,\h+.5);
				\draw[->, ultra thick] (3.5,-.5)--++(0,.5)--++(1,0)--++(0,\h)--++(1,0)--++(0,.5);
				\draw[->, ultra thick] (4.5,-.5)--++(0,.5)--++(1,0)--++(0,\h)--++(2,0)--++(0,.5);
				\draw[->, ultra thick] (5.5,-.5)--++(0,.5)--++(3,0);
				\node at (-.2,-.3) {$\lambda$};
				\node at (-.2,\h/2) {$\kappa$};
				\node at (-.2,\h+.3) {$\mu$};
				\node at (8,-.4) {(a)};
				\node at (.5,\h+.7) {I};
				\node at (1.5,\h+.7) {II};
				\node at (2.5,\h+.7) {I};
				\node at (3.25,\h+.7) {II};
				\node at (4,\h+.7) {II};
				\node at (5,\h+.7) {I};
				\node at (6.5,\h+.7) {I};
				\node at (8,\h+.7) {II};
			\end{tikzpicture}
			\\\vspace{10pt}
			\begin{tikzpicture}
				[scale=1]
				\def\h{.7}
				\draw[dashed] (0,0)--++(3,0);
				\draw[dashed] (0,\h)--++(3,0);
				\draw[->, ultra thick] (0,0)--++(1,0)--++(0,\h)--++(1.5,0)--++(0,.5);
				\draw[->, ultra thick] (1.5,-.5)--++(0,.5)--++(1.5,0);
				\node at (1.8,-.3) {$\lambda_i$};
				\node at (1.3,\h/2) {$\kappa_i$};
				\node at (2.85,\h+.3) {$\mu_i$};
				\node at (2.7,-.4) {(b)};
			\end{tikzpicture}
			% \qquad \qquad
			% \begin{tikzpicture}
			%   [scale=1]
			%   \def\h{.7}
			%   \draw[dashed] (0,0)--++(3,0);
			%   \draw[dashed] (0,\h)--++(3,0);
			%   \draw[->, ultra thick] (0,0)--++(1,0)--++(0,\h)--++(1.5,0)--++(0,.5);
			%   \draw[->, ultra thick] (0,\h)--++(1,0)--++(0,.5);
			%   \draw[->, ultra thick] (1.5,-.5)--++(0,.5)--++(1.5,0);
			%   \node at (1.8,-.3) {$\lambda_i$};
			%   \node at (1.3,\h/2) {$\kappa_i$};
			%   \node at (2.85,\h+.3) {$\mu_i$};
			%   \node at (.5,\h+.3) {$\mu_{i+1}$};
			%   \node at (2.7,-.4) {(c)};
			% \end{tikzpicture}
			\qquad
			\begin{tikzpicture}
				[scale=1]
				\def\h{.7}
				\draw[dashed] (0,0)--++(3,0);
				\draw[dashed] (0,\h)--++(3,0);
				\draw[->, ultra thick] (0,0)--++(1,0)--++(0,\h)--++(.5,0)--++(0,.5);
				\draw[->, ultra thick] (1.5,-.5)--++(0,.5)--++(1.5,0);
				\node at (1.8,-.3) {$\lambda_i$};
				\node at (1.3,\h/2) {$\kappa_i$};
				\node at (1.8,\h+.3) {$\mu_i$};
				\node at (2.7,-.4) {(c)};
			\end{tikzpicture}
			\qquad 
			\begin{tikzpicture}
				[scale=1]
				\def\h{.7}
				\draw[dashed] (0,0)--++(3,0);
				\draw[dashed] (0,\h)--++(3,0);
				\draw[->, ultra thick] (0,0)--++(1,0)--++(0,\h)--++(.5,0)--++(0,.5);
				\draw[->, ultra thick] (1.55,-.5)--++(0,.5)--++(1.45,0);
				\draw[->, ultra thick] (1.45,-.5)--++(0,.5)--++(0.05,0)--++(0,\h)--++(1.5,0);
				\node at (1.8,-.3) {$\lambda_i$};
				\node at (1.99,\h/2) {$\kappa_{i-1}$};
				\node at (.7,\h/2) {$\kappa_{i}$};
				\node at (1.8,\h+.3) {$\mu_i$};
				\node at (2.7,-.4) {(d)};
			\end{tikzpicture}
			\qquad
			\begin{tikzpicture}
				[scale=1]
				\def\h{.7}
				\draw[dashed] (0,0)--++(3,0);
				\draw[dashed] (0,\h)--++(3,0);
				\draw[->, ultra thick] (0,0)--++(1,0)--++(0,\h)--++(0,0)--++(0,.5);
				\draw[->, ultra thick] (1.5,-.5)--++(0,.5)--++(1.5,0);
				\node at (1.8,-.3) {$\lambda_i$};
				\node at (1.3,\h/2) {$\kappa_i$};
				\node at (1.35,\h+.3) {$\mu_i$};
				\node at (2.7,-.4) {(e)};
			\end{tikzpicture}
			% \qquad  \qquad
			% \begin{tikzpicture}
			%   [scale=1]
			%   \def\h{.7}
			%   \draw[dashed] (0,0)--++(3,0);
			%   \draw[dashed] (0,\h)--++(3,0);
			%   \draw[->, ultra thick] (0,0)--++(1,0)--++(0,\h)--++(.05,0)--++(0,.5);
			%   \draw[->, ultra thick] (0,\h)--++(.95,0)--++(0,.5);
			%   \draw[->, ultra thick] (1.5,-.5)--++(0,.5)--++(1.5,0);
			%   \node at (1.8,-.3) {$\lambda_i$};
			%   \node at (1.3,\h/2) {$\kappa_i$};
			%   \node at (2,\h+.3) {$\mu_i=\mu_{i+1}$};
			%   \node at (2.7,-.4) {(g)};
			% \end{tikzpicture}
		\caption{Arrow configurations 
		$\lambda\mathop{\dot\succ}\kappa\prec \mu$ in the proof of
		\Cref{prop:YB_Schur_cont_field}.}
		\label{fig:Schur_YB_proof}
		%\end{noindent}
	\end{figure}

	\smallskip
	\noindent
	\textit{Move propagation.}
	Assume now that $\lambda\ne \kappa$, and the difference between these
	two signatures at level $k-1$ can be only at one location,
	$\lambda_i=\kappa_i+1$. This fact would follow by induction on levels of the
	array after we show that the move propagation mechanism is as in
	\Cref{def:Schur_dyn_YB}. Indeed, this would imply that a single move of a
	particle by one cannot result in a move of a particle by more than one, or
	moves by more than one particle, at the level one higher.

	Then in the process of dragging the cross vertex through the arrow
	configuration $\lambda\mathop{\dot\succ}\kappa\prec \mu$ to obtain the
	signature $\nu$ all updates are deterministic: an event with probability
	$O(v^{-1})$ has already occured at level $k-1$ or below, and at a single time
	moment two or more such events cannot occur. Updates through the parts of the
	configuration where $\lambda_j=\kappa_j$ have been considered above: they all
	lead to setting $\nu_j=\mu_j$. Thus, it remains to consider the update coming
	from the passing of the cross vertex through the part of the configuration
	where $\lambda_i=\kappa_i+1$. There are four basic cases, see
	\Cref{fig:Schur_YB_proof}(b)-(e):
	\begin{itemize}
		\item
		      (b)
		      When $\mu_i>\lambda_i$ and $\mu_{i+1}<\kappa_{i}$, the
		      cross vertex is updated as $\oiYBio{.25}{3.5}{10.5pt}
			      \rightsquigarrow\oiYBoi{.25}{3.5}{10.5pt}
			      \rightsquigarrow\iiYBii{.25}{3.5}{10.5pt}$. This removes the arrow at
		      $\kappa_i$ and corresponds to $U(\mu\to\nu\mid
			      \kappa\to\lambda)=\mathbf{1}_{\nu=\mu}$.
		\item
		      (c)
		      When $\mu_i=\lambda_i=\kappa_i+1<\lambda_{i-1}$, the
		      update is $\oiYBio{.25}{3.5}{10.5pt}
			      \rightsquigarrow\oiYBoi{.25}{3.5}{10.5pt}
			      \rightsquigarrow\ioYBio{.25}{3.5}{10.5pt}
			      \rightsquigarrow\oiYBio{.25}{3.5}{10.5pt}$, which removes the arrow at
		      $\kappa_i$ and places a new arrow (corresponding to $\nu_i$ after the update)
		      at $\lambda_i+1$, which corresponds to the push under conditions
		      \eqref{YB_Schur_cont_field_exotic_push}.
		\item
		      (d) When $\mu_i=\lambda_i=\kappa_i+1=\lambda_{i-1}$, the
		      update is $\oiYBio{.25}{3.5}{10.5pt}
			      \rightsquigarrow\oiYBoi{.25}{3.5}{10.5pt}
			      \rightsquigarrow\iiYBii{.25}{3.5}{10.5pt}$, which removes the arrow at
		      $\kappa_i$, and does not affect the arrow at $\kappa_{i-1}=\lambda_{i-1}$
		      which becomes $\nu_i=\mu_i$. This case violates of
		      \eqref{YB_Schur_cont_field_exotic_push}, and thus the update rule is
		      $U(\mu\to\nu\mid \kappa\to\lambda)=\mathbf{1}_{\nu=\mu}$.
		\item
		      (e)
		      When $\mu_i<\lambda_i$ (and necessarily
		      $\mu_i=\lambda_i+1$), the update is $\oiYBio{.25}{3.5}{10.5pt}
			      \rightsquigarrow\ooYBoo{.25}{3.5}{10.5pt}
			      \rightsquigarrow\oiYBio{.25}{3.5}{10.5pt}$, which removes an arrow at
		      $\kappa_i$ and adds a new arrow for $\nu_i$ at $\lambda_i$. This corresponds
		      to the mandatory pushing to restore interlacing.
	\end{itemize}
	%\begin{noindent}
	Each of the cases (c)-(e) admits a slight variation when $\mu_{i+1}=\kappa_i>\kappa_{i+1}$. Then in
	the update of the cross vertex state the initial state is
	$\iiYBii{.25}{3.5}{10.5pt}$ instead of $\oiYBio{.25}{3.5}{10.5pt}$, but the
	rows of the table in \Cref{fig:fwd_Schur_cont} corresponding to these two
	states are the same up to $O(v^{-1})$. There is also another variation of (e)
	when $\lambda_{i-1}=\kappa_{i-1}=\lambda_i<\mu_{i-1}$, in which case the update is
	$\oiYBio{.25}{3.5}{10.5pt}\rightsquigarrow\ooYBoo{.25}{3.5}{10.5pt}
	\rightsquigarrow\iiYBii{.25}{3.5}{10.5pt}$. This does not remove an arrow
	at $\kappa_{i-1}$ which becomes $\nu_i$ after the passing of the cross, and this
	agrees with the mandatory pushing. This completes the proof.
	%\end{noindent}
\end{proof}
\begin{remark}
	\label{rmk:YB_Schur_directly_check}
	One can directly check that the Yang-Baxter dynamics on interlacing
	arrays described in \Cref{prop:YB_Schur_cont_field} in the language of
	interlacing arrays satisfies equation (2.20) of \cite{BorodinPetrov2013NN}.
	This equation implies that the dynamics acts nicely on Schur processes (i.e.,
	in agreement with \Cref{thm:YB_field_spin_HL_process}). However, after
	establishing \Cref{prop:YB_Schur_cont_field}, this fact also follows as a
	degeneration of \Cref{thm:YB_field_spin_HL_process}.
\end{remark}

The dynamics of \Cref{prop:YB_Schur_cont_field} is very similar to the one
constructed in \cite{BorFerr2008DF} using an idea of coupling Markov chains
from \cite{DiaconisFill1990}. Namely, in the latter dynamics the absence of
independent jumps and additional pushing in the special cases
\eqref{YB_Schur_cont_field_exotic_jump},
\eqref{YB_Schur_cont_field_exotic_push} are eliminated. In other words, in the
dynamics of \cite{BorFerr2008DF} every particle simply jumps to the right by
one at rate $u$ while obeying the blocking and the mandatory pushing rules.

On the other hand, the Hall-Littlewood RSK field introduced in
\cite{BufetovMatveev2017} in the half-continuous $t=0$ limit turns into a
continuous time dynamics on interlacing arrays related to the column insertion
Robinson-Schensted-Knuth (RSK) correspondence. In this dynamics, only the
leftmost particles $\lambda^{(j)}_j$ can independently jump. At the same time,
each move (to the right by one) of a particle $\lambda^{(j)}_i$ triggers a
move of a particle to the right of it on the upper level. Typically, this
triggered particle is $\lambda^{(j+1)}_i$, but the move is donated to the
right if it is blocked. We refer to \cite[Section 7]{BorodinPetrov2013NN} for
a detailed description of this dynamics related to the (column) RSK.

Since this RSK dynamics differs from the one coming from the Yang-Baxter field
via \Cref{prop:YB_Schur_cont_field}, we see that the Hall-Littlewood RSK field
of \cite{BufetovMatveev2017} also differs from the $s=0$ Yang-Baxter field of
\Cref{sub:degen_HL}.

\subsection{Half-continuous Hall-Littlewood degeneration with rescaling}
\label{sub:hc_degen_IHL}

Renaming $u=-su$ and slowing the continuous time (equivalently, rescaling the
continuous horizontal direction in the vertex model language) by the factor
$(-s)$ makes the rates and probabilities in the half-continuous DS6V
independent of $s$. Then we can send $s\to0$ and obtain a well-defined dynamic
half-continuous vertex model. This model can be also obtained as a
half-continuous limit $v\to+\infty$ of the one described in
\Cref{sub:degen_IHL}. The resulting rates and probabilities for this model are
listed in \Cref{fig:dynamic_vertex_weights_degen_half_cont}(e).

\subsection{Half-continuous Schur degeneration with rescaling}
\label{sub:hc_degen_t0s0}

Further setting $t=0$ in the model of \Cref{sub:hc_degen_IHL} turns the rates
and probabilities into the ones in
\Cref{fig:dynamic_vertex_weights_degen_half_cont}(f). Via a simple time
rescaling, this model becomes the same as the modified continuous time
PushTASEP considered in \Cref{sub:hc_degen_t0}.

\subsection{Rational limit $t\to1$}
\label{sub:degen_rational}

In this and the following subsections we return to the original DS6V weights
as in \Cref{fig:L_dynamicS6V_probabilities}. Let us take limit $t\to1$ in
these weights, simultaneously rescaling all other parameters:
\begin{equation*}
	t=e^{\varepsilon},\qquad s=e^{\varepsilon\zeta}, \qquad
	u=e^{x\varepsilon},\qquad v=e^{-y\varepsilon}, \qquad \varepsilon\to0.
\end{equation*}
In this limit the vertex weights turn into the following:
\begin{align}
	\nonumber
	%\begin{noindent}
	\biggl[\scalebox{.7}{
			\begin{tikzpicture}[scale=1.2, very thick,
					baseline=-4pt]
				\draw[dashed] (0,-.6)--++(0,1.2);
				\draw[dashed] (-.6,0)--++(1.2,0);
				\node[anchor=north east] at (-.1,-.1)
				{$\ell$};
				\node[anchor=south east] at (-.1,.1)
				{$\ell+1$};
				\node[anchor=north west] at (.1,-.1)
				{$\ell-1$};
				\node[anchor=south west] at (.1,.1) {$\ell$};
				\draw[ultra thick,->] (-.6,0)--++(.55,0);
				\draw[ultra thick,->] (0,-.6)--++(0,.55);
				\draw[ultra thick,->] (0,0)--++(.55,0);
				\draw[ultra thick,->] (0,0)--++(0,.55);
			\end{tikzpicture}
	}\biggr] & =1,
	\qquad
	\biggl[\scalebox{.7}{
			\begin{tikzpicture}[scale=1.2, very thick,
					baseline=-4pt]
				\draw[dashed] (0,-.6)--++(0,1.2);
				\draw[dashed] (-.6,0)--++(1.2,0);
				\node[anchor=north east] at (-.1,-.1)
				{$\ell$};
				\node[anchor=south east] at (-.1,.1) {$\ell$};
				\node[anchor=north west] at (.1,-.1) {$\ell$};
				\node[anchor=south west] at (.1,.1) {$\ell$};
			\end{tikzpicture}
		}\biggr]=1,
	\\
	\biggl[\scalebox{.7}{
			\begin{tikzpicture}[scale=1.2, very thick,
					baseline=-4pt]
				\draw[dashed] (0,-.6)--++(0,1.2);
				\draw[dashed] (-.6,0)--++(1.2,0);
				\node[anchor=north east] at (-.1,-.1)
				{$\ell$};
				\node[anchor=south east] at (-.1,.1)
				{$\ell+1$};
				\node[anchor=north west] at (.1,-.1) {$\ell$};
				\node[anchor=south west] at (.1,.1) {$\ell$};
				\draw[ultra thick,->] (-.6,0)--++(.6,0);
				\draw[ultra thick,->] (0,0)--++(0,.6);
			\end{tikzpicture}
	}\biggr] & =\frac{\ell-x+\zeta}{(x+y+1)(\ell+y+\zeta)},
	\qquad \qquad
	\biggl[\scalebox{.7}{
			\begin{tikzpicture}[scale=1.2, very thick,
					baseline=-4pt]
				\draw[dashed] (0,-.6)--++(0,1.2);
				\draw[dashed] (-.6,0)--++(1.2,0);
				\node[anchor=north east] at (-.1,-.1)
				{$\ell$};
				\node[anchor=south east] at (-.1,.1)
				{$\ell+1$};
				\node[anchor=north west] at (.1,-.1) {$\ell$};
				\node[anchor=south west] at (.1,.1)
				{$\ell+1$};
				\draw[ultra thick,->] (-.6,0)--++(.6,0);
				\draw[ultra thick,->] (0,0)--++(.6,0);
			\end{tikzpicture}
		}\biggr]=
	\frac{(x+y)(\ell+y+\zeta+1)}{(x+y+1)(\ell+y+\zeta)},
	\label{rational_limit}
	\\
	\biggl[\scalebox{.7}{
			\begin{tikzpicture}[scale=1.2, very thick,
					baseline=-4pt]
				\draw[dashed] (0,-.6)--++(0,1.2);
				\draw[dashed] (-.6,0)--++(1.2,0);
				\node[anchor=north east] at (-.1,-.1)
				{$\ell$};
				\node[anchor=south east] at (-.1,.1) {$\ell$};
				\node[anchor=north west] at (.1,-.1)
				{$\ell-1$};
				\node[anchor=south west] at (.1,.1) {$\ell$};
				\draw[ultra thick,->] (0,-.6)--++(0,.6);
				\draw[ultra thick,->] (0,0)--++(.6,0);
			\end{tikzpicture}
	}\biggr] & =
	\frac{\ell+y+\zeta}{(x+y+1)(\ell-x+\zeta)}\,\mathbf{1}_{\ell\ge1},\qquad
	\biggl[\scalebox{.7}{
			\begin{tikzpicture}[scale=1.2, very thick,
					baseline=-4pt]
				\draw[dashed] (0,-.6)--++(0,1.2);
				\draw[dashed] (-.6,0)--++(1.2,0);
				\node[anchor=north east] at (-.1,-.1)
				{$\ell$};
				\node[anchor=south east] at (-.1,.1) {$\ell$};
				\node[anchor=north west] at (.1,-.1)
				{$\ell-1$};
				\node[anchor=south west] at (.1,.1)
				{$\ell-1$};
				\draw[ultra thick,->] (0,-.6)--++(0,.6);
				\draw[ultra thick,->] (0,0)--++(0,.6);
			\end{tikzpicture}
		}\biggr]=
	\frac{(x+y)(\ell-x+\zeta-1)}{(x+y+1)(\ell-x+\zeta)}\,\mathbf{1}_{\ell\ge1}
	.
	%\end{noindent}
	\nonumber
\end{align}
These weights are dynamic in the sense that they depend on the height function
$\ell$. Moreover, under certain restrictions on the parameters (for example,
if $x,y>0$ and $\zeta>x$), these weights are between $0$ and $1$ for all
$\ell\in \mathbb{Z}_{\ge0}$. Thus, the weights \eqref{rational_limit} define a
dynamic stochastic vertex model. Its height function is identified via
\Cref{cor:dyn6V_spin_HL_process} with an observable of a measure constructed
out of rational symmetric functions of \cite[Section 8.5]{Borodin2014vertex}.

The Hall-Littlewood case $(s=0)$ corresponds to setting $\zeta\to+\infty$ in
the weights \eqref{rational_limit}. This vertex model is no longer dynamic, it
has symmetric vertex weights (i.e., the probabilities for a path to turn right
or left are both equal to $1/(1+x+y)$) and can be regarded as a discrete time
version of the symmetric simple exclusion process (SSEP). One can thus say
that the limit $t\to1$ for $s=0$ corresponds to the transition from the XXZ to
the XXX model, and the model \eqref{rational_limit} can be regarded as a
dynamic version of SSEP/XXX.

\subsection{Limit to a dynamic version of ASEP}
\label{sub:degen_ASEP}

Here we consider a limit of DS6V to a continuous time particle system
generalizing the ASEP (Asymmetric Simple Exclusion Process). For the
stochastic six vertex model a limit to the usual ASEP was observed in
\cite{GwaSpohn1992}
(see also \cite{BCG6V} for details).

Recall that the spectral parameters of the DS6V weights satisfy $0\le u<v<1$.
Taylor expand the vertex weights as $v-u\to0$ (we omit the vertices
$(0,0;0,0)$ and $(1,1;1,1)$ always having weight $1$):
\begin{align*}
	%\begin{noindent}
	\biggl[\scalebox{.7}{
			\begin{tikzpicture}[scale=1.2, very thick,
					baseline=-4pt]
				\draw[dashed] (0,-.6)--++(0,1.2);
				\draw[dashed] (-.6,0)--++(1.2,0);
				\node[anchor=north east] at (-.1,-.1)
				{$\ell$};
				\node[anchor=south east] at (-.1,.1)
				{$\ell+1$};
				\node[anchor=north west] at (.1,-.1) {$\ell$};
				\node[anchor=south west] at (.1,.1) {$\ell$};
				\draw[ultra thick,->] (-.6,0)--++(.6,0);
				\draw[ultra thick,->] (0,0)--++(0,.6);
			\end{tikzpicture}
	}\biggr] & =
	1+O(v-u),
	\qquad
	\biggl[\scalebox{.7}{
			\begin{tikzpicture}[scale=1.2, very thick,
					baseline=-4pt]
				\draw[dashed] (0,-.6)--++(0,1.2);
				\draw[dashed] (-.6,0)--++(1.2,0);
				\node[anchor=north east] at (-.1,-.1)
				{$\ell$};
				\node[anchor=south east] at (-.1,.1)
				{$\ell+1$};
				\node[anchor=north west] at (.1,-.1) {$\ell$};
				\node[anchor=south west] at (.1,.1)
				{$\ell+1$};
				\draw[ultra thick,->] (-.6,0)--++(.6,0);
				\draw[ultra thick,->] (0,0)--++(.6,0);
			\end{tikzpicture}
		}\biggr]=
		\frac{(v-u) (u-s t^{\ell+1})}{(1-t) u (u-s t^\ell)}+O(v-u)^2,
	\\
	\biggl[\scalebox{.7}{
			\begin{tikzpicture}[scale=1.2, very thick,
					baseline=-4pt]
				\draw[dashed] (0,-.6)--++(0,1.2);
				\draw[dashed] (-.6,0)--++(1.2,0);
				\node[anchor=north east] at (-.1,-.1)
				{$\ell$};
				\node[anchor=south east] at (-.1,.1) {$\ell$};
				\node[anchor=north west] at (.1,-.1)
				{$\ell-1$};
				\node[anchor=south west] at (.1,.1) {$\ell$};
				\draw[ultra thick,->] (0,-.6)--++(0,.6);
				\draw[ultra thick,->] (0,0)--++(.6,0);
			\end{tikzpicture}
	}\biggr] & =
	1+O(v-u)
	,
	\qquad 
	\biggl[\scalebox{.7}{
			\begin{tikzpicture}[scale=1.2, very thick,
					baseline=-4pt]
				\draw[dashed] (0,-.6)--++(0,1.2);
				\draw[dashed] (-.6,0)--++(1.2,0);
				\node[anchor=north east] at (-.1,-.1)
				{$\ell$};
				\node[anchor=south east] at (-.1,.1) {$\ell$};
				\node[anchor=north west] at (.1,-.1)
				{$\ell-1$};
				\node[anchor=south west] at (.1,.1)
				{$\ell-1$};
				\draw[ultra thick,->] (0,-.6)--++(0,.6);
				\draw[ultra thick,->] (0,0)--++(0,.6);
			\end{tikzpicture}
		}\biggr]=
		\frac{(v-u)t (u-s t^{\ell-1})}{(1-t) u (u-s t^\ell)}
		+O(v-u)^2
	.
	%\end{noindent}
\end{align*}
We thus see that the up-right lattice paths perform staircase like movements
most of the time. Occasionally, however, these staircases move up or down
according to the weights of the vertices $(1,0;1,0)$ and $(0,1;0,1)$,
respectively. Subtracting the staircase movement, rescaling the vertical
direction by the factor of $\frac{v-u}{u(1-t)}$, and interpreting it as time
leads to the following continuous time particle system on $\mathbb{Z}$.

The particles are ordered as $y_1(\tau)>y_2(\tau)>\ldots $, and at most one
particle per site is allowed. The six vertex boundary condition translates
into the step initial condition $y_i(0)=-i$, $i\ge1$. In continuous time, each
particle $y_{\ell}$ tries to jump to the right by one at rate
$\dfrac{u-st^{\ell}}{u-st^{\ell-1}}$, and to the left by one at rate
$t\,\dfrac{u-st^{\ell-1}}{u-st^{\ell}}$. If the destination is occupied, the
corresponding jump is blocked and $y_{\ell}$ does not move. See
\Cref{fig:dyn_ASEP} in the Introduction. Thus, one can say that our dynamic
ASEP is a generalization of the ASEP with certain particle-dependent jump
rates. The connection to spin Hall-Littlewood measures might provide tools for
asymptotic analysis of this model.

The dynamic version of the ASEP obtained above is somewhat similar to the one
of \cite{borodin2017elliptic}, \cite{BorodinCorwin2017dynamic} coming from
vertex models at elliptic level. However, these two models are different. In
particular, in our model the dynamic dependence on the height function is via
the quantities $h_x=\#\{\textnormal{number of particles to the right of
		$x$}\}$, while in \cite{borodin2017elliptic}, \cite{BorodinCorwin2017dynamic}
the dynamic parameter is $s_x=2h_x+x$ which incorporates both the particle's
number and location.

\subsection{Finite vertical spin}
\label{sub:degen_finite_spin}

Setting $s=t^{-I/2}$, where $I\in \mathbb{Z}_{\ge1}$, turns the vertical
representation giving rise to the vertex weights \eqref{vertex_weights} into a
spin $\frac{I}{2}$ one. This gives rise to a vertex model with at most $I$
vertical arrows per edge allowed. Let us briefly discuss what this means for
the main constructions of the present paper. For simplicity, we only consider
the case $I=1$ when the higher spin six vertex model turns into the six vertex
model.

Call a signature $\lambda\in \mathsf{Sign}_N$ \emph{strict} if
$\lambda_1>\lambda_2>\ldots>\lambda_N $. Observe that for $s=t^{-\frac{1}{2}}$
the weight
\begin{equation*}
	\Big[
	\Voi{.6}g{g-1}{-4.5pt} \Big]_{u}=
	\dfrac{(1-t^{g-2})u}{1-ut^{-\frac{1}{2}}}
\end{equation*}
vanishes for $g=2$. Thus, $G_{\mu/\nu}^c(v)$ also vanishes if $\nu$ is strict
and $\mu$ is not, see \Cref{ssub:G_definition}. At the same time the function
$G^{c}_{\lambda/(0^N)}$ entering the spin Hall-Littlewood measure
\eqref{spin_HL_measure} is not well-defined since $(0^{N})$ is not strict.
This presents an obstacle in degenerating spin Hall-Littlewood measures and
processes to $s=t^{-\frac{1}{2}}$ in a straightforward way.

On the other hand, the vertex weights for $s=t^{-\frac{1}{2}}$ satisfy a
Yang-Baxter equation, and bijectivisation can be applied to it, too. Following
the lines of \Cref{sec:local_transition_probabilities}, one can define forward
transition probabilities
$\mathsf{U}^{\mathrm{fwd}}(\kappa\to\nu\mid\lambda,\mu)$, where
$\kappa,\lambda\in\mathsf{Sign}_{N-1}$ and $\mu,\nu\in\mathsf{Sign}_{N}$ are
strict. Using these probabilities, it is possible to define an analogue of the
Yang-Baxter field $\boldsymbol\lambda^{(x,y)}$, $x,y\in\mathbb{Z}_{\ge0}$,
with boundary conditions $\boldsymbol\lambda^{(x,0)}=\varnothing$,
$\boldsymbol\lambda^{(0,y)}=(-1,-2,\ldots,-y )$. It is not clear whether this
version of the Yang-Baxter field leads via Markov projections to an analogue
of the dynamic stochastic six vertex model of \Cref{sec:dynamicS6V}, and we do
not discuss this issue here.

\section{Yang-Baxter equation}
\label{app:YB_equation}

Here we write out all the explicit identities between rational functions which
comprise the Yang-Baxter equation. This equation states that certain
combinations of vertex weights \eqref{vertex_weights},
\eqref{cross_vertex_weights} are equal to each other. Writing all possible
cases out we arrive at the following 16 identities. For better notation, in
the vertex weights we put cross vertices together with pairs of vertices, and
use the shorthand
\begin{equation}
	\label{YB_spectral_parameters_convention}
	\left[ \cdots \right]:=\left[ \cdots \right]_{u,v},
	\qquad
	\left[ \cdots \right]':=\left[ \cdots \right]_{v,u}
\end{equation}
for the vertex weights. Moreover, by agreement, the weights of the cross
vertices are not affected by the swapping of spectral parameters, and are
given by \eqref{cross_vertex_weights} in both sides of each of the identities.

Below are all the 16 identities comprising the Yang-Baxter equation. They
depend on an arbitrary nonnegative integer $g$ subject to the agreement that
once an arrow configuration in either side of a formula contains $g-1$ or
$g-2$, we assume that $g\ge1$ or $g\ge2$, respectively. Each of the identities
below is readily verified by hand:
\begin{align}
	\label{YB1.1}\tag{YB1.1} &
	\biggl[
	\ooYBoo{.3}{3}{13.5pt}\ooWoo{.6}ggg{6.75pt}
	\biggr]
	=
	\tfrac{(1-st^gu)(1-st^gv)}{(1-su)(1-sv)}
	=
	\biggl[
	\ooWoo{.6}ggg{6.75pt}
	\ooYBoo{.3}{3}{13.5pt}
	\biggr]'
	;                          \\
	\label{YB1.2}\tag{YB1.2}
	\begin{split}
		&\biggl[
		\ooYBoo{.3}{3}{13.5pt}\ooWio{.6}g{g-1}{g-1}{6.75pt}
		\biggr]
		=
		\tfrac{(1-s^2t^{g-1})u(1-st^{g-1}v)}{(1-su)(1-sv)}
		\\[-8pt]&\hspace{5pt}=
		\tfrac{(1-s^2t^{g-1})v(1-st^{g-1}u)}{(1-sv)(1-su)}
		\tfrac{(1-t)u}{u-tv}+
		\tfrac{(1-st^gv)(1-s^2t^{g-1})u}{(1-sv)(1-su)}
		\tfrac{u-v}{u-tv}
		=
		\biggl[
		\ooWio{.6}g{g-1}{g-1}{6.75pt}\ioYBio{.3}{3}{13.5pt}
		\biggr]'+
		\biggl[
		\ooWoi{.6}g{g}{g-1}{6.75pt}\oiYBio{.3}{3}{13.5pt}
		\biggr]';
	\end{split}
	\\
	\label{YB1.3}\tag{YB1.3}
	\begin{split}
		&\biggl[
		\ooYBoo{.3}{3}{13.5pt}\ooWoi{.6}g{g}{g-1}{6.75pt}
		\biggr]
		=
		\tfrac{(1-st^gu)(1-s^2t^{g-1})v}{(1-su)(1-sv)}
		\\[-8pt]&\hspace{5pt}=
		\tfrac{(1-st^gv)(1-s^2t^{g-1})u}{(1-sv)(1-su)}
		\tfrac{(1-t)v}{u-tv}+
		\tfrac{(1-s^2t^{g-1})v(1-st^{g-1}u)}{(1-sv)(1-su)}
		\tfrac{t(u-v)}{u-tv}
		=
		\biggl[
		\ooWoi{.6}g{g}{g-1}{6.75pt}\oiYBoi{.3}{3}{13.5pt}
		\biggr]'+
		\biggl[
		\ooWio{.6}g{g-1}{g-1}{6.75pt}\ioYBoi{.3}{3}{13.5pt}
		\biggr]'
		;
	\end{split}
	\\
	\label{YB1.4}\tag{YB1.4} &
	\biggl[
	\ooYBoo{.3}{3}{13.5pt}\ooWii{.6}g{g-1}{g-2}{6.75pt}
	\biggr]
	=
	\tfrac{(1-s^2t^{g-1})u(1-s^2t^{g-2})v}{(1-su)(1-sv)}
	=
	\biggl[
	\ooWii{.6}g{g-1}{g-2}{6.75pt}\iiYBii{.3}{3}{13.5pt}
	\biggr]'
	;                          \\
	\label{YB2.1}\tag{YB2.1}
	\begin{split}
		&
		\biggl[
		\ioYBio{.3}{3}{13.5pt}\oiWoo{.6}g{g+1}{g+1}{6.75pt}
		\biggr]
		+
		\biggl[
		\ioYBoi{.3}{3}{13.5pt}\ioWoo{.6}gg{g+1}{6.75pt}
		\biggr]
		=
		\tfrac{(1-t)u}{u-tv}
		\tfrac{(1-t^{g+1})(1-st^{g+1}v)}{(1-su)(1-sv)}
		+
		\tfrac{t(u-v)}{u-tv}
		\tfrac{(1-st^gu)(1-t^{g+1})}{(1-su)(1-sv)}
		\\[-8pt]&
		\hspace{5pt}
		=
		\tfrac{(1-t^{g+1})(1-st^{g+1}u)}{(1-sv)(1-su)}
		=
		\biggl[
		\oiWoo{.6}g{g+1}{g+1}{6.75pt}\ooYBoo{.3}{3}{13.5pt}
		\biggr]'
		;
	\end{split}
	\\
	\label{YB2.2}\tag{YB2.2}
	\begin{split}
		&
		\biggl[
		\ioYBio{.3}{3}{13.5pt}\oiWio{.6}ggg{6.75pt}
		\biggr]
		+
		\biggl[
		\ioYBoi{.3}{3}{13.5pt}\ioWio{.6}g{g-1}g{6.75pt}
		\biggr]
		=
		\tfrac{(1-t)u}{u-tv}
		\tfrac{(u-st^g)(1-st^g v)}{(1-su)(1-sv)}
		+
		\tfrac{t(u-v)}{u-tv}
		\tfrac{(1-s^2t^{g-1})u(1-t^g)}{(1-su)(1-sv)}
		\\[-8pt]&
		\hspace{5pt}
		=
		\tfrac{(v-st^g)(1-st^gu)}{(1-sv)(1-su)}
		\tfrac{(1-t)u}{u-tv}
		+
		\tfrac{(1-t^{g+1})(1-s^2t^g)u}{(1-sv)(1-su)}
		\tfrac{u-v}{u-tv}
		=
		\biggl[
		\oiWio{.6}g{g}{g}{6.75pt}\ioYBio{.3}{3}{13.5pt}
		\biggr]'
		+
		\biggl[
		\oiWoi{.6}g{g+1}{g}{6.75pt}\oiYBio{.3}{3}{13.5pt}
		\biggr]'
		;
	\end{split}
	\\
	\label{YB2.3}\tag{YB2.3}
	\begin{split}
		&
		\biggl[
		\ioYBio{.3}{3}{13.5pt}\oiWoi{.6}g{g+1}g{6.75pt}
		\biggr]
		+
		\biggl[
		\ioYBoi{.3}{3}{13.5pt}\ioWoi{.6}ggg{6.75pt}
		\biggr]
		=
		\tfrac{(1-t)u}{u-tv}
		\tfrac{(1-t^{g+1})(1-s^2t^g)v}{(1-su)(1-sv)}
		+
		\tfrac{t(u-v)}{u-tv}
		\tfrac{(1-st^gu)(v-st^g)}{(1-su)(1-sv)}
		\\[-8pt]&
		\hspace{5pt}
		=
		\tfrac{(1-t^{g+1})(1-s^2t^g)u}{(1-sv)(1-su)}
		\tfrac{(1-t)v}{u-tv}
		+
		\tfrac{(v-st^g)(1-st^gu)}{(1-sv)(1-su)}
		\tfrac{t(u-v)}{u-tv}
		=
		\biggl[
		\oiWoi{.6}g{g+1}{g}{6.75pt}\oiYBoi{.3}{3}{13.5pt}
		\biggr]'
		+
		\biggl[
		\oiWio{.6}g{g}{g}{6.75pt}\ioYBoi{.3}{3}{13.5pt}
		\biggr]'
		;
	\end{split}
	\\
	\label{YB2.4}\tag{YB2.4}
	\begin{split}
		&
		\biggl[
		\ioYBio{.3}{3}{13.5pt}\oiWii{.6}gg{g-1}{6.75pt}
		\biggr]
		+
		\biggl[
		\ioYBoi{.3}{3}{13.5pt}\ioWii{.6}g{g-1}{g-1}{6.75pt}
		\biggr]
		=
		\tfrac{(1-t)u}{u-tv}
		\tfrac{(u-st^g)(1-s^2t^{g-1})v}{(1-su)(1-sv)}
		+
		\tfrac{t(u-v)}{u-tv}
		\tfrac{(1-s^2t^{g-1})u(v-st^{g-1})}{(1-su)(1-sv)}
		\\[-8pt]&
		\hspace{5pt}
		=
		\tfrac{(v-st^g)(1-s^2t^{g-1})u}{(1-sv)(1-su)}
		=
		\biggl[
		\oiWii{.6}g{g}{g-1}{6.75pt}\iiYBii{.3}{3}{13.5pt}
		\biggr]'
		;
	\end{split}
	\\
	\label{YB3.1}\tag{YB3.1}
	\begin{split}
		&
		\biggl[
		\oiYBoi{.3}{3}{13.5pt}\ioWoo{.6}gg{g+1}{6.75pt}
		\biggr]
		+
		\biggl[
		\oiYBio{.3}{3}{13.5pt}\oiWoo{.6}g{g+1}{g+1}{6.75pt}
		\biggr]
		=
		\tfrac{(1-t)v}{u-tv}
		\tfrac{(1-st^gu)(1-t^{g+1})}{(1-su)(1-sv)}
		+
		\tfrac{u-v}{u-tv}
		\tfrac{(1-t^{g+1})(1-st^{g+1}v)}{(1-su)(1-sv)}
		\\[-8pt]&\hspace{5pt}=
		\tfrac{(1-st^gv)(1-t^{g+1})}{(1-sv)(1-su)}
		=
		\biggl[
		\ioWoo{.6}gg{g+1}{6.75pt}\ooYBoo{.3}{3}{13.5pt}
		\biggr]'
		;
	\end{split}
	\\
	\label{YB3.2}\tag{YB3.2}
	\begin{split}
		&
		\biggl[
		\oiYBoi{.3}{3}{13.5pt}\ioWio{.6}g{g-1}g{6.75pt}
		\biggr]
		+
		\biggl[
		\oiYBio{.3}{3}{13.5pt}\oiWio{.6}ggg{6.75pt}
		\biggr]
		=
		\tfrac{(1-t)v}{u-tv}
		\tfrac{(1-s^2t^{g-1})u(1-t^g)}{(1-su)(1-sv)}
		+
		\tfrac{u-v}{u-tv}
		\tfrac{(u-st^g)(1-st^gv)}{(1-su)(1-sv)}
		\\[-8pt]&\hspace{5pt}=
		\tfrac{(1-s^2t^{g-1})v(1-t^g)}{(1-sv)(1-su)}
		\tfrac{(1-t)u}{u-tv}
		+
		\tfrac{(1-st^gv)(u-st^g)}{(1-sv)(1-su)}
		\tfrac{u-v}{u-tv}
		=
		\biggl[
		\ioWio{.6}g{g-1}g{6.75pt}\ioYBio{.3}{3}{13.5pt}
		\biggr]'
		+
		\biggl[
		\ioWoi{.6}ggg{6.75pt}\oiYBio{.3}{3}{13.5pt}
		\biggr]'
		;
	\end{split}
	\\
	\label{YB3.3}\tag{YB3.3}
	\begin{split}
		&
		\biggl[
		\oiYBoi{.3}{3}{13.5pt}\ioWoi{.6}ggg{6.75pt}
		\biggr]
		+
		\biggl[
		\oiYBio{.3}{3}{13.5pt}\oiWoi{.6}g{g+1}g{6.75pt}
		\biggr]
		=
		\tfrac{(1-t)v}{u-tv}
		\tfrac{(1-st^gu)(v-st^g)}{(1-su)(1-sv)}
		+
		\tfrac{u-v}{u-tv}
		\tfrac{(1-t^{g+1})(1-s^2t^g)v}{(1-su)(1-sv)}
		\\[-8pt]&\hspace{5pt}=
		\tfrac{(1-st^gv)(u-st^g)}{(1-sv)(1-su)}
		\tfrac{(1-t)v}{u-tv}
		+
		\tfrac{(1-s^2t^{g-1})v(1-t^g)}{(1-sv)(1-su)}
		\tfrac{t(u-v)}{u-tv}
		=
		\biggl[
		\ioWoi{.6}g{g}g{6.75pt}\oiYBoi{.3}{3}{13.5pt}
		\biggr]'
		+
		\biggl[
		\ioWio{.6}g{g-1}g{6.75pt}\ioYBoi{.3}{3}{13.5pt}
		\biggr]'
		;
	\end{split}
	\\
	\label{YB3.4}\tag{YB3.4}
	\begin{split}
		&
		\biggl[
		\oiYBoi{.3}{3}{13.5pt}\ioWii{.6}g{g-1}{g-1}{6.75pt}
		\biggr]
		+
		\biggl[
		\oiYBio{.3}{3}{13.5pt}\oiWii{.6}gg{g-1}{6.75pt}
		\biggr]
		=
		\tfrac{(1-t)v}{u-tv}
		\tfrac{(1-s^2t^{g-1})u(v-st^{g-1})}{(1-su)(1-sv)}
		+
		\tfrac{u-v}{u-tv}
		\tfrac{(u-st^g)(1-s^2t^{g-1})v}{(1-su)(1-sv)}
		\\[-8pt]&\hspace{5pt}=
		\tfrac{(1-s^2t^{g-1})v(u-st^{g-1})}{(1-sv)(1-su)}
		=
		\biggl[
		\ioWii{.6}g{g-1}{g-1}{6.75pt}\iiYBii{.3}{3}{13.5pt}
		\biggr]'
		;
	\end{split}
	\\
	\label{YB4.1}\tag{YB4.1} &
	\biggl[
	\iiYBii{.3}{3}{13.5pt}\iiWoo{.6}g{g+1}{g+2}{6.75pt}
	\biggr]
	=
	\tfrac{(1-t^{g+1})(1-t^{g+2})}{(1-su)(1-sv)}
	=
	\biggl[
	\iiWoo{.6}g{g+1}{g+2}{6.75pt}\ooYBoo{.3}{3}{13.5pt}
	\biggr]'
	;                          \\
	\label{YB4.2}\tag{YB4.2}
	\begin{split}
		&
		\biggl[
		\iiYBii{.3}{3}{13.5pt}\iiWio{.6}gg{g+1}{6.75pt}
		\biggr]
		=
		\tfrac{(u-st^g)(1-t^{g+1})}{(1-su)(1-sv)}
		\\[-8pt]&\hspace{5pt}
		=
		\tfrac{(v-st^g)(1-t^{g+1})}{(1-sv)(1-su)}
		\tfrac{(1-t)u}{u-tv}
		+
		\tfrac{(1-t^{g+1})(u-st^{g+1})}{(1-sv)(1-su)}
		\tfrac{u-v}{u-tv}
		=
		\biggl[
		\iiWio{.6}g{g}{g+1}{6.75pt}\ioYBio{.3}{3}{13.5pt}
		\biggr]'
		+
		\biggl[
		\iiWoi{.6}g{g+1}{g+1}{6.75pt}\oiYBio{.3}{3}{13.5pt}
		\biggr]'
		;
	\end{split}
	\\
	\label{YB4.3}\tag{YB4.3}
	\begin{split}
		&
		\biggl[
		\iiYBii{.3}{3}{13.5pt}\iiWoi{.6}g{g+1}{g+1}{6.75pt}
		\biggr]
		=
		\tfrac{(1-t^{g+1})(v-st^{g+1})}{(1-su)(1-sv)}
		\\[-8pt]&\hspace{5pt}
		=
		\tfrac{(1-t^{g+1})(u-st^{g+1})}{(1-sv)(1-su)}
		\tfrac{(1-t)v}{u-tv}
		+
		\tfrac{(v-st^g)(1-t^{g+1})}{(1-sv)(1-su)}
		\tfrac{t(u-v)}{u-tv}
		=
		\biggl[
		\iiWoi{.6}g{g+1}{g+1}{6.75pt}\oiYBoi{.3}{3}{13.5pt}
		\biggr]'
		+
		\biggl[
		\iiWio{.6}g{g}{g+1}{6.75pt}\ioYBoi{.3}{3}{13.5pt}
		\biggr]'
		;
	\end{split}
	\\
	\label{YB4.4}\tag{YB4.4} &
	\biggl[
	\iiYBii{.3}{3}{13.5pt}\iiWii{.6}g{g}{g}{6.75pt}
	\biggr]
	=
	\tfrac{(u-st^g)(v-st^g)}{(1-su)(1-sv)}
	=
	\biggl[
	\iiWii{.6}g{g}{g}{6.75pt}\iiYBii{.3}{3}{13.5pt}
	\biggr]
	.
\end{align}

\section{Probabilities of forward and backward Yang-Baxter moves}
\label{app:YB_probabilities}

Here we list in full detail the probabilities of forward and backward
Yang-Baxter moves discussed in \Cref{sub:YB_bijectivisation_new_label}. These
probabilities (coming from identities \eqref{YB1.1}--\eqref{YB4.4} listed in
\Cref{app:YB_equation}) depend on the spectral parameters $u,v$ and on an
arbitrary nonnegative integer $g$ (which is required to be $\ge1$ if the
corresponding arrow configurations contain $g-1$ vertical arrows).

Equation numbers \eqref{F1.1}--\eqref{F4.4} and \eqref{B1.1}--\eqref{B4.4}
below correspond to numbers of the Yang-Baxter identities in
\Cref{app:YB_equation} whose bijectivisation gives these forward and backward
transition probabilities. The forward transition probabilities look as follows
(we do not write down transitions whose probabilities are identically zero):
\begin{align}
	%\begin{noindent}
	\label{F1.1}\tag{F1.1} &
		P^{\mathrm{fwd}}_{u,v}\biggl( \ooYBoo{.3}{3}{13.5pt}\ooWoo{.6}ggg{6.75pt}\ ,\
		\ooWoo{.6}ggg{6.75pt}\ooYBoo{.3}{3}{13.5pt} \biggr)
		=
		1;
	\\
	\label{F1.2}\tag{F1.2}
		&
		P^{\mathrm{fwd}}_{u,v}\biggl( \ooYBoo{.3}{3}{13.5pt}\ooWio{.6}{g+1}gg{6.75pt}\ ,\
		\ooWio{.6}{g+1}gg{6.75pt}\ioYBio{.3}{3}{13.5pt} \biggr)
		=
		1-P^{\mathrm{fwd}}_{u,v}\biggl( \ooYBoo{.3}{3}{13.5pt}\ooWio{.6}{g+1}gg{6.75pt}\ ,\
		\ooWoi{.6}{g+1}{g+1}{g}{6.75pt}\oiYBio{.3}{3}{13.5pt} \biggr)
		=
		\dfrac{(1-t)v}{u-tv}\dfrac{1-st^{g}u}{1-st^{g}v}
		;
	\\
	\label{F1.3}\tag{F1.3}
		&
		P^{\mathrm{fwd}}_{u,v}
		\biggl( \ooYBoo{.3}{3}{13.5pt}\ooWoi{.6}{g}g{g-1}{6.75pt}\ ,\
		\ooWoi{.6}{g}g{g-1}{6.75pt}\oiYBoi{.3}{3}{13.5pt} \biggr)
		=
		1-
		P^{\mathrm{fwd}}_{u,v}
		\biggl( \ooYBoo{.3}{3}{13.5pt}\ooWoi{.6}{g}g{g-1}{6.75pt}\ ,\
		\ooWio{.6}{g}{g-1}{g-1}{6.75pt}\ioYBoi{.3}{3}{13.5pt} \biggr)
		=
		\dfrac{(1-t)u}{u-tv}\dfrac{1-st^gv}{1-st^gu}
		;
	\\
	\label{F1.4}\tag{F1.4}
		&
		P^{\mathrm{fwd}}_{u,v}
		\biggl( \ooYBoo{.3}{3}{13.5pt}\ooWii{.6}{g+1}g{g-1}{6.75pt}\ ,\
		\ooWii{.6}{g+1}{g}{g-1}{6.75pt}\iiYBii{.3}{3}{13.5pt} \biggr)
		=1
		;
	\\
	\label{F2.1}\tag{F2.1}
		&
		P^{\mathrm{fwd}}_{u,v}
		\biggl( \ioYBio{.3}{3}{13.5pt}\oiWoo{.6}{g-1}g{g}{6.75pt}\ ,\
		\oiWoo{.6}{g-1}{g}{g}{6.75pt}\ooYBoo{.3}{3}{13.5pt} \biggr)
		=P^{\mathrm{fwd}}_{u,v}
		\biggl( \ioYBoi{.3}{3}{13.5pt}\ioWoo{.6}{g}g{g+1}{6.75pt}\ ,\
		\oiWoo{.6}{g}{g+1}{g+1}{6.75pt}\ooYBoo{.3}{3}{13.5pt} \biggr)
		=1
		;
	\\
	\label{F2.2}\tag{F2.2}
	\begin{split}
		&
		P^{\mathrm{fwd}}_{u,v}
		\biggl( \ioYBio{.3}{3}{13.5pt}\oiWio{.6}{g}g{g}{6.75pt}\ ,\
		\oiWio{.6}{g}{g}{g}{6.75pt}\ioYBio{.3}{3}{13.5pt} \biggr)
		=
		1-
		P^{\mathrm{fwd}}_{u,v}
		\biggl( \ioYBio{.3}{3}{13.5pt}\oiWio{.6}{g}g{g}{6.75pt}\ ,\
		\oiWoi{.6}{g}{g+1}{g}{6.75pt}\oiYBio{.3}{3}{13.5pt} \biggr)
		=
		\dfrac{v-st^g}{u-st^g}
		\dfrac{1-st^gu}{1-st^gv}
		;
		\\
		&
		P^{\mathrm{fwd}}_{u,v}
		\biggl( \ioYBoi{.3}{3}{13.5pt}\ioWio{.6}{g+1}g{g+1}{6.75pt}\ ,\
		\oiWoi{.6}{g+1}{g+2}{g+1}{6.75pt}\oiYBio{.3}{3}{13.5pt} \biggr)
		=
		1
		;
	\end{split}
	\\
	\label{F2.3}\tag{F2.3}
		&
		P^{\mathrm{fwd}}_{u,v}
		\biggl( \ioYBio{.3}{3}{13.5pt}\oiWoi{.6}{g-1}g{g-1}{6.75pt}\ ,\
		\oiWoi{.6}{g-1}{g}{g-1}{6.75pt}\oiYBoi{.3}{3}{13.5pt} \biggr)
		=P^{\mathrm{fwd}}_{u,v}
		\biggl( \ioYBoi{.3}{3}{13.5pt}\ioWoi{.6}{g}g{g}{6.75pt}\ ,\
		\oiWio{.6}{g}{g}{g}{6.75pt}\ioYBoi{.3}{3}{13.5pt} \biggr)
		=
		1
		;
	\\
	\label{F2.4}\tag{F2.4}
		&
		P^{\mathrm{fwd}}_{u,v}
		\biggl( \ioYBio{.3}{3}{13.5pt}\oiWii{.6}{g}g{g-1}{6.75pt}\ ,\
		\oiWii{.6}{g}{g}{g-1}{6.75pt}\iiYBii{.3}{3}{13.5pt} \biggr)
		=P^{\mathrm{fwd}}_{u,v}
		\biggl( \ioYBoi{.3}{3}{13.5pt}\ioWii{.6}{g+1}g{g}{6.75pt}\ ,\
		\oiWii{.6}{g+1}{g+1}{g}{6.75pt}\iiYBii{.3}{3}{13.5pt} \biggr)
		=1
		;
	\\
	\label{F3.1}\tag{F3.1}
		&
		P^{\mathrm{fwd}}_{u,v}
		\biggl( \oiYBio{.3}{3}{13.5pt}\oiWoo{.6}{g-1}g{g}{6.75pt}\ ,\
		\ioWoo{.6}{g-1}{g-1}{g}{6.75pt}\ooYBoo{.3}{3}{13.5pt} \biggr)
		=P^{\mathrm{fwd}}_{u,v}
		\biggl( \oiYBoi{.3}{3}{13.5pt}\ioWoo{.6}{g}g{g+1}{6.75pt}\ ,\
		\ioWoo{.6}{g}{g}{g+1}{6.75pt}\ooYBoo{.3}{3}{13.5pt} \biggr)
		=1
		;
	\\
	\label{F3.2}\tag{F3.2}
		&
		P^{\mathrm{fwd}}_{u,v}
		\biggl( \oiYBoi{.3}{3}{13.5pt}\ioWio{.6}{g+1}g{g+1}{6.75pt}\ ,\
		\ioWio{.6}{g+1}{g}{g+1}{6.75pt}\ioYBio{.3}{3}{13.5pt} \biggr)
		=
		P^{\mathrm{fwd}}_{u,v}
		\biggl( \oiYBio{.3}{3}{13.5pt}\oiWio{.6}{g}g{g}{6.75pt}\ ,\
		\ioWoi{.6}{g}{g}{g}{6.75pt}\oiYBio{.3}{3}{13.5pt} \biggr)
		=
		1
		;
	\\
	\label{F3.3}\tag{F3.3}
	\begin{split}
		&
		P^{\mathrm{fwd}}_{u,v}
		\biggl( \oiYBio{.3}{3}{13.5pt}\oiWoi{.6}{g-1}g{g-1}{6.75pt}\ ,\
		\ioWoi{.6}{g-1}{g-1}{g-1}{6.75pt}\oiYBoi{.3}{3}{13.5pt} \biggr)
		=
		1-
		P^{\mathrm{fwd}}_{u,v}
		\biggl( \oiYBio{.3}{3}{13.5pt}\oiWoi{.6}{g-1}g{g-1}{6.75pt}\ ,\
		\ioWio{.6}{g-1}{g-2}{g-1}{6.75pt}\ioYBoi{.3}{3}{13.5pt} \biggr)
		=
		\dfrac{1-t}{1-t^{g}}
		\dfrac{1-s^2t^{2g-2}}{1-s^2t^{g-1}}
		;
		\\
		&
		P^{\mathrm{fwd}}_{u,v}
		\biggl( \oiYBoi{.3}{3}{13.5pt}\ioWoi{.6}{g}g{g}{6.75pt}\ ,\
		\ioWoi{.6}{g}{g}{g}{6.75pt}\oiYBoi{.3}{3}{13.5pt} \biggr)
		=
		1
		;
	\end{split}
	\\
	\label{F3.4}\tag{F3.4}
		&
		P^{\mathrm{fwd}}_{u,v}
		\biggl( \oiYBio{.3}{3}{13.5pt}\oiWii{.6}{g}g{g-1}{6.75pt}\ ,\
		\ioWii{.6}{g}{g-1}{g-1}{6.75pt}\iiYBii{.3}{3}{13.5pt} \biggr)
		=
		P^{\mathrm{fwd}}_{u,v}
		\biggl( \oiYBoi{.3}{3}{13.5pt}\ioWii{.6}{g+1}g{g}{6.75pt}\ ,\
		\ioWii{.6}{g+1}{g}{g}{6.75pt}\iiYBii{.3}{3}{13.5pt} \biggr)
		=1
		;
	\\
	\label{F4.1}\tag{F4.1}
		&
		P^{\mathrm{fwd}}_{u,v}
		\biggl( \iiYBii{.3}{3}{13.5pt}\iiWoo{.6}{g-1}g{g+1}{6.75pt}\ ,\
		\iiWoo{.6}{g-1}{g}{g+1}{6.75pt}\ooYBoo{.3}{3}{13.5pt} \biggr)
		=1
		;
	\\
	\label{F4.2}\tag{F4.2}
		&
		P^{\mathrm{fwd}}_{u,v}
		\biggl( \iiYBii{.3}{3}{13.5pt}\iiWio{.6}{g}g{g+1}{6.75pt}\ ,\
		\iiWio{.6}{g}{g}{g+1}{6.75pt}\ioYBio{.3}{3}{13.5pt} \biggr)
		=
		1-
		P^{\mathrm{fwd}}_{u,v}
		\biggl( \iiYBii{.3}{3}{13.5pt}\iiWio{.6}{g}g{g+1}{6.75pt}\ ,\
		\iiWoi{.6}{g}{g+1}{g+1}{6.75pt}\oiYBio{.3}{3}{13.5pt} \biggr)
		=
		\dfrac{(1-t)u}{u-tv}\dfrac{v-st^g}{u-st^g}
		;
	\\
	\label{F4.3}\tag{F4.3}
		&
		P^{\mathrm{fwd}}_{u,v}
		\biggl( \iiYBii{.3}{3}{13.5pt}\iiWoi{.6}{g-1}g{g}{6.75pt}\ ,\
		\iiWoi{.6}{g-1}{g}{g}{6.75pt}\oiYBoi{.3}{3}{13.5pt} \biggr)
		=
		1-
		P^{\mathrm{fwd}}_{u,v}
		\biggl( \iiYBii{.3}{3}{13.5pt}\iiWoi{.6}{g-1}g{g}{6.75pt}\ ,\
		\iiWio{.6}{g-1}{g-1}{g}{6.75pt}\ioYBoi{.3}{3}{13.5pt} \biggr)
		=
		\dfrac{(1-t)v}{u-tv}
		\dfrac{u-st^{g}}{v-st^{g}};
	\\
	\label{F4.4}\tag{F4.4}
	&
		P^{\mathrm{fwd}}_{u,v}
		\biggl( \iiYBii{.3}{3}{13.5pt}\iiWii{.6}{g}g{g}{6.75pt}\ ,\
		\iiWii{.6}{g}{g}{g}{6.75pt}\iiYBii{.3}{3}{13.5pt} \biggr)
		=1.
	%\end{noindent}
\end{align}

The backward transition probabilities have the following form (again, we omit
transitions having zero probability):
\begin{align}
	%\begin{noindent}
	\label{B1.1}\tag{B1.1} &
		P^{\mathrm{bwd}}_{u,v}\biggl(
			\ooWoo{.6}ggg{6.75pt}\ooYBoo{.3}{3}{13.5pt}
			\ ,\
			\ooYBoo{.3}{3}{13.5pt}\ooWoo{.6}ggg{6.75pt}
		\biggr)
		=
		1
		;
	\\
	\label{B1.2}\tag{B1.2}
		&
		P^{\mathrm{bwd}}_{u,v}\biggl(
			\ooWio{.6}{g+1}gg{6.75pt}\ioYBio{.3}{3}{13.5pt}
			\ ,\
			\ooYBoo{.3}{3}{13.5pt}\ooWio{.6}{g+1}gg{6.75pt}
		\biggr)
		=
		P^{\mathrm{bwd}}_{u,v}\biggl(
			\ooWoi{.6}{g}{g}{g-1}{6.75pt}\oiYBio{.3}{3}{13.5pt}
			\ ,\
			\ooYBoo{.3}{3}{13.5pt}\ooWio{.6}{g}{g-1}{g-1}{6.75pt}
		\biggr)
		=
		1
		;
	\\
	\label{B1.3}\tag{B1.3}
		&
		P^{\mathrm{bwd}}_{u,v}
		\biggl(
			\ooWoi{.6}{g}g{g-1}{6.75pt}\oiYBoi{.3}{3}{13.5pt}
			\ ,\
			\ooYBoo{.3}{3}{13.5pt}\ooWoi{.6}{g}g{g-1}{6.75pt}
		\biggr)
		=
		P^{\mathrm{bwd}}_{u,v}
		\biggl(
			\ooWio{.6}{g+1}{g}{g}{6.75pt}\ioYBoi{.3}{3}{13.5pt}
			\ ,\
			\ooYBoo{.3}{3}{13.5pt}\ooWoi{.6}{g+1}{g+1}{g}{6.75pt}
		\biggr)
		=1
		;
	\\
	\label{B1.4}\tag{B1.4}
		&
		P^{\mathrm{bwd}}_{u,v}
		\biggl(
			\ooWii{.6}{g+1}{g}{g-1}{6.75pt}\iiYBii{.3}{3}{13.5pt}
			\ ,\
			\ooYBoo{.3}{3}{13.5pt}\ooWii{.6}{g+1}g{g-1}{6.75pt}
		\biggr)
		=1
		;
	\\
	\label{B2.1}\tag{B2.1}
		&
		P^{\mathrm{bwd}}_{u,v}
		\biggl(
			\oiWoo{.6}{g-1}{g}{g}{6.75pt}\ooYBoo{.3}{3}{13.5pt}
			\ ,\
			\ioYBio{.3}{3}{13.5pt}\oiWoo{.6}{g-1}g{g}{6.75pt}
		\biggr)
		=
		1-
		P^{\mathrm{bwd}}_{u,v}
		\biggl(
			\oiWoo{.6}{g-1}{g}{g}{6.75pt}\ooYBoo{.3}{3}{13.5pt}
			\ ,\
			\ioYBoi{.3}{3}{13.5pt}\ioWoo{.6}{g-1}{g-1}{g}{6.75pt}
		\biggr)
		=
		\dfrac{(1-t)u}{u-tv}
		\dfrac{1-st^gv}{1-st^gu}
		;
	\\
	\label{B2.2}\tag{B2.2}
	\begin{split}
		&
		P^{\mathrm{bwd}}_{u,v}
		\biggl(
			\oiWio{.6}{g}{g}{g}{6.75pt}\ioYBio{.3}{3}{13.5pt}
			\ ,\
			\ioYBio{.3}{3}{13.5pt}\oiWio{.6}{g}g{g}{6.75pt}
		\biggr)
		=1
		;
		\\
		&
		P^{\mathrm{bwd}}_{u,v}
		\biggl(
			\oiWoi{.6}{g-1}{g}{g-1}{6.75pt}\oiYBio{.3}{3}{13.5pt}
			\ ,\
			\ioYBio{.3}{3}{13.5pt}\oiWio{.6}{g-1}{g-1}{g-1}{6.75pt}
		\biggr)
		=
		1
		-
		P^{\mathrm{bwd}}_{u,v}
		\biggl(
			\oiWoi{.6}{g-1}{g}{g-1}{6.75pt}\oiYBio{.3}{3}{13.5pt}
			\ ,\
			\ioYBoi{.3}{3}{13.5pt}\ioWio{.6}{g-1}{g-2}{g-1}{6.75pt}
		\biggr)
		=
		\frac{1-t}{1-t^g}\frac{1-s^2t^{2g-2}}{1-s^2t^{g-1}}
		;
	\end{split}
	\\
	\label{B2.3}\tag{B2.3}
		&
		P^{\mathrm{bwd}}_{u,v}
		\biggl(
			\oiWoi{.6}{g-1}{g}{g-1}{6.75pt}\oiYBoi{.3}{3}{13.5pt}
			\ ,\
			\ioYBio{.3}{3}{13.5pt}\oiWoi{.6}{g-1}g{g-1}{6.75pt}
		\biggr)
		=P^{\mathrm{bwd}}_{u,v}
		\biggl(
			\oiWio{.6}{g}{g}{g}{6.75pt}\ioYBoi{.3}{3}{13.5pt}
			\ ,\
			\ioYBoi{.3}{3}{13.5pt}\ioWoi{.6}{g}g{g}{6.75pt}
		\biggr)
		=1
		;
	\\
	\label{B2.4}\tag{B2.4}
		&
		P^{\mathrm{bwd}}_{u,v}
		\biggl(
			\oiWii{.6}{g}{g}{g-1}{6.75pt}\iiYBii{.3}{3}{13.5pt}
			\ ,\
			\ioYBio{.3}{3}{13.5pt}\oiWii{.6}{g}g{g-1}{6.75pt}
		\biggr)
		=
		1-
		P^{\mathrm{bwd}}_{u,v}
		\biggl(
			\oiWii{.6}{g}{g}{g-1}{6.75pt}\iiYBii{.3}{3}{13.5pt}
			\ ,\
			\ioYBoi{.3}{3}{13.5pt}\ioWii{.6}{g}{g-1}{g-1}{6.75pt}
		\biggr)
		=\frac{(1-t)v}{u-tv}
		\frac{u-st^g}{v-st^g}
		;
	\\
	\label{B3.1}\tag{B3.1}
		&
		P^{\mathrm{bwd}}_{u,v}
		\biggl(
			\ioWoo{.6}{g}{g}{g+1}{6.75pt}\ooYBoo{.3}{3}{13.5pt}
			\ ,\
			\oiYBoi{.3}{3}{13.5pt}\ioWoo{.6}{g}g{g+1}{6.75pt}
		\biggr)
		=1-
		P^{\mathrm{bwd}}_{u,v}
		\biggl(
			\ioWoo{.6}{g}{g}{g+1}{6.75pt}\ooYBoo{.3}{3}{13.5pt}
			\ ,\
			\oiYBio{.3}{3}{13.5pt}\oiWoo{.6}{g}{g+1}{g+1}{6.75pt}
		\biggr)
		=\frac{(1-t)v}{u-tv}
		\frac{1-st^gu}{1-st^gv}
		;
	\\
	\label{B3.2}\tag{B3.2}
		&
		P^{\mathrm{bwd}}_{u,v}
		\biggl(
			\ioWio{.6}{g+1}{g}{g+1}{6.75pt}\ioYBio{.3}{3}{13.5pt}
			\ ,\
			\oiYBoi{.3}{3}{13.5pt}\ioWio{.6}{g+1}g{g+1}{6.75pt}
		\biggr)
		=
		P^{\mathrm{bwd}}_{u,v}
		\biggl(
			\ioWoi{.6}{g}{g}{g}{6.75pt}\oiYBio{.3}{3}{13.5pt}
			\ ,\
			\oiYBio{.3}{3}{13.5pt}\oiWio{.6}{g}g{g}{6.75pt}
		\biggr)
		=
		1
		;
	\\
	\label{B3.3}\tag{B3.3}
	\begin{split}
		&
		P^{\mathrm{bwd}}_{u,v}
		\biggl(
			\ioWio{.6}{g+1}{g}{g+1}{6.75pt}\ioYBoi{.3}{3}{13.5pt}
			\ ,\
			\oiYBio{.3}{3}{13.5pt}\oiWoi{.6}{g+1}{g+2}{g+1}{6.75pt}
		\biggr)
		=1
		;
	\\
		&
		P^{\mathrm{bwd}}_{u,v}
		\biggl(
			\ioWoi{.6}{g}{g}{g}{6.75pt}\oiYBoi{.3}{3}{13.5pt}
			\ ,\
			\oiYBoi{.3}{3}{13.5pt}\ioWoi{.6}{g}g{g}{6.75pt}
		\biggr)
		=
		1-
		P^{\mathrm{bwd}}_{u,v}
		\biggl(
			\ioWoi{.6}{g}{g}{g}{6.75pt}\oiYBoi{.3}{3}{13.5pt}
			\ ,\
			\oiYBio{.3}{3}{13.5pt}\oiWoi{.6}{g}{g+1}{g}{6.75pt}
		\biggr)
		=
		\frac{v-st^g}{u-st^g}
		\frac{1-st^gu}{1-st^gv}
		;
	\end{split}
	\\
	\label{B3.4}\tag{B3.4}
		&
		P^{\mathrm{bwd}}_{u,v}
		\biggl(
			\ioWii{.6}{g+1}{g}{g}{6.75pt}\iiYBii{.3}{3}{13.5pt}
			\ ,\
			\oiYBoi{.3}{3}{13.5pt}\ioWii{.6}{g+1}g{g}{6.75pt}
		\biggr)
		=
		1-
		P^{\mathrm{bwd}}_{u,v}
		\biggl(
			\ioWii{.6}{g+1}{g}{g}{6.75pt}\iiYBii{.3}{3}{13.5pt}
			\ ,\
			\oiYBio{.3}{3}{13.5pt}\oiWii{.6}{g+1}{g+1}{g}{6.75pt}
		\biggr)
		=
		\frac{(1-t)u}{u-tv}\frac{v-st^g}{u-st^g}
		;
	\\
	\label{B4.1}\tag{B4.1}
		&
		P^{\mathrm{bwd}}_{u,v}
		\biggl(
			\iiWoo{.6}{g-1}{g}{g+1}{6.75pt}\ooYBoo{.3}{3}{13.5pt}
			\ ,\
			\iiYBii{.3}{3}{13.5pt}\iiWoo{.6}{g-1}g{g+1}{6.75pt}
		\biggr)
		=1
		;
	\\
	\label{B4.2}\tag{B4.2}
		&
		P^{\mathrm{bwd}}_{u,v}
		\biggl(
			\iiWio{.6}{g}{g}{g+1}{6.75pt}\ioYBio{.3}{3}{13.5pt}
			\ ,\
			\iiYBii{.3}{3}{13.5pt}\iiWio{.6}{g}g{g+1}{6.75pt}
		\biggr)
		=
		P^{\mathrm{bwd}}_{u,v}
		\biggl(
			\iiWoi{.6}{g-1}{g}{g}{6.75pt}\oiYBio{.3}{3}{13.5pt}
			\ ,\
			\iiYBii{.3}{3}{13.5pt}\iiWio{.6}{g-1}{g-1}{g}{6.75pt}
		\biggr)
		=1
		;
	\\
	\label{B4.3}\tag{B4.3}
		&
		P^{\mathrm{bwd}}_{u,v}
		\biggl(
			\iiWoi{.6}{g-1}{g}{g}{6.75pt}\oiYBoi{.3}{3}{13.5pt}
			\ ,\
			\iiYBii{.3}{3}{13.5pt}\iiWoi{.6}{g-1}g{g}{6.75pt}
		\biggr)
		=
		P^{\mathrm{bwd}}_{u,v}
		\biggl(
			\iiWio{.6}{g}{g}{g+1}{6.75pt}\ioYBoi{.3}{3}{13.5pt}
			\ ,\
			\iiYBii{.3}{3}{13.5pt}\iiWoi{.6}{g}{g+1}{g+1}{6.75pt}
		\biggr)
		=1
		;
	\\
	\label{B4.4}\tag{B4.4}
	&
		P^{\mathrm{bwd}}_{u,v}
		\biggl(
			\iiWii{.6}{g}{g}{g}{6.75pt}\iiYBii{.3}{3}{13.5pt}
			\ ,\
			\iiYBii{.3}{3}{13.5pt}\iiWii{.6}{g}g{g}{6.75pt}
		\biggr)
		=1.
	%\end{noindent}
\end{align}

\section{Another form of the skew Cauchy identity}
\label{sub:another_Cauchy}

The spin Hall-Littlewood symmetric functions satisfy another form of Cauchy
identities which is worth mentioning. These identities involve the functions
$G^c$ (\Cref{ssub:G_definition}) along with the functions $G$. The latter are
variations of the $F$ functions (\Cref{ssub:F_definition}), the only
difference is that the boundary condition on the left as in in
\Cref{fig:connecting_interlacing} (left) is also empty. We refer to
\cite[Section 3]{Borodin2014vertex} for a detailed definition of the functions
$G$. Let us focus on the variant of the skew Cauchy identity with single
variables (analogue of \Cref{thm:skew_Cauchy_one}):
\begin{proposition}
	\label{prop:anotForm}
	Under assumption \eqref{condition_on_convergence}, let $\lambda, \mu
		\in \mathsf{Sign}_N$. We have
	\begin{equation}
		\label{eq:anotForm}
		\sum_{\kappa \in \mathsf{Sign}_N}
		G^c_{\lambda/\kappa}(v^{-1})G_{\mu/\kappa} (u) =\sum_{\nu \in
			\mathsf{Sign}_N} G^c_{\nu/\mu} (v^{-1}) G_{\nu/\lambda} (u).
	\end{equation}
\end{proposition}
\begin{proof}
	The proof is analogous to out proof of \Cref{thm:skew_Cauchy_one}
	presented in \Cref{sub:bijective_proof_skew_Cauchy}. The only difference is
	that we consider boundary conditions as in \Cref{fig:another_skew_Cauchy}
	instead of \Cref{fig:skew_Cauchy}. Namely, one defines the modified transition
	probabilities on signatures $\bar{\mathsf{U}}^{\mathrm{fwd}}_{v,u} (\kappa
		\to\nu \mid \lambda, \mu)$ ($\bar{\mathsf{U}}^{\mathrm{bwd}}_{v,u}
		(\nu\to\kappa	     \mid \lambda, \mu)$) obtained by dragging the cross
	vertex $\oiYBio{.25}{3.5}{10.5pt}$ from $-\infty$ to $+\infty$ (from $+\infty$
	to $-\infty$, respectively), proves an analog of
	\Cref{prop:reversibility_on_signatures} and obtains a bijective proof of
	\eqref{eq:anotForm}.
\end{proof}

\begin{figure}[htpb]
\centering
\begin{tikzpicture}[scale=.6, thick]
%\begin{noindent}
		\draw (-4.5,0)--++(10.5,0) node[below right] {$v$};
		\draw (-4.5,1)--++(10.5,0) node[above right] {$u$};
		\draw[densely dotted, line width=.4]
		(-4.5,2)--++(10.2,0);
		\draw[densely dotted, line width=.4]
		(-4.5,-1)--++(10.2,0);
		\foreach \ii in {-4,...,5}
			{
				\draw (\ii,1.1)--++(0,-.2);
				\draw (\ii,.1)--++(0,-.2)
				node[below, yshift=-25] {$\ii$};
				\draw[densely dotted, line width=.4]
				(\ii,-1.3)--++(0,3.6);
			}
		\node at (0.5,-1.1) {$\lambda$};
		\draw [line width=2,->] (-2,-1)--++(0,1);
		\draw [line width=2,->] (.9,-1)--++(0,1);
		\draw [line width=2,->] (1.1,-1)--++(0,1);
		\draw [line width=2,->] (3,-1)--++(0,1);
		\node at (0.5,2.1) {$\mu$};
		\draw [line width=2,->] (-1,1)--++(0,1);
		\draw [line width=2,->] (0,1)--++(0,1);
		\draw [line width=2,->] (1,1)--++(0,1);
		\draw [line width=2,->] (4,1)--++(0,1);
		\node at (.5,0.5) {$\kappa$};
		\draw [line width=2,->] (-3,0)--++(0,.9);
		\draw [line width=2,->] (0,0)--++(0,1);
		\draw [line width=2,->] (2,0)--++(0,1);
		\draw [line width=2,->] (-5,0)--++(1,0);
		\draw [line width=2,->] (-4,0)--++(1,0);
		\draw [line width=2,->] (-3,.9)--++(.1,.1)--++(.9,0);
		\draw [line width=2,->] (-2,1)--++(1,0);
		\draw [line width=2,->] (-2,0)--++(1,0);
		\draw [line width=2,->] (-1,0)--++(1,0);
		\draw [line width=2,->] (.9,0)--++(.1,.1)--++(0,.9);
		\draw [line width=2,->] (1.1,0)--++(.9,0);
		\draw [line width=2,->] (2,1)--++(1,0);
		\draw [line width=2,->] (3,1)--++(1,0);
		\draw [line width=2,->] (3,0)--++(1,0);
		\draw [line width=2,->] (4,0)--++(1,0);
		\draw [line width=2,->] (5,0)--++(1,0);
		\draw [line width=2,->] (6,0)--++(1,0);
	\end{tikzpicture}\qquad \quad
	\begin{tikzpicture}[scale=.6, thick]
		\draw (-4.5,0)--++(10.5,0) node[below right] {$u$};
		\draw (-4.5,1)--++(10.5,0) node[above right] {$v$};
		\draw[densely dotted, line width=.4]
		(-4.5,2)--++(10.2,0);
		\draw[densely dotted, line width=.4]
		(-4.5,-1)--++(10.2,0);
		\foreach \ii in {-4,...,5}
			{
				\draw (\ii,1.1)--++(0,-.2);
				\draw (\ii,.1)--++(0,-.2)
				node[below, yshift=-25] {$\ii$};
				\draw[densely dotted, line width=.4]
				(\ii,-1.3)--++(0,3.6);
			}
		\node at (0.5,-1.1) {$\lambda$};
		\draw [line width=2,->] (-2,-1)--++(0,.9);
		\draw [line width=2,->] (.9,-1)--++(0,1);
		\draw [line width=2,->] (1.1,-1)--++(0,1);
		\draw [line width=2,->] (3,-1)--++(0,1);
		\node at (0.5,2.1) {$\mu$};
		\draw [line width=2,->] (-1.1,1)--++(.1,.1)--++(0,.9);
		\draw [line width=2,->] (0,1)--++(0,1);
		\draw [line width=2,->] (1,1)--++(0,1);
		\draw [line width=2,->] (4,1)--++(0,1);
		\node at (.5,0.5) {$\nu$};
		\draw [line width=2,->] (-5,1)--++(1,0);
		\draw [line width=2,->] (-4,1)--++(1,0);
		\draw [line width=2,->] (-3,1)--++(1,0);
		\draw [line width=2,->] (-2,0)--++(1,0);
		\draw [line width=2,->] (-2,1)--++(1,0);
		\draw [line width=2,->] (-1,0)--++(0,.9);
		\draw [line width=2,->] (-1,.9)--++(.1,.1)--++(.9,0);
		\draw [line width=2,->] (.9,0)--++(.1,.1)--++(0,.9);
		\draw [line width=2,->] (1.1,0)--++(.9,0);
		\draw [line width=2,->] (2,0)--++(0,1);
		\draw [line width=2,->] (2,1)--++(1,0);
		\draw [line width=2,->] (3,1)--++(1,0);
		\draw [line width=2,->] (3,0)--++(1,0);
		\draw [line width=2,->] (4,0)--++(1,0);
		\draw [line width=2,->] (5,0)--++(0,1);
		\draw [line width=2,->] (5,1)--++(1,0);
		\draw [line width=2,->] (6,1)--++(1,0);
	\end{tikzpicture}
	\caption{Illustration of the sums in both sides of identity
		\eqref{eq:anotForm}.}
	\label{fig:another_skew_Cauchy}
	%\end{noindent}
\end{figure}
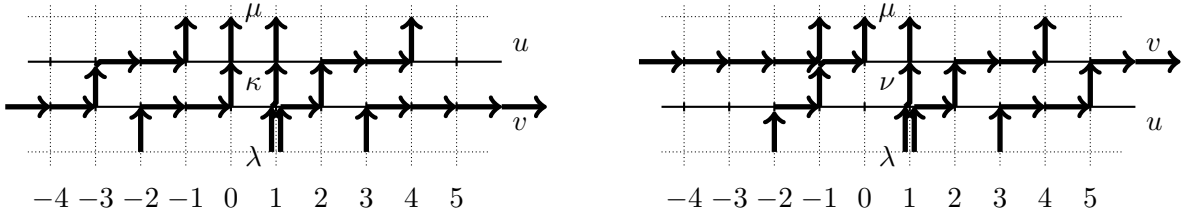

\Cref{prop:anotForm} is new. Its $s=0$ degeneration was mentioned
in \cite[Sections 3.1 and 3.7]{BufetovMatveev2017}. The significance of this
variation of the skew Cauchy identity is in the fact that it does not have any
prefactors, which is neat from the combinatorial point of view. Another
property which is better visible in this variation is a symmetry between
$\lambda$ and $\mu$:
\begin{proposition}
	\label{prop:symmetrF}
	Let $\bar{\mathsf{U}}^{\mathrm{fwd}}_{v,u}(\kappa \to \nu \mid
		\lambda, \mu)$ and
	$\bar{\mathsf{U}}^{\mathrm{bwd}}_{v,u}(\nu\to\kappa\mid\lambda,\mu)$ be
	transition probabilities defined in the proof of Proposition
	\ref{prop:anotForm}. We have
	\begin{equation*}
		\bar{\mathsf{U}}^{\mathrm{fwd}}_{v,u} (\kappa \to \nu \mid
		\lambda, \mu) =
		\bar{\mathsf{U}}^{\mathrm{fwd}}_{u^{-1},v^{-1}}(\kappa \to \nu
		\mid \mu, \lambda),
		\qquad
		\bar{\mathsf{U}}^{\mathrm{bwd}}_{v,u}
		(\nu \to \kappa \mid \lambda, \mu) =
		\bar{\mathsf{U}}^{\mathrm{bwd}}_{u^{-1},v^{-1}}
		(\nu \to \kappa \mid \mu, \lambda) .
	\end{equation*}
\end{proposition}
\begin{proof}
	Readily follows from Proposition 3.4.
\end{proof}

Note also that in the Hall-Littlewood case ($s=0$) \Cref{prop:symmetrF}
becomes fully symmetric:
\begin{equation*}
	\bar{\mathsf{U}}^{\mathrm{fwd}}_{v,u} (\kappa \to \nu \mid \lambda,
	\mu) =
	\bar{\mathsf{U}}^{\mathrm{fwd}}_{v,u}(\kappa \to \nu \mid \mu,
	\lambda),
	\qquad
	\bar{\mathsf{U}}^{\mathrm{bwd}}_{v,u}
	(\nu \to \kappa \mid \lambda, \mu) =
	\bar{\mathsf{U}}^{\mathrm{bwd}}_{v,u}
	(\nu \to \kappa \mid \mu, \lambda) .
\end{equation*}
Indeed, this is because the local transition probabilities
(\Cref{fig:fwd_YB,fig:bwd_YB}) are invariant under the swap $(u,v) \to
	(v^{-1},u^{-1})$ if $s=0$.

\section{Inhomogeneous modifications}
\label{app:inhomogeneous_construction}

Most constructions and results of the present paper can be generalized to
allow the spectral parameter $u$ and the spin parameter $s$ in the higher spin
weights \eqref{vertex_weights} vary along columns. Versions of the spin
Hall-Littlewood functions $F$ and $G^c$ with this type of inhomogeneity, as
well as Cauchy summation identities for these functions, are discussed in
detail in \cite{BorodinPetrov2016inhom}. Such Cauchy identities were employed
in that work to compute observables of the inhomogeneous stochastic higher
spin six vertex model which are amenable to asymptotic analysis (performed in,
e.g., \cite{BorodinPetrov2016Exp}).

Let us briefly discuss the modifications needed to introduce inhomogeneity
parameters into our constructions. These parameters form two families,
$\{\xi_i\}_{i\in\mathbb{Z}}$ and $\{s_i\}_{i\in\mathbb{Z}}$. The vertex
weights \eqref{vertex_weights} in the column number $i$ now depend on the
parameters $\xi_i u$ (replacing $u$) and $s_i$. These parameters $\xi_i,s_i$
do not enter the cross vertex weights \eqref{cross_vertex_weights} involved in
the Yang-Baxter equation. However, they do enter the local transition
probabilities $P^{\mathrm{fwd}},P^{\mathrm{bwd}}$: in the tables in
\Cref{fig:fwd_YB,fig:bwd_YB} one should replace the parameters $u,v,s$ with
$\xi_iu,\xi_iv,s_i$, respectively, where $i\in \mathbb{Z}$ is the location
through which the cross vertex is dragged.

Next, the definitions of the functions $F$ and $G^c$ should be modified as in
\cite{BorodinPetrov2016inhom}, by first replacing $(u,s)\to(\xi_m u,s_m)$ in
\eqref{F_skew_one_variable_definition} and
$(u^{-1},s)\to(\xi_m^{-1}u^{-1},s_m)$ in
\eqref{G_skew_one_variable_definition}, and then defining the multivariable
functions as in \Cref{ssub:F_G_multivar_definition}. Note that in Cauchy
identities (e.g., in \eqref{nonskew_multi_Cauchy}) the parameters in the
functions $F$ and $G^c$ should be $u_i\xi_m$ and $v_j^{-1}\xi_m^{-1}$,
respectively. Remarkably, the double product
$\prod\prod\frac{v_j-u_i}{v_j-tu_i}$ entering \eqref{nonskew_multi_Cauchy}
remains the same in the inhomogeneous setting.

Having inhomogeneous versions of the spin Hall-Littlewood functions $F$ and
$G^c$, one can define the corresponding measures and processes as in
\Cref{sub:spin_HL_measures_processes}. The local transition probabilities
assembled into $\mathsf{U}^{\mathrm{fwd}}_{v,u}$ and
$\mathsf{U}^{\mathrm{bwd}}_{v,u}$ thus give rise to an inhomogeneous version
of the Yang-Baxter field depending on $t$, the parameters $\{u_i\}$, $\{v_j\}$
as in \Cref{fig:YB_field}, and two series of inhomogeneous parameters
$\{\xi_m\}$ and $\{s_m\}$. The latter parameters may be thought of as
belonging to the third dimension in \Cref{fig:YB_field}, the one where the
signatures $\boldsymbol\lambda^{(x,y)}$ live.

The dynamic stochastic six vertex model (DS6V) arising in
\Cref{sec:dynamicS6V} as a Markov projection of the Yang-Baxter field onto the
column number zero does not feel the inhomogeneous parameters $\{\xi_m\}$ and
$\{s_m\}$ for $m\ge1$. This follows by the very construction of the
Yang-Baxter field using the probabilities $\mathsf{U}^{\mathrm{fwd}}_{v,u}$.
In other words:
\begin{corollary}
	\label{cor:independence_on_inhom_parameters}
	The distribution of the number of zero parts $\lambda^{[0]}$ under the
	inhomogeneous version of the spin Hall-Littlewood measure described above does
	not depend on the inhomogeneity parameters $\xi_m,s_m$ for $m\ge1$. A similar
	statement holds for spin Hall-Littlewood processes.
\end{corollary}

On the other hand, the parameters $\{u_i\}$, $\{v_j\}$ entering the
Yang-Baxter field, carry over to the DS6V model. The height function in this
inhomogeneous DS6V model is identified with $\lambda^{[0]}$ under a spin
Hall-Littlewood measure, in which the inhomogeneous parameters $u_i,v_j$ serve
as variables in the functions $F$ and $G^c$. See
\Cref{cor:dyn6V_spin_HL_process}. The presence of the inhomogeneous parameters
$\{u_i\}$ and $\{v_j\}$ carries over to most of the degenerations of the DS6V
model considered in \Cref{sec:degenerations}. An exception is the ASEP type
limit of \Cref{sub:degen_ASEP} since this limit is performed along the
diagonal of the quadrant.

\printbibliography

\end{document}

%% file: single_vertex_text.tex
\newcommand{\Voo}[4]{
	\scalebox{#1}{\begin{tikzpicture}
			[scale=1.5, ultra thick, baseline=#4]
		\draw[dotted] (-.5,0)--++(.4,0);
		\draw[dotted] (.1,0)--++(.4,0);
		\node at (0,.25) {\LARGE$#3$};
		\node at (0,-.25) {\LARGE$#2$};
	\end{tikzpicture}}}
\newcommand{\Vio}[4]{
	\scalebox{#1}{\begin{tikzpicture}
			[scale=1.5, ultra thick, baseline=#4]
		\draw[->] (-.5,0)--++(.4,0);
		\draw[dotted] (.1,0)--++(.4,0);
		\node at (0,.25) {\LARGE$#3$};
		\node at (0,-.25) {\LARGE$#2$};
	\end{tikzpicture}}}
\newcommand{\Voi}[4]{
	\scalebox{#1}{\begin{tikzpicture}
			[scale=1.5, ultra thick, baseline=#4]
		\draw[dotted] (-.5,0)--++(.4,0);
		\draw[->] (.1,0)--++(.4,0);
		\node at (0,.25) {\LARGE$#3$};
		\node at (0,-.25) {\LARGE$#2$};
	\end{tikzpicture}}}
\newcommand{\Vii}[4]{
	\scalebox{#1}{\begin{tikzpicture}
			[scale=1.5, ultra thick, baseline=#4]
		\draw[->] (-.5,0)--++(.4,0);
		\draw[->] (.1,0)--++(.4,0);
		\node at (0,.25) {\LARGE$#3$};
		\node at (0,-.25) {\LARGE$#2$};
	\end{tikzpicture}}}

\newcommand{\BVoo}[4]{
	\scalebox{#1}{\begin{tikzpicture}
			[scale=1.5, ultra thick, baseline=#4]
		\draw[dotted] (-.5,0)--++(.4,0);
		\draw[dotted] (.1,0)--++(.4,0);
		\node at (0,.25) {\LARGE$#3$};
		\node at (0,-.25) {\LARGE$#2$};
	\end{tikzpicture}}}
\newcommand{\BVio}[4]{
	\scalebox{#1}{\begin{tikzpicture}
			[scale=1.5, ultra thick, baseline=#4]
		\draw[<-] (-.5,0)--++(.4,0);
		\draw[dotted] (.1,0)--++(.4,0);
		\node at (0,.25) {\LARGE$#3$};
		\node at (0,-.25) {\LARGE$#2$};
	\end{tikzpicture}}}
\newcommand{\BVoi}[4]{
	\scalebox{#1}{\begin{tikzpicture}
			[scale=1.5, ultra thick, baseline=#4]
		\draw[dotted] (-.5,0)--++(.4,0);
		\draw[<-] (.1,0)--++(.4,0);
		\node at (0,.25) {\LARGE$#3$};
		\node at (0,-.25) {\LARGE$#2$};
	\end{tikzpicture}}}
\newcommand{\BVii}[4]{
	\scalebox{#1}{\begin{tikzpicture}
			[scale=1.5, ultra thick, baseline=#4]
		\draw[<-] (-.5,0)--++(.4,0);
		\draw[<-] (.1,0)--++(.4,0);
		\node at (0,.25) {\LARGE$#3$};
		\node at (0,-.25) {\LARGE$#2$};
	\end{tikzpicture}}}

%% file: YB_vertices.tex
\newcommand{\ooYBoo}[3]{\scalebox{#1}{\begin{tikzpicture}
[scale=1.5, baseline=#3]
\draw[line width=#2,dashed] (0,0)--++(1,1);
\draw[line width=#2,dashed] (0,1)--++(1,-1);
\end{tikzpicture}}}
\newcommand{\iiYBii}[3]{\scalebox{#1}{\begin{tikzpicture}
[scale=1.5, baseline=#3]
\draw[line width=#2,->] (0,0)--++(.47,.47);
\draw[line width=#2,->] (.5,.5)--++(.5,.5);
\draw[line width=#2,->] (0,1)--++(.47,-.47);
\draw[line width=#2,->] (.5,.5)--++(.5,-.5);
\end{tikzpicture}}}
\newcommand{\oiYBoi}[3]{\scalebox{#1}{\begin{tikzpicture}
[scale=1.5, baseline=#3]
\draw[line width=#2,dashed] (0,0)--++(.9,.9);
\draw[line width=#2,dashed] (0,1)--++(1,-1);
\draw[line width=#2,->] (.5,.5)--++(.5,.5);
\draw[line width=#2,->] (0,1)--++(.47,-.47);
\end{tikzpicture}}}
\newcommand{\oiYBio}[3]{\scalebox{#1}{\begin{tikzpicture}
[scale=1.5, baseline=#3]
\draw[line width=#2,dashed] (0,0)--++(1,1);
\draw[line width=#2,dashed] (0,1)--++(.9,-.9);
\draw[line width=#2,->] (0,1)--++(.47,-.47);
\draw[line width=#2,->] (.5,.5)--++(.5,-.5);
\end{tikzpicture}}}
\newcommand{\ioYBio}[3]{\scalebox{#1}{\begin{tikzpicture}
[scale=1.5, baseline=#3]
\draw[line width=#2,dashed] (0,0)--++(1,1);
\draw[line width=#2,dashed] (0,1)--++(.9,-.9);
\draw[line width=#2,->] (0,0)--++(.47,.47);
\draw[line width=#2,->] (.5,.5)--++(.5,-.5);
\end{tikzpicture}}}
\newcommand{\ioYBoi}[3]{\scalebox{#1}{\begin{tikzpicture}
[scale=1.5, baseline=#3]
\draw[line width=#2,dashed] (0,0)--++(.9,.9);
\draw[line width=#2,dashed] (0,1)--++(1,-1);
\draw[line width=#2,->] (0,0)--++(.47,.47);
\draw[line width=#2,->] (.5,.5)--++(.5,.5);
\end{tikzpicture}}}

%% file: six_vertex.tex
\def\epsx{.12}

\newcommand{\SVoioi}[2]{\scalebox{#1}{\begin{tikzpicture}
[scale=1, very thick, baseline=#2]
\draw[->] (-.5,0) -- (-\epsx/4,0);
\draw[->] (0,0) -- (.5,0);
\draw[dotted] (0,-.5) -- (0,.5);
\end{tikzpicture}}}

%% file: two_vertices_text.tex
\newcommand{\ooWoo}[5]{
\scalebox{#1}{\begin{tikzpicture} [scale=1.5, ultra thick, baseline=#5]
		\draw[dotted] (-.5,0)--++(.4,0);
		\draw[dotted] (-.5,.5)--++(.4,0);
		\draw[dotted] (.1,0)--++(.4,0);
		\draw[dotted] (.1,.5)--++(.4,0);
		\node at (0,-.25) {\LARGE$#2$};
		\node at (0,.25) {\LARGE$#3$};
		\node at (0,.75) {\LARGE$#4$};
	\end{tikzpicture}}}
\newcommand{\ooWoi}[5]{
\scalebox{#1}{\begin{tikzpicture} [scale=1.5, ultra thick, baseline=#5]
		\draw[dotted] (-.5,0)--++(.4,0);
		\draw[dotted] (-.5,.5)--++(.4,0);
		\draw[dotted] (.1,0)--++(.4,0);
		\draw[->] (.1,.5)--++(.4,0);
		\node at (0,-.25) {\LARGE$#2$};
		\node at (0,.25) {\LARGE$#3$};
		\node at (0,.75) {\LARGE$#4$};
	\end{tikzpicture}}}
\newcommand{\ooWio}[5]{
\scalebox{#1}{\begin{tikzpicture} [scale=1.5, ultra thick, baseline=#5]
		\draw[dotted] (-.5,0)--++(.4,0);
		\draw[dotted] (-.5,.5)--++(.4,0);
		\draw[->] (.1,0)--++(.4,0);
		\draw[dotted] (.1,.5)--++(.4,0);
		\node at (0,-.25) {\LARGE$#2$};
		\node at (0,.25) {\LARGE$#3$};
		\node at (0,.75) {\LARGE$#4$};
	\end{tikzpicture}}}
\newcommand{\ooWii}[5]{
\scalebox{#1}{\begin{tikzpicture} [scale=1.5, ultra thick, baseline=#5]
		\draw[dotted] (-.5,0)--++(.4,0);
		\draw[dotted] (-.5,.5)--++(.4,0);
		\draw[->] (.1,0)--++(.4,0);
		\draw[->] (.1,.5)--++(.4,0);
		\node at (0,-.25) {\LARGE$#2$};
		\node at (0,.25) {\LARGE$#3$};
		\node at (0,.75) {\LARGE$#4$};
	\end{tikzpicture}}}

\newcommand{\oiWoo}[5]{
\scalebox{#1}{\begin{tikzpicture} [scale=1.5, ultra thick, baseline=#5]
		\draw[->] (-.5,0)--++(.4,0);
		\draw[dotted] (-.5,.5)--++(.4,0);
		\draw[dotted] (.1,0)--++(.4,0);
		\draw[dotted] (.1,.5)--++(.4,0);
		\node at (0,-.25) {\LARGE$#2$};
		\node at (0,.25) {\LARGE$#3$};
		\node at (0,.75) {\LARGE$#4$};
	\end{tikzpicture}}}
\newcommand{\oiWoi}[5]{
\scalebox{#1}{\begin{tikzpicture} [scale=1.5, ultra thick, baseline=#5]
		\draw[->] (-.5,0)--++(.4,0);
		\draw[dotted] (-.5,.5)--++(.4,0);
		\draw[dotted] (.1,0)--++(.4,0);
		\draw[->] (.1,.5)--++(.4,0);
		\node at (0,-.25) {\LARGE$#2$};
		\node at (0,.25) {\LARGE$#3$};
		\node at (0,.75) {\LARGE$#4$};
	\end{tikzpicture}}}
\newcommand{\oiWio}[5]{
\scalebox{#1}{\begin{tikzpicture} [scale=1.5, ultra thick, baseline=#5]
		\draw[->] (-.5,0)--++(.4,0);
		\draw[dotted] (-.5,.5)--++(.4,0);
		\draw[->] (.1,0)--++(.4,0);
		\draw[dotted] (.1,.5)--++(.4,0);
		\node at (0,-.25) {\LARGE$#2$};
		\node at (0,.25) {\LARGE$#3$};
		\node at (0,.75) {\LARGE$#4$};
	\end{tikzpicture}}}
\newcommand{\oiWii}[5]{
\scalebox{#1}{\begin{tikzpicture} [scale=1.5, ultra thick, baseline=#5]
		\draw[->] (-.5,0)--++(.4,0);
		\draw[dotted] (-.5,.5)--++(.4,0);
		\draw[->] (.1,0)--++(.4,0);
		\draw[->] (.1,.5)--++(.4,0);
		\node at (0,-.25) {\LARGE$#2$};
		\node at (0,.25) {\LARGE$#3$};
		\node at (0,.75) {\LARGE$#4$};
	\end{tikzpicture}}}

\newcommand{\ioWoo}[5]{
\scalebox{#1}{\begin{tikzpicture} [scale=1.5, ultra thick, baseline=#5]
		\draw[dotted] (-.5,0)--++(.4,0);
		\draw[->] (-.5,.5)--++(.4,0);
		\draw[dotted] (.1,0)--++(.4,0);
		\draw[dotted] (.1,.5)--++(.4,0);
		\node at (0,-.25) {\LARGE$#2$};
		\node at (0,.25) {\LARGE$#3$};
		\node at (0,.75) {\LARGE$#4$};
	\end{tikzpicture}}}
\newcommand{\ioWoi}[5]{
\scalebox{#1}{\begin{tikzpicture} [scale=1.5, ultra thick, baseline=#5]
		\draw[dotted] (-.5,0)--++(.4,0);
		\draw[->] (-.5,.5)--++(.4,0);
		\draw[dotted] (.1,0)--++(.4,0);
		\draw[->] (.1,.5)--++(.4,0);
		\node at (0,-.25) {\LARGE$#2$};
		\node at (0,.25) {\LARGE$#3$};
		\node at (0,.75) {\LARGE$#4$};
	\end{tikzpicture}}}
\newcommand{\ioWio}[5]{
\scalebox{#1}{\begin{tikzpicture} [scale=1.5, ultra thick, baseline=#5]
		\draw[dotted] (-.5,0)--++(.4,0);
		\draw[->] (-.5,.5)--++(.4,0);
		\draw[->] (.1,0)--++(.4,0);
		\draw[dotted] (.1,.5)--++(.4,0);
		\node at (0,-.25) {\LARGE$#2$};
		\node at (0,.25) {\LARGE$#3$};
		\node at (0,.75) {\LARGE$#4$};
	\end{tikzpicture}}}
\newcommand{\ioWii}[5]{
\scalebox{#1}{\begin{tikzpicture} [scale=1.5, ultra thick, baseline=#5]
		\draw[dotted] (-.5,0)--++(.4,0);
		\draw[->] (-.5,.5)--++(.4,0);
		\draw[->] (.1,0)--++(.4,0);
		\draw[->] (.1,.5)--++(.4,0);
		\node at (0,-.25) {\LARGE$#2$};
		\node at (0,.25) {\LARGE$#3$};
		\node at (0,.75) {\LARGE$#4$};
	\end{tikzpicture}}}

\newcommand{\iiWoo}[5]{
\scalebox{#1}{\begin{tikzpicture} [scale=1.5, ultra thick, baseline=#5]
		\draw[->] (-.5,0)--++(.4,0);
		\draw[->] (-.5,.5)--++(.4,0);
		\draw[dotted] (.1,0)--++(.4,0);
		\draw[dotted] (.1,.5)--++(.4,0);
		\node at (0,-.25) {\LARGE$#2$};
		\node at (0,.25) {\LARGE$#3$};
		\node at (0,.75) {\LARGE$#4$};
	\end{tikzpicture}}}
\newcommand{\iiWoi}[5]{
\scalebox{#1}{\begin{tikzpicture} [scale=1.5, ultra thick, baseline=#5]
		\draw[->] (-.5,0)--++(.4,0);
		\draw[->] (-.5,.5)--++(.4,0);
		\draw[dotted] (.1,0)--++(.4,0);
		\draw[->] (.1,.5)--++(.4,0);
		\node at (0,-.25) {\LARGE$#2$};
		\node at (0,.25) {\LARGE$#3$};
		\node at (0,.75) {\LARGE$#4$};
	\end{tikzpicture}}}
\newcommand{\iiWio}[5]{
\scalebox{#1}{\begin{tikzpicture} [scale=1.5, ultra thick, baseline=#5]
		\draw[->] (-.5,0)--++(.4,0);
		\draw[->] (-.5,.5)--++(.4,0);
		\draw[->] (.1,0)--++(.4,0);
		\draw[dotted] (.1,.5)--++(.4,0);
		\node at (0,-.25) {\LARGE$#2$};
		\node at (0,.25) {\LARGE$#3$};
		\node at (0,.75) {\LARGE$#4$};
	\end{tikzpicture}}}
\newcommand{\iiWii}[5]{
\scalebox{#1}{\begin{tikzpicture} [scale=1.5, ultra thick, baseline=#5]
		\draw[->] (-.5,0)--++(.4,0);
		\draw[->] (-.5,.5)--++(.4,0);
		\draw[->] (.1,0)--++(.4,0);
		\draw[->] (.1,.5)--++(.4,0);
		\node at (0,-.25) {\LARGE$#2$};
		\node at (0,.25) {\LARGE$#3$};
		\node at (0,.75) {\LARGE$#4$};
	\end{tikzpicture}}}

%% file: single_vertex_full.tex
\def\epsx{.12}
\newcommand{\vertexoo}[1]{\scalebox{#1}{\begin{tikzpicture}
[scale=1, very thick]
\node (i1) at (0,-1) {$g$};
\node (j1) at (-1,0) {$0$};
\node (i2) at (0,1) {$g$};
\node (j2) at (1,0) {$0$};
\draw[densely dotted] (j1) -- (j2);
\foreach \shi in {(0,0), (\epsx,0), (-\epsx,0)}
{\begin{scope}[shift=\shi]
\node (shi1) at (0,-1) {\phantom{$g$}};
\node (shi2) at (0,1) {\phantom{$g$}};
\draw[->] (shi1) -- (shi2);
\draw[->] (shi1) --++ (0,.5);
\end{scope}}
\end{tikzpicture}}}
\newcommand{\vertexol}[1]{\scalebox{#1}{\begin{tikzpicture}
[scale=1, very thick]
\node (i1) at (0,-1) {$g$};
\node (j1) at (-1,0) {$0$};
\node (i2) at (0,1) {$g-1$};
\node (j2) at (1,0) {$1$};
\foreach \shi in {(0,0), (-\epsx,0)}
{\begin{scope}[shift=\shi]
\node (shi1) at (0,-1) {\phantom{$g$}};
\node (shi2) at (0,1) {\phantom{$g$}};
\draw[->] (shi1) -- (shi2);
\draw[->] (shi1) --++ (0,.5);
\end{scope}}
\foreach \shi in {(\epsx,0)}
{\begin{scope}[shift=\shi]
\node (shi1) at (0,-1) {\phantom{$g$}};
\node (shi2) at (0,1) {\phantom{$g$}};
\draw[->] (shi1) -- (0,0) -- (j2);
\draw[->] (shi1) --++ (0,.5);
\end{scope}}
\draw[densely dotted] (j1) -- (j2);
\end{tikzpicture}}}

\newcommand{\vertexll}[1]{\scalebox{#1}{\begin{tikzpicture}
[scale=1, very thick]
\node (i1) at (0,-1) {$g$};
\node (i1shm1) at (-\epsx,-1) {\phantom{$g$}};
\node (i1sh1) at (\epsx,-1) {\phantom{$g$}};
\node (j1) at (-1,0) {$1$};
\node (i2) at (0,1) {$g$};
\node (i2shm1) at (-\epsx,1) {\phantom{$g$}};
\node (i2sh1) at (\epsx,1) {\phantom{$g$}};
\node (j2) at (1,0) {$1$};
\draw[densely dotted] (j1) -- (j2);
\draw[densely dotted] (i1) -- (i2);
\draw[->] (j1) -- (-\epsx*1.5,0)--++(.5*\epsx,.5*\epsx) -- (i2shm1);
\draw[->] (i1shm1) -- (-\epsx,-\epsx) -- (0,\epsx) -- (i2);
\draw[->] (i1) -- (0,-\epsx) -- (\epsx,\epsx) -- (i2sh1);
\draw[->] (i1sh1) -- (\epsx,-.5*\epsx)--++(.5*\epsx,.5*\epsx) -- (j2);
\draw[->] (j1) --++ (.5,0);
\draw[->] (i1) --++ (0,.5);
\draw[->] (i1sh1) --++ (0,.5);
\draw[->] (i1shm1) --++ (0,.5);
\end{tikzpicture}}}
\newcommand{\vertexlo}[1]{\scalebox{#1}{\begin{tikzpicture}
[scale=1, very thick]
\node (i1) at (0,-1) {$g$};
\node (j1) at (-1,0) {$1$};
\node (i2) at (0,1) {$g+1$};
\node (i2shm2) at (-3/2*\epsx,1) {\phantom{$g$}};
\node (j2) at (1,0) {$0$};
\draw[densely dotted] (j1) -- (j2);
\draw[->] (j1) -- (-3/2*\epsx,0) -- (i2shm2);
\foreach \shi in {(1/2*\epsx,0), (3/2*\epsx,0), (-1/2*\epsx,0)}
{\begin{scope}[shift=\shi]
\node (shi1) at (0,-1) {\phantom{$g$}};
\node (shi2) at (0,1) {\phantom{$g$}};
\draw[->] (shi1) -- (shi2);
\draw[->] (shi1) --++ (0,.5);
\end{scope}}
\draw[->] (j1) --++ (.5,0);
\end{tikzpicture}}}

%% file: table_colors.tex
\newcommand{\coli}{\cellcolor{red!15}}
\newcommand{\colii}{\cellcolor{red!45}}
\newcommand{\colx}{\cellcolor{gray!15}}
\newcommand{\colxx}{\cellcolor{gray!75}}